\pdfoutput=1

\documentclass{article}
\usepackage{savesym}
\usepackage{amsthm,amsmath}
\usepackage{amssymb,latexsym,graphicx}
\usepackage{amscd}
 \usepackage[all,cmtip]{xy}
\usepackage{tikz-cd}
\usepackage{accents}
\usepackage{cite}
\usepackage{mathtools}
\usepackage{stmaryrd} % provides \llbracket (formal power series)

\usepackage{calligra}
\DeclareMathAlphabet{\mathcalligra}{T1}{calligra}{m}{n}
\DeclareMathAlphabet{\mathpzc}{OT1}{pzc}{m}{it}
\usepackage{hycolor}
\usepackage{xcolor}
\usepackage[
                       colorlinks=true,
                       linkcolor=black, %red,
                       citecolor=black, %green,
                       urlcolor=blue,
                     ]{hyperref}
\usepackage{soul} %provides strikethrough-command st{.. text ..}

\usepackage{makeidx}
\usepackage[intoc]{nomentbl}
\makenomenclature

\usepackage{stackengine} %%% provides Nequation - named and numbered equation
\usepackage{booktabs} % tables

\newtheorem{theoremABC}{Theorem}

\newtheorem{theorem}{Theorem}[section]

\newtheorem{corollary}[theorem]{Corollary}

\newtheorem{lemma}[theorem]{Lemma}
\newtheorem{proposition}[theorem]{Proposition}

\theoremstyle{definition}
\newtheorem{definition}[theorem]{Definition}
\newtheorem{hypothesis}[theorem]{Hypothesis}

\newtheorem{convention}[theorem]{Convention}
\newtheorem{remark}[theorem]{Remark}

\theoremstyle{remark}

\newcommand{\N}{{\mathbb{N}}}

\newcommand{\R}{{\mathbb{R}}}
\renewcommand{\SS}{{\mathbb{S}}}
\newcommand{\T}{{\mathbb{T}}}

\newcommand{\Aa}{{\mathcal{A}}}   % connections
\newcommand{\Dd}{{\mathcal{D}}}

\newcommand{\Ff}{{\mathcal{F}}}
\newcommand{\Hh}{{\mathcal{H}}}

\newcommand{\Ll}{{\mathcal{L}}}   % Lagrangian planes
\newcommand{\Mm}{{\mathcal{M}}}   % moduli space

\newcommand{\Oo}{{\mathcal{O}}}

\newcommand{\Qq}{{\mathcal{Q}}}

\newcommand{\Tt}{{\mathcal{T}}}

\newcommand{\Ww}{{\mathcal{W}}}
\newcommand{\Xx}{{\mathcal{X}}}

\newcommand{\Zz}{{\mathcal{Z}}}
\newcommand{\coker}{{\rm coker\, }}  % cokernel
\newcommand{\im}{{\rm im\, }}             % image
\newcommand{\osc}{{\rm osc }}        % oscilation
\newcommand{\id}{{\rm id}}                % identity
\newcommand{\Id}{{\rm Id}}
\newcommand{\dist}{{\rm dist}}            % distance
\newcommand{\INDEX}{\mathop{\mathrm{index}}}     % (Fredholm) index
\newcommand{\IND}{{\rm ind}}                  % (Morse) index
\newcommand{\cgraph}[1]{\Gamma_{\kern-.5ex{}#1}}     % contact graph map
\newcommand{\Fix}{{\rm Fix}}            % Fixed points
\newcommand{\Crit}{{\rm Crit}}        % Critical points
\newcommand{\HM}{{\rm HM}}             % Morse homology
\newcommand{\Arnold}{{Arnol$'$d}}           % V.I. Arnol'd
\newcommand{\norm}{{\rm norm}}

\newcommand{\eps}{{\varepsilon}}

\newcommand{\G}{{\rm G}}

\newcommand{\qfrak}{{\mathfrak q}}

\newcommand{\FF}{{\mathfrak F}}
\newcommand{\ff}{{\mathfrak f}}

\newcommand{\inner}[2]{\langle #1, #2\rangle}   % inner product < , >
\newcommand{\INNER}[2]{\left\langle #1, #2\right\rangle}

\def\NABLA#1{{\mathop{\nabla\kern-.5ex\lower1ex\hbox{$#1$}}}}
\def\Nabla#1{\nabla\kern-.5ex{}_{#1}}
\def\Tabla#1{\Tilde\nabla\kern-.5ex{}_{#1}}
\def\Babla#1{\widebar\nabla\kern-.5ex{}_{#1}}
\def\abs#1{\mathopen|#1\mathclose|}   
\def\Abs#1{\left|#1\right|}            
\def\norm#1{\mathopen\|#1\mathclose\|}
\def\Norm#1{\left\|#1\right\|}

\renewcommand{\Tilde}{\widetilde}

\newcommand{\p}{{\partial}}

\newcommand{\INTO}{\hookrightarrow}              % embedding
\renewcommand{\1}{{{\mathchoice {\rm 1\mskip-4mu l} {\rm 1\mskip-4mu l}
{\rm 1\mskip-4.5mu l} {\rm 1\mskip-5mu l}}}}
\newlength\eqshift
\renewcommand\theequation{\thesection.\arabic{equation}}
\let\savetheequation\theequation

\makeatletter
\renewcommand*\env@matrix[1][\arraystretch]{%
  \edef\arraystretch{#1}%
  \hskip -\arraycolsep
  \let\@ifnextchar\new@ifnextchar
  \array{*\c@MaxMatrixCols c}}
\makeatother

\makeatletter
\let\save@mathaccent\mathaccent
\newcommand*\if@single[3]{%
  \setbox0\hbox{${\mathaccent"0362{#1}}^H$}%
  \setbox2\hbox{${\mathaccent"0362{\kern0pt#1}}^H$}%
  \ifdim\ht0=\ht2 #3\else #2\fi
  }
\newcommand*\rel@kern[1]{\kern#1\dimexpr\macc@kerna}
\newcommand*\widebar[1]{\@ifnextchar^{{\wide@bar{#1}{0}}}{\wide@bar{#1}{1}}}
\newcommand*\wide@bar[2]{\if@single{#1}{\wide@bar@{#1}{#2}{1}}{\wide@bar@{#1}{#2}{2}}}
\newcommand*\wide@bar@[3]{%
  \begingroup
  \def\mathaccent##1##2{%
    \let\mathaccent\save@mathaccent
    \if#32 \let\macc@nucleus\first@char \fi
    \setbox\z@\hbox{$\macc@style{\macc@nucleus}_{}$}%
    \setbox\tw@\hbox{$\macc@style{\macc@nucleus}{}_{}$}%
    \dimen@\wd\tw@
    \advance\dimen@-\wd\z@
    \divide\dimen@ 3
    \@tempdima\wd\tw@
    \advance\@tempdima-\scriptspace
    \divide\@tempdima 10
    \advance\dimen@-\@tempdima
    \ifdim\dimen@>\z@ \dimen@0pt\fi
    \rel@kern{0.6}\kern-\dimen@
    \if#31
      \overline{\rel@kern{-0.6}\kern\dimen@\macc@nucleus\rel@kern{0.4}\kern\dimen@}%
      \advance\dimen@0.4\dimexpr\macc@kerna
      \let\final@kern#2%
      \ifdim\dimen@<\z@ \let\final@kern1\fi
      \if\final@kern1 \kern-\dimen@\fi
    \else
      \overline{\rel@kern{-0.6}\kern\dimen@#1}%
    \fi
  }%
  \macc@depth\@ne
  \let\math@bgroup\@empty \let\math@egroup\macc@set@skewchar
  \mathsurround\z@ \frozen@everymath{\mathgroup\macc@group\relax}%
  \macc@set@skewchar\relax
  \let\mathaccentV\macc@nested@a
  \if#31
    \macc@nested@a\relax111{#1}%
  \else
    \def\gobble@till@marker##1\endmarker{}%
    \futurelet\first@char\gobble@till@marker#1\endmarker
    \ifcat\noexpand\first@char A\else
      \def\first@char{}%
    \fi
    \macc@nested@a\relax111{\first@char}%
  \fi
  \endgroup
}
\makeatother

\newcommand{\nor}{\mathrm{nor}}     % 
\newcommand{\Exp}{\mathrm{Exp}}     % 
\long\def\symbolfootnote[#1]#2{\begingroup%
\def\thefootnote{\fnsymbol{footnote}}\footnote[#1]{#2}\endgroup}
\begin{document}
\sloppy
%\author{Urs Frauenfelder \qquad Joa Weber
%                 Instituto de Matem\'{a}tica, Estat\'{\i}stica
%                 e Computa\c{c}\~{a}o Scient\'{\i}fica \\
%                 Universidade Estadual de Campinas \\
%                 Rua S\'{e}rgio Buarque de Holanda~651,
%                 Cidade Universit\'{a}ria "Zeferino Vaz",
%                CEP~13083-859, Campinas-SP, Brasil.
           %     \\
%                 joa@math.sunysb.edu.
%              \\
%              Tel.: +55-19-3521xxxx\\
%              Fax: +55-19-3521xxxx\\
%            }

\author{\quad Urs Frauenfelder \quad \qquad\qquad
             Joa Weber\footnote{
  Email: urs.frauenfelder@math.uni-augsburg.de
  \hfill
%  joa-weber@protonmail.com
%  joa@ime.unicamp.br
    joa@math.uni-bielefeld.de
%                 joa@math.sunysb.edu.
  }
%\footnote{
%        {\bf Financial support:}
%        Funda\c{c}\~{a}o de Amparo
%        \`{a} Pesquisa do Estado de S\~{a}o Paulo
%        (FAPESP), processo $\mathrm{n}^{\rm o}$ 2017/19725-6,
        %and CNPq, Conselho Nacional de Desenvolvimento Cient\'{\i}fico
        %e Tecnol\'ogico - Brasil.
        %Bolsista do CNPq - Brasil.
        %%%\hfill
        %  se publicado individualmente:
        % "O presente trabalho foi realizado com apoio do CNPq, Conselho Nacional de Desenvol          vimento Científico  e Tecnológico - Brasil".
        %  se publicado em co-autoria:
        % "Bolsista do CNPq - Brasil".
        %
        %%%\newline
%        {\bf Address:}
%        Instituto de Matem\'{a}tica, Estat\'{\i}stica
%        e Computa\c{c}\~{a}o Scient\'{\i}fica,
%        Universidade Estadual de Campinas,
%        Rua S\'{e}rgio Buarque de Holanda~651,
        %SP~13083-859 ,
%        Campinas, SP, Brasil.
        % MSC 37Dxx 58E05
        %
%        \hfill
%        joa@ime.unicamp.br
%  }
    \\
    Universit\"at Augsburg \qquad\qquad%\quad
    UNICAMP
}

\title{Lagrange multipliers and adiabatic limits I}

%\subtitle{-- Monograph --}  %%% Springer book style only
\date{\today}

%\begin{titlepage}
\maketitle %(to set the title page and copyright page; see note)
                 %\include files (e.g., preface, introduction)
%\thispagestyle{empty}
%\newpage
%
%{\color{red}
%  \subsection*{To do}
%  \begin{itemize}
%  \item
%    ...
%  \end{itemize}
%}

%\end{titlepage}

%%%%%%%%%%%%%%%%%%%%%%%%%%%%%%%%%%
%%%%%%%%%% FRONTMATTER %%%%%%%%%%%%%
%%%%%%%%%%%%%%%%%%%%%%%%%%%%%%%%%%
%\frontmatter
%\include{dedic}
%\include{foreword}
%\include{preface}
%\include{acknow}
%
%\tableofcontents
%
%\include{acronym}

%\frontmatter %• title page and copyright page information
%\include{0-dedic}
% dedication by hand:
%     \clearpage\thispagestyle{empty}
%      \par\vspace*{.35\textheight}{\centering Dedicated to ...\par}\clearpage
%
%\include{0-preface}

% ACKNOWLEDGEMENTS %
% * various anonymous referees?
% * DISCLAIMER: - references based on authors knowledge
%                        - more people contributed etc etc

%%%%%%%%%%%%%%%%%%%%%%%%%%%%%%%%%%
%%%%%%% main matter %%%%%%%%%%%%%%%%%%
%%%%%%%%%%%%%%%%%%%%%%%%%%%%%%%%%%
%\mainmatter
%\include{part}
%\include{chapter}
%\include{appendix}

%\mainmatter %\include files (e.g., main chapters, appendices)

% introduction
%\cleardoublepage
%\phantomsection
%\include{1_sc-smoothness}      % INTRODUCTION

%%%%%%%%%%%%%%%%%%%%%%%%%%%%%%%%%%%
%%%%%%% Abstract %%%%%%%%%%%%%%%%%%%%%
%%%%%%%%%%%%%%%%%%%%%%%%%%%%%%%%%%%
\begin{abstract}
Critical points of a function subject to a constraint
can be either detected by restricting the function to the constraint
or by looking for critical points of the Lagrange multiplier
functional.
Although the critical points of the two functionals, namely the
restriction and the Lagrange multiplier functional
are in natural one-to-one correspondence this does not need to
be true for their gradient flow lines.
We consider a singular deformation of the metric and show by an
adiabatic limit argument that close to the singularity
we have a one-to-one correspondence between
gradient flow lines connecting critical points of Morse index
difference one.
We present a general overview of the adiabatic limit technique
in the article~\cite{Frauenfelder:2022h}.

The proof of the correspondence is carried out in two parts.
The current part I deals with linear methods leading to a singular
version of the implicit function theorem.
We also discuss possible infinite dimensional generalizations
in Rabinowitz-Floer homology.
In part II~\cite{Frauenfelder:2022g} we apply non-linear methods and
prove, in particular, a compactness result and uniform exponential
decay independent of the deformation parameter.
\end{abstract}

%\newpage
\tableofcontents

\newpage
\boldmath
%%%%%%%%%%%%%%%%%%%%%%%%%%%%%%%%%%%
%%%%%%%%%%%%%%%%%%%%%%%%%%%%%%%%%%%
%%%%%%% Section:  %%%%%%%%%%%%%%%%%%%%%
%%%%%%%%%%%%%%%%%%%%%%%%%%%%%%%%%%%
%%%%%%%%%%%%%%%%%%%%%%%%%%%%%%%%%%%
\section{Introduction}
\label{sec:introduction}
\unboldmath

%\medskip
In 1806 it was the observation of Joseph Louis de
Lagrange~\cite{Lagrange:1806a} that critical points
of a function $F(x)$ subject to a constraint $H(x)=0$
correspond to critical points of the unconstrained function
$F_H(x,\tau)=F(x)+\tau H(x)$ which also depends
on a Lagrange multiplier $\tau$.
%\\
More precisely, suppose that $M$ is a finite dimensional manifold, not
necessarily symplectic, but equipped with a Riemannian metric~$G$.
Let $F$ and $H$ be smooth functions on $M$ such that
zero is a regular value of~$H$. Thus $\Sigma:=H^{-1}(0)$ is a
smooth level hypersurface. Under these assumptions
there is a bijection between the critical point sets of the following two
functions, namely the \textbf{Lagrange multiplier functional}
\[
   F_H\colon M\times\R\to\R,\quad (u,\tau)\mapsto F(u)+\tau H(u),
\]
and the \textbf{restriction function} of $F$ to the constraint
$\Sigma=H^{-1}(0)$, in symbols
\[
   f\colon \Sigma\to\R,\quad q\mapsto F(q) .
\]
The natural bijection is by forgetting the first factor, in symbols
\begin{equation}\label{eq:Crit=Crit}
   \Crit\, F_H\to \Crit f,\quad
   (x,\tau)\mapsto x .
\end{equation}
The Morse indices differ by $1$, namely
\[
   \IND_{F_H}(x,\tau)=\IND_f(x)+1.
\]
In particular, the difference of the Morse indices at two critical
points is independent of the choice of function $F_H$ or $f$.

Under local properness conditions it was shown in~\cite{Frauenfelder:2006a}
that the Morse homologies of the two functions coincide up to an
index shift by $1$, namely $\HM_*(F_H)\simeq\HM_{*+1}(f)$.
Therefore the Lagrange multiplier function computes the homology
of $\Sigma$ up to a grading shift by $1$.
The proof of this fact in~\cite{Frauenfelder:2006a}
uses normal deformations of the function $F$
and is hard to generalize to infinite dimensions.
Therefore we focus in the present paper on a completely different approach to
this homology equivalence which, as well, is much stronger since it
gives an isomorphism on chain level and not just on homology level.
This approach is based on the adiabatic limit technique
developed by Dostoglou and Salamon~\cite{dostoglou:1994a}
in their proof of a special case of the Atiyah-Floer conjecture.
The technique was successfully used and developed further
in the context of symplectic vortex equations~\cite{gaio:1999a,Gaio:2005a}
and the heat flow~\cite{weber:1999a,salamon:2006a}.

\medskip
In the context of Lagrange multipliers this adiabatic limit technique
works as follows. Pick a parameter $\eps\in(0,1]$. Then the gradient
flow equation of $F_H$ with respect to the product metric
$\G\oplus \eps^2$ on $M\times \R$ is given by
\begin{equation}\label{eq:DGF-eps}
   \p_s(u,\tau)+\nabla^\eps F_H(u,\tau)
   =
   \begin{pmatrix}
      \p_su+\Babla{} F|_u+\tau\Babla{} H|_u\\
      \tau^\prime+\eps^{-2}H\circ u
   \end{pmatrix}
   =
   \begin{pmatrix}
      0\\0
   \end{pmatrix} .
\end{equation}
for smooth maps $(u,\tau)\colon \R\to  M\times\R$
and where $\nabla^\eps$ is the gradient in the Riemannian manifold
$(M\times\R,G\oplus\eps^2)$
and $\Babla{}$ is the gradient in $(M,G)$.

Letting $\eps$ formally go to zero one obtains the pair of equations
\begin{equation}\label{eq:DGF-0}
   \begin{pmatrix}
      \p_su+\Babla{} F|_u+\tau\Babla{} H|_u\\
      H\circ u
   \end{pmatrix}
   =
   \begin{pmatrix}
      0\\0
   \end{pmatrix} .
\end{equation}
Equation two tells that $u$ actually takes values in $\Sigma=H^{-1}(0)$.
In this case equation one is the downward gradient equation on $\Sigma$
of the restriction $f$ of $F$ and with respect to the Riemannian
metric $g$ given by restricting~$G$ (Lemma~\ref{def:DGF-ff}).

\medskip
The main result of part I is the following theorem.
Suppose that $x^\mp\in\Crit f$ are critical points of Morse index
difference one. Then for each $\eps\in(0,\eps_0]$ we construct a time
shift invariant map $\Tt^\eps\colon\Mm^0_{x^-,x^+}\to\Mm^\eps_{x^-,x^+}$
between moduli spaces of gradient flow trajectories
$q\colon\R\to\Sigma$ and $(u,\tau)\colon\R\to M\times\R$
which at $\mp\infty$ converge to the critical points $x^\mp$,
respectively to $(x^\mp,\tau^\mp)\in\Crit F_H$.

\begin{theoremABC}\label{thm:main}
Assume $(f,g)$ is Morse-Smale.
Then there is a constant $\eps_0\in(0,1]$, such that
for every $\eps\in(0,\eps_0]$ and every pair $x^\mp\in\Crit f$ of
index difference one, the map
$\Tt^\eps\colon\Mm^0_{x^-,x^+}\to\Mm^\eps_{x^-,x^+}$ is injective.
\end{theoremABC}
 
\begin{remark}\label{rem:Schecter-Xu}
Under the assumption that the ambient manifold $M$ is compact
Theorem~\ref{thm:main} was first proved by Stephen Schecter and
Guangbo Xu~\cite{Schecter:2014a}.

\end{remark}

To prove Theorem~\ref{thm:main} we associate to $q\in\Mm^0_{x^-,x^+}$
a suitable pair $(q,\tau)$ which almost solves the
$\eps$-equation~(\ref{eq:DGF-eps}). Then we use the Newton
method to find a unique true solution nearby.
This is the content of part I (this article).

In part II~\cite{Frauenfelder:2022g} we shall prove surjectivity by
contradiction. If $\Tt^\eps$ is not surjective for $\eps>0$ small, 
there is a sequence of positive reals $\eps_i\to 0$ and a sequence
$(u_i,\tau_i)\in\Mm^{\eps_i}_{x^-,x^+}$ not in the image of $\Tt^{\eps_i}$.
We show that the maps $u_i$ take values near $\Sigma$ and
that they naturally project to maps $\qfrak_i\colon\R\to\Sigma$
which are almost solutions of the base equation~(\ref{eq:DGF-0}).
We identify true solutions $q_i\colon\R\to\Sigma$ nearby
and show that after suitable time shift $\sigma_i\in\R$ we have
$(u_i,\tau_i)=\Tt^{\eps_i}(q_i(\sigma_i+\cdot))$.
This contradiction proves surjectivity.

\begin{convention}[Notation]\label{con:notation}
\mbox{}
\\
a) Tangent and normal bundle of $\Sigma$ in $M$ are denoted by
$
   T\Sigma\oplus N\Sigma=T_\Sigma M
$.
Tangent vectors to $M$ based at $\Sigma$ decompose 
$X=\xi+\nu=\tan X+\nor\, X$.
The dimension of $\Sigma$ is $n$, hence $n+1=\dim M$.
\\
b) Arguments of maps $H(u)$ are likewise denoted by $H|_u$.
\\
c) For $u\colon\R\to M$, $q\colon\R\to\Sigma$, $\tau\colon\R\to\R$
we often de-parenthesify and write
\[
   u_s:=u(s),\quad q_s:=q(s),\quad \tau_s:=\tau(s),
\]
and
\[
   \p_su:=\tfrac{d}{ds}u,\quad \p_sq:=\tfrac{d}{ds}q
   ,\qquad\text{but $\tau^\prime:=\tfrac{d}{ds}\tau$} .
\]
d) The symbol $\abs{\cdot}$, applied to real numbers means absolute
value, applied to vectors it means vector norm, for example
$\Abs{\p_s u}:=\Abs{\p_s u}_G$ on $(M,G)$
and $\Abs{\p_s q}:=\Abs{\p_s q}_g$~on~$(\Sigma,g)$.
Throughout $\norm{\cdot}$ denotes $L^2$-norm.
\\
e) Inner products are denoted by $\INNER{\cdot}{\cdot}$.
Depending on context $\INNER{\cdot}{\cdot}$ abbreviates
$\INNER{\cdot}{\cdot}_g$ on $T\Sigma$,
$\INNER{\cdot}{\cdot}_G$ on $TM$,
$\INNER{\cdot}{\cdot}_2$ on an $L^2$ space, or other inner products.
\end{convention}

\smallskip\noindent
{\bf Acknowledgements.}
UF acknowledges support by DFG grant FR 2637/2-2.

%\newpage
\boldmath
%%%%%%%%%%%%%%%%%%%%%%%%%%%%%%%%%%%
%%%%%%%%%%%%%%%%%%%%%%%%%%%%%%%%%%%
%%%%%%% Subsection:  %%%%%%%%%%%%%%%%%%
%%%%%%%%%%%%%%%%%%%%%%%%%%%%%%%%%%%
\subsection{Outline}
%\label{sec:}
\unboldmath

Let $(M,G)$ be a Riemannian manifold.
Let $F$ and~$H$ be smooth functions on~$M$.
The \textbf{Lagrange multiplier function} is defined by
\begin{equation*}%\label{eq:findim-F_H}
   F_H\colon M\times\R\to\R,\quad (x,\tau)\mapsto
   F(x)+\tau H(x) .
\end{equation*}

\begin{hypothesis}\label{hyp:uniform-Lag-bound-intro}
(i) Zero is a regular value of $H$.
(ii) {\bf Local properness:}
There exists a constant $\kappa>0$
such that $\Sigma_\kappa:=H^{-1}[-\kappa,\kappa]\subset M$ is compact.
\\
(iii) The Riemannian  metric $G$ on $M$ is geodesically complete.
\end{hypothesis}

By (i) and (ii) the zero level $\Sigma:=H^{-1}(0)$
is a smooth compact hypersurface~in~$M$,
we assume without boundary.
By~(iii) closed and bounded is equivalent to compact
(Theorem of Hopf-Rinow; see e.g.~\cite[Ch.\,5 Thm.\,21]{oneill:1983a}).
Local properness excludes that $H$ tends to zero at infinity.

\medskip%\noindent
\textbf{Section~\ref{sec:functionals}
``Lagrange multiplier function and restriction''.}
The map
$$
   \iota\colon \Sigma=H^{-1}(0)\INTO M
   ,\quad
   q\mapsto q {\color{gray}\;=\iota(q)},
$$
given by inclusion
induces on $\Sigma$ the Riemannian metric $g:=\iota^*G$
and the function $f:=\iota^* F$, both given by restriction.
Let $\Babla{}$ be the Levi-Civita connection of $(M,G)$
and $\nabla$ the one of $(\Sigma,g)$. In
Section~\ref{sec:hypersurface-geometry} we briefly recall some
Riemannian hypersurface geometry of $(\Sigma,g,\nabla)$ in
$(M,G,\Babla{})$. Since $0$ is a regular value of $H$, along
$\Sigma=H^{-1}(0)$ there is an orthogonal decomposition
\[
   T_\Sigma M
   =T\Sigma\stackrel{\perp}{\oplus}\R\Babla{}H
   ,\qquad X=\xi+\nu.
\]
Let $\tan$ and $\nor$ be the corresponding orthogonal projections.
The function
\begin{equation*}%\label{eq:chi}
   \chi:=-\inner{\Babla{} F}{V}
   ,\;\;
   V:=\tfrac{\Babla{} H}{\abs{\Babla{} H}^2}
   ,\;\;
   \text{along $M_{\rm reg}:=\{p\in M\mid dH(p)\not= 0\}\supset\Sigma$}
\end{equation*}
has the fundamental significance that at each point of $\Sigma$
the value of $\chi$ is the unique real that makes the linear combination
\[
   \Babla{}F(q)+\chi(q)\Babla{}H(q)\in T_q\Sigma
   ,\quad  q\in\Sigma
\]
of the two $T_qM$-valued vectors $\Babla{}F|_q$ and $\Babla{}H|_q$
be tangent to $\Sigma$.
The function $\chi$ plays a crucial role throughout this article,
as hinted at by the gradient identities
\begin{equation*}%\label{eq:findim-gradf}
   \tan \Babla{} F=\Nabla{} f
   ,\qquad
   \nor\, \Babla{} F=-\chi \Babla{} H
   ,\qquad
   \Nabla{}f=\Babla{} F+\chi\Babla{}H ,
\end{equation*}
along $\Sigma$. The last identity translates the gradient flow of $f$ on the
\textbf{base} $\Sigma$ to the terminology of the \textbf{ambience} $M$.
The local flow $\{\varphi_r\colon\Sigma\to M_{\rm reg}\}$
generated by $V$ near $\Sigma$ transforms $H$ to the normal form
$H(\varphi_r q)=r$ in~(\ref{eq:normal-form-u}).
Further important roles play the graph map of $\chi$,
called the \textbf{canonical embedding}
\begin{equation*}%\label{eq:can-emb}
   i:\Sigma\to M\times\R,\quad
   q\mapsto \left(q,\chi(q)\right)
   {\color{gray}\;=\left(\iota(q),\chi(\iota(q))\right)},
\end{equation*}
and the derivative $I_q\xi:=di(q)\xi=\left(\xi,d\chi(q)\xi\right)$
for $q\in\Sigma$.
We show that the critical point sets $i\left(\Crit f\right)=\Crit F_H$
are in bijection through the canonical embedding $i$,
the inverse of the forgetful map~(\ref{eq:Crit=Crit}).
Then we show the Morse index identity
$\IND_{F_H}(x,\tau)=\IND_f(x)+1$ for critical points.

\medskip%\noindent
\textbf{Section~\ref{sec:DGF}
``Downward gradient flows''.}
We introduce the downward gradient flow~(\ref{eq:DGF-0})
on the base $(\Sigma,g)$, whose solutions $q$ are called
\textbf{\boldmath$0$-solutions}.
We introduce the downward gradient flow~(\ref{eq:DGF-eps}) on the product
$(M\times\R,G\oplus\eps^2)$ where the metric is deformed by a
parameter $\eps>0$ and whose
solutions $z=(u,\tau)$ of~(\ref{eq:DGF-eps}) are called
\textbf{\boldmath$\eps$-solutions}.
\\
We define the base energy $E^0(q)$ and the $\eps$-energy
$E^\eps(u,\tau)$ for smooth maps $q\colon\R\to\Sigma$ and
$(u,\tau)\colon\R\to M\times\R$ and show the uniform
\textbf{energy estimates}
$
   E^0(q)
   =
   \Norm{\p_sq}^2
   \le \osc f
$
for base flow trajectories~$q$, but for $\eps$-flow trajectories
\begin{equation*}
   E^\eps(u,\tau)<\infty
   \quad\Rightarrow\quad
   E^\eps(u,\tau)
   =
   \Norm{\p_su}^2+\eps^2\norm{\tau^\prime}^2
   \le \osc f:=\max f-\min f .
\end{equation*}
Two critical points $x^\mp$ of $f\colon\Sigma\to\R$
are called \textbf{asymptotic boundary conditions}
of a smooth map $q\colon\R\to\Sigma$ if
%\begin{equation}\label{eq:limit-ff-intro}
$
   \lim_{s\to\mp\infty} q(s)=x^\mp 
$
%\end{equation}
and of a pair of smooth maps $(u,\tau)\colon\R\to M\times\R$ if
\begin{equation*}%\label{eq:limit-F_H-intro}
   \lim_{s\to\mp\infty} \left(u(s),\tau(s)\right)
   =\left(x^\mp,\chi(x^\mp)\right) .
\end{equation*}
Observe that $\left(x^\mp,\chi(x^\mp)\right)\in\Crit F_H$.
With gradient equations and asymptotic boundary conditions
in place there are the usual \textbf{energy identities}
\begin{equation*}
   E^0(q)=f(x^-)-f(x^+)
   ,\qquad
   E^\eps(u,\tau)
   =
   f(x^-)-f(x^+)
   =: c^* ,
\end{equation*}
for base flow trajectories~$q$, respectively
for $\eps$-flow trajectories $(u,\tau)$.

\medskip%\noindent
%\textbf{Section~\ref{sec:apriori-NEW}
\textbf{``A priori estimates''.}
%The ambient linear estimate in Section~\ref{sec:linear-operators}
%rests on the fact that by part~(ii) of the following theorem along
%$\eps$-trajectories $(u,\tau)$ the gradient
%$\abs{\Babla{}H(u)}\ge c_\kappa$ is bounded
%away from zero, uniformly in $\eps>0$ small. 
The following theorem, proved in
part~II~\cite{Frauenfelder:2022g}, provides uniform a priori bounds
for $\eps$-solutions $(u,\tau)$ and all derivatives.
The theorem is fundamental for all subsequent sections and it is also
rather surprising in view of the factor $\eps^{-2}$ in the deformed
equations~(\ref{eq:DGF-eps}). The theorem assumes only
finite energy of the $\eps$-solutions.

\begin{theorem}[Uniform a priori bounds for finite energy trajectories]
\label{thm:apriori-intro}
Assume Hypothesis~\ref{hyp:uniform-Lag-bound-intro} with
constant~$\kappa$.
Then there are, a compact subset $K\subset M$, and
constants $c_0,c_1,c_2,c_3>0$, with the following significance.
Assume $(u,\tau)\colon\R\to M\times\R$ solves the
$\eps$-equations~(\ref{eq:DGF-eps}) and is of 
finite energy ${E^\eps(u,\tau)<\infty}$.
\begin{itemize}\setlength\itemsep{0ex}
\item[\rm (i)]
If $\eps\in(0,1]$, then the component $u$ takes values in $K$
and there are bounds
\[
   \Abs{\tau(s)}\le c_0
   ,\qquad
   \Abs{\p_s u(s)}+ \Abs{\tau^\prime(s)}\le c_1
   ,\qquad
   \Abs{\Babla{s}\p_s u(s)}+ \Abs{\tau^{\prime\prime}(s)}\le c_2 ,
\]
and $\Abs{\Babla{s}\Babla{s}\p_s u(s)}\le c_3$
at every instant $s\in\R$.
%
%\item[\rm (ii)]
%If $\eps\in(0,\eps_\kappa]$,
%then $\Abs{\Babla{}  H(u(s))}\ge c_\kappa$ at every instant $s\in\R$.
\end{itemize}
\end{theorem}

In part II~\cite{Frauenfelder:2022g} there is actually a
part (ii) of the theorem which generalizes the fact that along the
compact set $\Sigma$ the gradient $\abs{\Babla{}H}$ is bounded away
from zero to, roughly speaking, neighborhoods of $\Sigma$.

\medskip%\noindent
\textbf{Section~\ref{sec:linear-operators}
``Linearized operators''.}
Fix $x^\mp\in\Crit f$. Let $\Qq_{x^-,x^+}$ be the Hilbert manifold
of $W^{1,2}$ paths $q\colon\R\to\Sigma$ with asymptotics $x^\mp$.
The formula
\[
   \Ff^0(q):=\p_s q+\Nabla{} f|_q
   =\p_sq+\Babla{}F(q)+\chi(q)\cdot\Babla{} H(q)
\]
defines a section of the Hilbert bundle $\Ll\to \Qq_{x^-,x^+}$ whose
fiber $\Ll_q$ over $q$ consists of the $T\Sigma$-valued $L^2$
vector fields along $q$. The zero set $\Mm^0_{x^-,x^+}:=(\Ff^0)^{-1}(0)$
is called \textbf{base moduli space}, the zeroes $q$
\textbf{connecting base trajectories}.
Linearize $\Ff^0$ at a zero $q$ to get a linear operator
$W^{1,2}(\R,q^*T\Sigma)\to L^2$ given by
\[
   D^0_q\xi
   =\Nabla{s}\xi-\Nabla{\xi}\Nabla{}f|_q
   =\Babla{s}\xi
   +\Babla{\xi}   \bigl(\Babla{}F|_q+\chi|_q\Babla{} H|_q\bigr).
\]
A pair $(f,g)$ is said \textbf{Morse-Smale} if
$D^0_q\colon W^{1,2}\to L^2$ is surjective for all
$q\in \Mm^0_{x^-,x^+}$ and $x^\mp\in\Crit f$.
The \textbf{trivialization} of $\Ff^0$ at $q\in \Qq_{x^-,x^+}$ is the map
\[
   \Ff^0_q\colon W^{1,2}(\R,q^*T\Sigma)\to L^2(\R,q^*T\Sigma),\quad
   \Ff^0_q(\xi)
   :=\phi(q,\xi)^{-1}\Ff^0(\exp_q\xi)
\]
defined for every $\xi$ of norm smaller than the
injectivity radius of $(\Sigma,g)$, cf.~(\ref{eq:Linfty-est}).
Here $\phi=\phi(q,\xi)\colon T_q\Sigma\to T_{\exp_q(\xi)}\Sigma$
is parallel transport, pointwise for $s\in\R$, along the geodesic
$r\mapsto\exp_{q(s)}(r\xi(s))$ defined in terms of the exponential map
of $(\Sigma,g)$.
The above formula for $D^0_q\xi$ makes sense for general
$q\in \Qq_{x^-,x^+}$, indeed we shall see that $d\Ff^0_q(0)\xi=D^0_q\xi$.
In the formula for the
\textbf{formal \boldmath$L^2$ adjoint}~$(D^0_q)^*$,
see~(\ref{eq:adjoint0-findim-lin}), the term $\Nabla{s}\xi$
changes sign, as is well known, but it is an interesting little detail
that in the ambient formulation a new term $\mathrm{II}$ appears twice
with the same sign, whereas in $D^0_q$ the two signs were opposite.
\\
The operators $D^0_q$ and $(D^0_q)^*$ are bounded,
see~(\ref{eq:findim-BaseEst-prop.D.2}).
If the asymptotics $x^\mp$ are non-degenerate, then both operators
are Fredholm and the Fredholm index is the Morse index difference
of the asymptotics, see Proposition~\ref{prop:findim-0-Fredholm}.

\smallskip
Let $\Zz_{x^-,x^+}$ be the Hilbert manifold of
$W^{1,2}$ paths $z=(u,\tau)\colon\R\to M\times\R$ with asymptotics
$z^\mp=(x^\mp,\chi(x^\mp))$. For $\eps>0$ the formula
\[
   \Ff^\eps(u,\tau)
   \stackrel{(\ref{eq:DGF-def-I})}{:=}
   \begin{pmatrix}
      \p_su+\Babla{} F|_u+\tau\Babla{} H|_u\\
      \tau^\prime+\eps^{-2}H\circ u
   \end{pmatrix}
\]
defines a section of the Hilbert bundle $\Ll\to \Qq_{x^-,x^+}$ whose
fiber $\Ll_{u,\tau}$ over $(u,\tau)$ consists of the $L^2$ vector fields
along $(u,\tau)$. The zero set $\Mm^\eps_{x^-,x^+}:=(\Ff^\eps)^{-1}(0)$
is called \textbf{\boldmath$\eps$-moduli space}, the zeroes $(u,\tau)$
\textbf{connecting \boldmath$\eps$-trajectories}.
Linearize $\Ff^\eps$ at a zero to get a linear map
$W^{1,2}(\R,u^*TM\oplus\R)\to L^2$ of the form
\[
   D^\eps_{u,\tau}
   \begin{pmatrix} X\\\ell\end{pmatrix}
   =\begin{pmatrix}
     \Babla{s} X
     +\Babla{X}\Babla{} F|_u+\tau\Babla{X}\Babla{} H|_u
     +\ell \Babla{} H|_u\\
     \ell^\prime+\eps^{-2}dH|_u X
   \end{pmatrix}.
\]
For general maps $(u,\tau)\in\Zz_{x^-,x^+}$ define
$D^\eps_{u,\tau}$ by the right hand side.
We use the exponential map $\Exp$ of $(M,G)$ to define,
about any map $(u,\tau)\in \Zz_{x^-,x^+}$, a
\textbf{trivialization} $\Ff^\eps_{u,\tau}$, see~(\ref{eq:eps-triv}),
and in~(\ref{eq:deriv-triv-origin}) we show that
$d \Ff^\eps_{u,\tau}(0)=D^\eps_{u,\tau}$.

\smallskip
To get uniform estimates with constants
independent of $\eps>0$ small, we must work
with $\eps$-dependent norms suggested
on $L^2$ by the $\eps$-energy identity
$E^\eps(u,\tau)=\Norm{\p_su}^2+\eps^2\norm{\tau^\prime}^2$
and on $W^{1,2}$ by the ambient linear estimate below.
%~(\ref{eq:amblinest-intro}).
For $\eps>0$ define
\begin{equation*}%\label{eq:findim-0-2-eps-intro}
\begin{split}
   \norm{Z}_{0,2,\eps}^2
   :&=\norm{X}^2+\eps^2\norm{\ell }^2
\\
   \norm{Z}_{1,2,\eps}^2
   :&=\norm{X}^2+\eps^2\norm{\ell }^2
   +\eps^2\norm{\Nabla{s} X}^2+\eps^4\norm{\ell ^\prime}^2
\\
   \norm{Z}_{0,\infty,\eps}
   :&=\norm{X}_\infty+\eps\norm{\ell}_\infty
   \le3\eps^{-1/2}\norm{Z}_{1,2,\eps} 
\end{split}
\end{equation*}
where $Z=(X,\ell)$; cf.~(\ref{eq:findim-0-2-eps}).
The formal adjoint $(D^\eps_{u,\tau})^*$ is defined via
the associated $(0,2,\eps)$ inner product and given by
formula~(\ref{eq:lin-epsadjoint-findim-lin}).
For non-degenerate boundary conditions $x^\mp$ both operators
$D^\eps_{u,\tau}$ and $(D^\eps_{u,\tau})^*$ are
Fredholm~(\ref{eq:F-index-eps}).

\smallskip
This article, part I, focusses on pairs $(u,\tau)=(q,\chi(q))$
with $q\in\Mm^0_{x^-,x^+}$. We abbreviate
(for the formulas see~(\ref{eq:triv-eps-q}) and~(\ref{eq:D^eps_q}))
\[
   \Ff^\eps_q:=\Ff^\eps_{q,\chi(q)}
   ,\qquad
   D^\eps_q:=D^\eps_{q,\chi(q)}
   ,\qquad
   (D^\eps_q)^*:=(D^\eps_{q,\chi(q)})^* .
\]
One of two most important linear estimates in adiabatic limit analysis
is the \textbf{ambient linear estimate}
\begin{equation*}%\label{eq:amblinest-intro}
\begin{aligned}
   \eps^{-1}\norm{dH_qX}
   +\norm{\ell }
   +\norm{\Babla{s} X}
   +\eps\norm{\ell ^\prime}
   &
   \le C\left(\norm{D^\eps_q Z}_{0,2,\eps}+\norm{X}\right)
\end{aligned}
\end{equation*}
for every $Z=(X,\ell)\in W^{1,2}(\R,q^*TM\oplus\R)$,
see~(\ref{eq:amblinest}).

\medskip%\noindent
\textbf{Section~\ref{sec:linear-estimates}
``Linear estimates''.}
The canonical embedding extends via pointwise evaluation to a map
$i\colon\Qq_{x^-,x^+}\to\Zz_{x^-,x^+}$, $q\mapsto(q,\chi(q))$,
between Hilbert manifolds. The linearization
$I_q=di(q)\colon T_q \Qq_{x^-,x^+}\to T_{i(q)}i(\Qq_{x^-,x^+})$
is the map $\xi\mapsto(\xi,d\chi|_q\xi)$.
To prepare Section~\ref{sec:IFT}, where we view
$q\in \Qq_{x^-,x^+}$ as an approximate zero $i(q)$ of $\Ff^\eps$,
see~(\ref{eq:approx-zero-intro}),
Section~\ref{sec:linear-estimates} provides estimates
for the linear operators \emph{along the image of $i$}.
For pairs $(q,\chi(q))$ we have nice control of the $\tau=\chi(q)$
component, because $q$ takes values in $\Sigma$ and $\Sigma$ is compact.

We need to show that if the base flow is Morse-Smale,
then so is the ambient $\eps$-flow for all $\eps>0$ small.
Let $x^\mp\in\Crit f$ be non-degenerate and
$q\in\Mm^0_{x^-,x^+}$ a connecting base trajectory.
Theorem~\ref{thm:KeyEst-thm.3.3} provides
the key estimates for $D^\eps_q$ along the image of $(D^\eps_q)^*$.
So the operator
\begin{equation*}%\label{eq:Reps}
   R^\eps_q
   :=(D^\eps_q)^*\left(D^\eps_q(D^\eps_q)^*\right)^{-1}
   \colon L^2
   \stackrel{(...)^{-1}}{\longrightarrow} W^{2,2}
   \stackrel{(D^\eps_q)^*}{\longrightarrow} W^{1,2}
\end{equation*}
is a right inverse of the linearization $D^\eps_q$
and uniformly bounded in $\eps>0$ small.
Uniformity of the bound is crucial for the Newton iteration to work
in Section~\ref{sec:IFT}, it triggers the need for weighted
Sobolev norms, as mentioned above.

To carry out this program one needs to compare the, by Morse-Smale,
surjective base operator $D^0_q$ with the ambient operator $D^\eps_q$.
To this end we introduce the orthogonal projection
\[
   \Pi_\eps^\perp\colon
   T_{i(q)} \Zz_{x^-,x^+}
   \stackrel{\pi^\perp_\eps}{\longrightarrow}
   T_q\Qq_{x^-,x^+}
   \stackrel{I_q}{\longrightarrow}
   T_{i(q)} i(\Qq_{x^-,x^+})\subset T_{i(q)} \Zz_{x^-,x^+}
\]
onto the image of $I_q=di(q)$ and we show that the linear map
$\pi^\perp_\eps$ is given by
\begin{equation*}\label{eq:pi_eps-ff-intro}
   \pi_\eps (X,\ell)
   =\left(\1+\eps^\alpha\mu^2\, P\right)^{-1}
   \left(\tan X+\eps^\beta \ell \Nabla{}\chi|_q\right)
\end{equation*}
with $\alpha=\beta=2$ and where by definition
$P(q(s))\colon T_{q(s)}\Sigma\to \R\Nabla{}\chi(q(s))$
is the orthogonal projection, at each $s\in\R$, see~(\ref{eq:findim-P}).
In~(\ref{no-eq:4.1.5-findim}) we show that
$\norm{(\1+\eps^\alpha\mu^2\, P)^{-1}}\le 1$.
The linearizations are compared in the form
$D^0_q\pi_\eps-\pi_\eps D^\eps_q$.
The resulting \textbf{key estimates} are of the form
\begin{equation*}%\label{eq:thm:4.4.4}
\begin{aligned}
   \norm{ Z^*}_{1,2,\eps}
   &\le c_1\left(\eps
   \norm{D^\eps_q  Z^*}_{0,2,\eps}
   +\norm{\pi_\eps (D^\eps_q  Z^*)}
   \right)
\\
   \norm{dH|_qX^*}+\eps\norm{\ell^*}
   &\le c_1\eps
   \norm{D^\eps_q  Z^*}_{0,2,\eps} .
\end{aligned}
\end{equation*}
for every pair
$
   Z^*:=(X^*,\ell^*)
   \in\im (D^\eps_q)^*|_{W^{2,2}} \subset W^{1,2}(\R,q^*TM\oplus\R)
$.
In this article the analysis works
for $\alpha\in[1,2]$ and $\beta=2$,
so the orthogonal projection works. This is in sharp
contrast to the PDE adiabatic limit~\cite[(139)]{salamon:2006a}
where the analysis did work for the non-orthogonal projection
where $\alpha=1$ and $\beta=2$.

In~\cite{salamon:2006a} there was no analogue of the second of the
above key estimates. That second estimate plays a crucial role
to prove the uniqueness Theorem~\ref{thm:uniqueness-findim},
see estimate after~(\ref{eq:beta=2}). We arrived at this new twist
in the uniqueness proof by following the philosophy of {\Arnold}
that mathematics reveals itself through simple non-trivial examples,
in our case~\cite{Frauenfelder:2022h}.

\medskip%\noindent
\textbf{Section~\ref{sec:IFT}
``Implicit function theorem I -- Ambience''.}
Suppose $(f,g)$ is Morse-Smale and pick a base connecting trajectory
$q\in\Mm^0_{x^-.x^+}$.
To find an $\eps$-solution near $q$
we utilize Newton's iteration method which requires
a map, say $\Ff^\eps_q$, defined on a Banach space, so it can be
iterated, and whose zeroes are in bijection with the zeroes of
$\Ff^\eps$. Qualitatively, three conditions need to be met. One needs,
firstly, a good starting point $Z_0$ in the sense that its value
$\Ff^\eps_q(Z_0)$ is almost zero, secondly, the derivative
$d \Ff^\eps_q(Z_0)$ must be 'steep enough' in the sense
it must admit a right inverse bounded uniformly in $\eps$ small and,
thirdly, the derivative must not oscillate too wildly near $Z_0$ which is
guaranteed via suitable quadratic estimates.

We are in good shape: The trivialized
ambient section $\Ff^\eps_q$ at the initial point $Z_0:=(0,0)$
of the Newton iteration has a vanishing first component
\begin{equation}\label{eq:approx-zero-intro}
   \Ff^\eps_q(0,0)
   =\Ff^\eps(q,\chi(q))
   :=\begin{pmatrix}
      \p_sq+\Babla{} F(q)+\chi(q)\Babla{} H(q)\\
      (\chi(q))^\prime+\eps^{-2}H(q)
   \end{pmatrix}
   =\begin{pmatrix}
      0\\
      d\chi|_q \p_sq
   \end{pmatrix}
\end{equation}
since $-\p_s q=\Nabla{}f(q)=\Babla{} F(u)+\chi(q)\Babla{} H(u)$.
So $\Norm{\Ff^\eps(q,\chi(q))}_{0,2,\eps}=\eps\Norm{d\chi|_q\p_s q}$
is small for $\eps$ small.
Use the right inverse to define the initial correction term
\[
   \zeta_0
   :=-{D^\eps_q}^*\left(D^\eps_q{D^\eps_q}^*\right)^{-1}\Ff^\eps_q(0)
   =-R^\eps_q \Ff^\eps_q(0) .
\]
Thus
$
   D^\eps_q\zeta_0
   =-\Ff^\eps_q(0)
   =(0,-d\chi|_q\p_s q)
$
and so by key estimate one we get
\begin{equation*}
\begin{split}
   \norm{\zeta_0}_{1,2,\eps}
   &\le
   c_1\left(
   \eps\norm{(0, d\chi|_q\p_s q)}_{0,2,\eps}
   +\norm{(\1+\eps^2\mu^2P)^{-1}
      \left(0+\eps^2(d\chi|_q \p_sq)\Nabla{}\chi\right)}
   \right)
   \\
   &\le {\rm const}\cdot \eps^2.
\end{split}
\end{equation*}
Now define $Z_1:=Z_0+\zeta_0$ and add zero in the form
$-\Ff^\eps_q(0)-D^\eps_u\zeta_0$ to get
\begin{equation*}
   \norm{\Ff^\eps_q(Z_1)}_{0,2,\eps}
   =\norm{\Ff^\eps_q(\zeta_0)-\Ff^\eps_q(0)-D^\eps_u\zeta}_{0,2,\eps}
   \le {\rm const}\cdot \eps^{5/2}
\end{equation*}
where the inequality uses the quadratic estimate~(\ref{eq:quadest-I}).
To the next correction term $\zeta_1:=-R^\eps_q\Ff^\eps_q(Z_1)$ apply
the key estimate observing that $D^\eps_q\zeta_1=-\Ff^\eps_q(Z_1)$.
Iteration provides a Cauchy sequence $Z_\nu$ whose limit $Z^\eps$
corresponds to a zero of $\Ff^\eps_q$ and
$\norm{Z^\eps}_{1,2,\eps}\le {\rm const}\cdot \eps^2$.
For the precise statement see the existence
Theorem~\ref{thm:existence-findim}.
The zero is unique in the sense of the uniqueness
Theorem~\ref{thm:uniqueness-findim}.
These two theorems allow to define the map
$\Tt^\eps$ and the short argument in Lemma~\ref{le:injectivity}
then completes the proof of Theorem~\ref{thm:main}.

\boldmath
%%%%%%%%%%%%%%%%%%%%%%%%%%%%%%%%%%%
%%%%%%%%%%%%%%%%%%%%%%%%%%%%%%%%%%%
%%%%%%% Subsection:  %%%%%%%%%%%%%%%%%%
%%%%%%%%%%%%%%%%%%%%%%%%%%%%%%%%%%%
\subsection{Motivation and general perspective}
%\label{sec:}
\unboldmath

Let $(M,\omega)$ be an exact symplectic manifold
where $\omega=d\lambda$.
On the free loop space $\Ll M:=C^\infty(\SS^1,M)$
consider the negative \textbf{area functional} given by
\begin{equation*}
   \Aa\colon \Ll M\to\R,\quad
   v\mapsto -\int_0^1 v^*\lambda .
\end{equation*}
A smooth function $H\colon M\to\R$, called \textbf{Hamiltonian},
induces on the loop space the corresponding \textbf{mean value functional}
\begin{equation*}
   \Hh=\Hh_H\colon \Ll M\to\R,\quad
   v\mapsto \int_0^1 H\circ v(t)\, dt.
\end{equation*}
On loop space there is the \textbf{time reversal involution} defined by
\[
   \Tt\colon \Ll M\to\Ll M,\quad v\mapsto v^-,\qquad v^-(t):=v(-t).
\]
There are the following relations
\begin{equation}\label{eq:relations-Aa-Hh}
   \Aa\circ \Tt=-\Aa,\qquad \Hh\circ \Tt=\Hh.
\end{equation}
The \textbf{Rabinowitz action functional} is defined by
\begin{equation*}
   \Aa_\Hh\colon \Ll M\times\R\to\R,\quad
   (v,\tau)\mapsto \Aa(v)+\tau\Hh(v).
\end{equation*}
The \textbf{extended time reversal involution} is defined by
\[
   \widetilde\Tt\colon \Ll M\times\R\to\Ll M\times\R,\quad
   (v,\tau)\mapsto (v^-,-\tau).
\]
From~(\ref{eq:relations-Aa-Hh}) it follows the anti-invariance of the
Rabinowitz action functional under extended time reversal involution, in symbols
\[
   \Aa_\Hh\circ\widetilde\Tt=-\Aa_\Hh.
\]
This has the consequence that the extended time reversal involution
also acts involutive on the critical point set, in symbols
\[
   (v,\tau)\in\Crit \Aa_\Hh\quad\Leftrightarrow\quad
   \widetilde\Tt(v,\tau)=(v^-,-\tau)\in\Crit \Aa_\Hh .
\]
A critical point $(v,\tau)$ for $\tau$ positive corresponds to a periodic orbit
of the Hamiltonian vector field of $H$ of energy zero and period $\tau$.
The critical point $\widetilde\Tt(v,\tau)=(v^-,-\tau)$ corresponds to
this orbit traversed \emph{backward} in time.
The fixed point set $\Fix\,\widetilde\Tt|_{\Crit\Aa_\Hh}$ are pairs
$(x,0)$ where $x$ is a point on the energy hypersurface $\Sigma:=H^{-1}(0)$
interpreted as a constant loop.

There is no analogue of the time reversal
anti-invariance of the Rabinowitz action functional $\Aa_\Hh$
in symplectic homology or symplectic field theory
where periodic orbits are always traversed in \emph{forward} time.

From a physical perspective the time reversal anti-invariance
is reminiscent of the Feynman-Stueckelberg
interpretation~\cite{Stueckelberg:1941b,Feynman:1948a}
of a positron as an electron going backward in time.

From a mathematical perspective the time reversal anti-invariance of the
Rabinowitz action functional has strong connections to Tate cohomology,
Poincar\'e-duality, and Frobenius algebras. It led to the discovery by
Cieliebak and Oancea~\cite{Cieliebak:2022a}
of the structure of a topological quantum field theory (TQFT) on
Rabinowitz-Floer homology.
However, the topological quantum field theory structure of
Cieliebak and Oancea is not defined on Rabinowitz-Floer homology
directly, but on $V$-shaped symplectic homology. The latter is
known to be isomorphic to Rabinowitz-Floer homology
as shown by Cieliebak, Frauenfelder, and Oancea~\cite{Cieliebak:2010b}.
The difficulty to define the TQFT structure directly on Rabinowitz-Floer
homology is that, in general, the Rabinowitz action functional does not
behave additively with respect to concatenation of loops.
For that reason, to our knowledge, nobody defined product structures
directly on Rabinowitz-Floer homology. Instead of that,
product structures were defined on homologies
isomorphic to Rabinowitz-Floer homology, namely, $V$-shaped
symplectic homology by Cieliebak and Oancea~\cite{Cieliebak:2015a},
respectively, on extended phase space by Abbondandolo and
Merry~\cite{Abbondandolo:2018c}.

For the following reasons we would like to see TQFT structure
on Rabinowitz-Floer homology directly.
\begin{enumerate}
\item
  Time reversal anti-invariance for the functional gets lost
  when going over to $V$-shaped symplectic homology,
  respectively, to extended phase space homology.
  Therefore Poincar\'e-duality only holds on homology level
  and not on chain level, as in the case of Rabinowitz action functional.
\item
  In contrast to symplectic homology the Rabinowitz gradient flow
  equation is not a PDE but a delay equation.
  Although the critical points of the Rabinowitz action functional are
  still solutions of an ODE, the Rabinowitz action functional can
  easily be generalized to delay equations. In fact, the functional
  $\Hh$ not necessarily has to be the mean value of a Hamiltonian on 
  the underlying manifold, but can be a more interesting functional on
  the free loop space. In particular, in this way one can model
  interacting particles whose interaction is not necessarily
  instantaneous, but can happen with some delay~\cite{Frauenfelder:2020b}.
  This is in particular of interest in a semi-classical treatment of
  Helium~\cite{Cieliebak:2021a}.
\end{enumerate}

As mentioned above the major difficulty to define a TQFT structure
on Rabinowitz-Floer homology directly is the complicated behavior of
the Rabinowitz action functional on the concatenation of loops.
To remedy this situation it was proposed in~\cite{Frauenfelder:2022a}
to take advantage of the following elementary fact.
Critical points of a Lagrange multiplier functional are in 1-1
correspondence with critical points of the restriction of the first
function to the constraint given by the vanishing of the second function.
In the case of the Rabinowitz action functional it means the
following. One restricts the negative area functional $\Aa$ to the constraint
$\Hh^{-1}(0)$, namely the hypersurface in the free loop space
consisting of loops whose mean value vanishes.
Note that concatenating two loops of mean value zero leads to another
loop of mean value zero.
Therefore the hypersurface $\Hh^{-1}(0)$ is invariant under concatenation.
Moreover, note that the area functional is additive with respect to
concatenation. Therefore the restriction of the area functional to
$\Hh^{-1}(0)$ has the potential of leading to a TQFT for which
Poincar\'e-duality holds on chain level and which should also lead to
topological quantum field theories for Hamiltonian delay equations.

\medskip
In view of the above remarks it is of major interest to understand how
the semi-infinite dimensional Morse homology in the sense of Floer
of the Rabinowitz action functional $\Aa_\Hh$
is related to the one of the restriction of the area functional $\Aa$
to $\Hh^{-1}(0)$.
Motivated by the general perspective we treat in this article the 
finite dimensional analogue of this question which already has its own
interest.

\boldmath
%%%%%%%%%%%%%%%%%%%%%%%%%%%%%%%%%%%
%%%%%%%%%%%%%%%%%%%%%%%%%%%%%%%%%%%
%%%%%%% Section:  %%%%%%%%%%%%%%%%%%%%%
%%%%%%%%%%%%%%%%%%%%%%%%%%%%%%%%%%%
%%%%%%%%%%%%%%%%%%%%%%%%%%%%%%%%%%%
\section[The Lagrange multiplier function \boldmath$F_H$ and its restriction $f$]
{Lagrange multiplier function and restriction}
%{Two Morse functions}
\label{sec:functionals}
\unboldmath

Suppose that on a Riemannian manifold $(M,G)$ are given two
smooth functions
$$
   F,H\colon M\to\R
$$
such that $0$ is a regular value of $H$, in symbols $H\pitchfork 0$.
The function $H$ plays the role of providing a \textbf{constraint}, namely
the smooth Riemannian hypersurface
\begin{equation}\label{eq:iota}
   \Sigma:=H^{-1}(0)\stackrel{\iota}{\INTO} M
   ,\qquad
   g:=\iota^* G
   ,\qquad
   f:=F|_\Sigma:=F\circ\iota\colon\Sigma\to\R ,
\end{equation}
equipped with the restriction of $F$ and were $\iota$ is the inclusion map.
Throughout we assume that $\Sigma$ is compact and without boundary.
We call $\Sigma$ the \textbf{base} of the adiabatic limit construction.
Now add to $F$ the constraint function $H$ times a parameter $\tau$
to define the \textbf{Lagrange multiplier function}
\begin{equation}\label{eq:findim-F_H}
   F_H\colon M\times\R\to\R,\quad (x,\tau)\mapsto
   F(x)+\tau H(x) .
\end{equation}
The restriction $F_H|_\Sigma=f$ is equal to the restriction of $F$.
The function $F_H$ has the significance that its critical points are
in bijection with the critical points $x$ of the restriction $f$ via their
so-called Lagrange multipliers~$\chi(x)$, see Lemma~\ref{le:Crit-F_H}.

\boldmath
%%%%%%%%%%%%%%%%%%%%%%%%%%%%%%%%%%%
%%%%%%%%%%%%%%%%%%%%%%%%%%%%%%%%%%%
\subsection{Hypersurface geometry}
\label{sec:hypersurface-geometry}
\unboldmath

As a preparation we recall relevant facts about the geometry
of Riemannian submanifolds following the excellent presentation
of O'Neill~\cite[Chap.\,4]{oneill:1983a}.

\smallskip
Let $(M,G)$ be a smooth Riemannian manifold and $H\colon M\to\R$ a
smooth\footnote{throughout smooth means $C^\infty$ smooth}
function with regular value $0$.
The level set~(\ref{eq:iota}) endowed with the restriction metric
is a smooth Riemannian hypersurface $(\Sigma,g)$ of $(M,G)$.
Let $\Xx(M)$ be the smooth vector fields along $M$ and $\Xx(\Sigma)$
those along~$\Sigma$.
Let $\widebar\Xx(\Sigma)$ be the restrictions to $\Sigma$ of
vector fields along $M$, equivalently, the sections of the
pull-back bundle $\iota^*TM\to\Sigma$.
On $(M,G)$ and $(\Sigma,g)$, respectively,
we denote the Levi-Civita connections by $\Babla{}$ and $\nabla$
and the exponential maps by $\Exp$ and $\exp$.

Gradients are orthogonal to level sets.
By definition of regular value and codimension $1$
the gradient of $H$ is nowhere zero along the hypersurface
$\Sigma=H^{-1}(0)$. Thus $\Babla{} H$ generates the normal bundle
$N\Sigma=\R\Babla{} H$ of $\Sigma$ and
\[
   T_\Sigma M
   =T\Sigma\stackrel{\perp}{\oplus} N\Sigma
   ,\qquad X=\xi+\nu,
\]
is an orthogonal direct sum along $\Sigma$. Hence for any
and $X\in T\Sigma M$ there are unique vectors $\xi\in T\Sigma$ and
$\nu\in N\Sigma$ such that $X=\xi+\nu$.
This defines two orthogonal projections $\tan$ and $\nor$,
see~(\ref{eq:findim-orth-split}) and~(\ref{eq:findim-gradf}).

We denote vectors of $TM$ and vector fields taking values in $TM$
by capital letters such as $X,Y$ and, in contrast, vectors of $T\Sigma$
and vector fields taking values in $T\Sigma$ be greek letters such as
$\xi,\eta$. By $\nu$ we denote elements of $N\Sigma$.
See Convention~\ref{con:notation} for notation of norms and inner
products. Here and throughout we silently identify $q\in\Sigma$ with
$\iota(q)\in M$ and $\xi\in T\Sigma$ with $T\iota(\xi)\in TM$.

\boldmath
%%%%%%%%%%%%%%%%%%%%%%%%%%%%%%%%%%%
%%%%%%%%%%%%%%%%%%%%%%%%%%%%%%%%%%%
\subsubsection{Orthogonal splitting of $TM$ along a neighborhood of $\Sigma$}
\unboldmath

For $p\in M$ the \textbf{gradient} $\Babla{} H(p)$ is determined
by $dH(p) X=\inner{\Babla{} H(p)}{X}$ $\forall X\in T_pM$. An open
neighborhood of $\Sigma$ is provided by the set of regular points
\[
   \Sigma \subset M_{\rm reg}
   :=\{p\in M\mid dH(p)\not= 0\}
   \subset M.
\]
Since $\Babla{} H(p)\not=0$ for $p\in M_{\rm reg}$, there are
the \textbf{canonical vector fields}
\[
   U:=\frac{\Babla{} H}{\abs{\Babla{} H}}
   ,\quad
   V:=\frac{\Babla{} H}{\abs{\Babla{} H}^2}
   ,\quad\text{along $M_{\rm reg}$} .
\]
The smooth function defined by
\begin{equation}\label{eq:chi}
   \chi:=-\tfrac{\inner{\Babla{} F}{\Babla{} H}}{\abs{\Babla{} H}^2}
   \quad\text{along $M_{\rm reg}$}
\end{equation}
provides the coefficient of the orthogonal projection of $\Babla{}F$
onto $-\Babla{}H$; see~(\ref{eq:findim-orth-split}).
Since $\inner{\Babla{} H}{\xi}=dH\,\xi=0$ for $\xi\in T\Sigma$, the sum
$T_\Sigma M=T\Sigma+\R\cdot\Babla{} H$ is direct and orthogonal. Thus
the line bundle $N\Sigma:=\R\Babla{} H$ is the normal bundle of
$\Sigma$. There are the associated orthogonal projections
\begin{equation}\label{eq:findim-orth-proj}
   \tan\colon T_qM\to T_q\Sigma,\quad
   \nor\colon T_qM\to \R U_q,\quad
   \tan+\nor=\Id_{T_qM} .
\end{equation}
The vectors of $\R U_q$ are said \textbf{normal} to $\Sigma$.
A vector field $Z\in\widebar{\Xx}(\Sigma)$ is called normal
to $\Sigma$ if each vector $Z(q)$ is.
Let $\Xx(\Sigma)^\perp$ be the vector fields normal to $\Sigma$, 
that is the sections of the line bundle $\R\Babla{}H\to\Sigma$.
There is the orthogonal vector bundle sum
$\widebar{\Xx}(\Sigma)=\Xx(\Sigma)\oplus\Xx(\Sigma)^\perp$.
The resulting orthogonal projections
\begin{equation}\label{eq:findim-orth-split}
\begin{aligned}
   \nor\colon \widebar{\Xx}(\Sigma)&\to \Xx(\Sigma)^\perp,
   &&&
   \tan\colon \widebar{\Xx}(\Sigma)&\to \Xx(\Sigma),
\\
   X&\mapsto\inner{X}{U}\, U
   =\tfrac{\inner{X}{\Babla{} H}}{\abs{\Babla{} H}^2}\, \Babla{} H.
   &&&
   X&\mapsto X-\nor\, X,
\end{aligned}
\end{equation}
are $C^\infty(\Sigma)$-linear and there is the identity
$\widebar{\Xx}(\Sigma)\ni X=\tan X+\nor\, X$.

\begin{lemma}[Gradients and orthogonal decomposition]
\label{le:findim-grad-f}
It holds that
\begin{equation}\label{eq:findim-gradf}
\begin{aligned}
   \tan \Babla{} F
   &=\Nabla{} f
   &&&
   \nor\, \Babla{} F&=-\chi \Babla{} H
   &&&
   \Babla{} F
   &=\Nabla{}f-\chi\Babla{}H
\\
   \nor\, X
   &=\tfrac{(dH)\, X}{\abs{\Babla{} H}^2} \Babla{} H
   &&&
   \Abs{\nor\, X}&\le\tfrac{\abs{(dH)\, X}}{m_H}
   &&&
   m_H:&=\min_\Sigma \Abs{\Babla{} H}>0
\end{aligned}
\end{equation}
pointwise at $q\in\Sigma$ and for every tangent vector $X\in T_q M$.
\end{lemma}

\begin{proof}
To identify $\Nabla{} f$ with the tangential part,
pick $\xi\in \Xx(\Sigma)$. Then
\begin{equation*}
\begin{split}
   \inner{\Nabla{}f}{\xi}_g
   &=df(\xi)=dF|_{\iota} \,d\iota(\xi)
   =\inner{\Babla{}F|_{\iota}}{d\iota(\xi)}_G
   =\inner{\Babla{}F|_{\iota}-\nor\, \Babla{}F|_{\iota}}{d\iota(\xi)}_G\\
   &=\inner{\tan \Babla{}F|_{\iota}}{\xi}_g .
\end{split}
\end{equation*}
We subtracted the normal since its inner product with the
tangent $d\iota(\xi)$ is zero. As the difference is tangent,
we change $G$ to $g$.
Next write $\nor(\Babla{} F)=\alpha \Babla{} H$ for some
$\alpha\in C^\infty(\Sigma)$.
Then the identity $\Babla{} F=\Nabla{}f+\alpha\Babla{}H$
is the splitting~(\ref{eq:findim-orth-proj}).
Scalar multiply the identity by the normal
$\Babla{} H$ to get that
\[
   \inner{\Babla{} F}{\Babla{} H}
  =0+\alpha\abs{\Babla{} H}^2 .
\]
Hence $\alpha=-\chi$ by~(\ref{eq:chi}). The term $\nor\, X$ is obvious.
\end{proof}

\boldmath
%%%%%%%%%%%%%%%%%%%%%%%%%%%%%%%%%%%
%%%%%%%%%%%%%%%%%%%%%%%%%%%%%%%%%%%
\subsubsection{Normal form of $H$ near $\Sigma$}
\unboldmath

Let $\kappa>0$ be the constant from the local properness
Hypothesis~\ref{hyp:uniform-Lag-bound-intro}.
The vector field $V:=\Babla{}H/\abs{\Babla{}H}^2$
along the open neighborhood $M_{\rm reg}:=\{dH\not=0\}$ of $M$
of $\Sigma$ in $M$ generates 
a local flow $\{\varphi_r\}$ on $M_{\rm reg}$.
Since $\Sigma$ is compact for $\delta\in(0,\kappa)$ small enough
the following map is a diffeomorphism onto its image
$$
   \varphi\colon\Sigma\times(-\delta,\delta)
   \to U_\Sigma=U_\Sigma(\delta):=\im\varphi\subset M
   ,\quad
   (q.r)\mapsto\varphi_r q .
$$
(The map $\varphi$ provides a retraction
$\rho=\rho^2\colon U_\Sigma\to U_\Sigma$.\footnote{
  To match the abstract approach~\cite{Frauenfelder:2022h} define,
  for each $t\in[0,1]$, a map $\rho_t\colon U_\Sigma\to U_\Sigma$,
  $p=\varphi_r q\mapsto\varphi_{-tr}p$. Then
  $\rho_0=\id_{U_\Sigma}$, $\rho_1\colon U_\Sigma\to\Sigma$,
  $\rho_t|_\Sigma=\id_\Sigma$ $\forall t\in[0,1]$.
  So $\rho:=\rho_1=\rho^2$.
  })
The identities
\[
   H(\varphi_0 q)=0
   ,\qquad
   \tfrac{d}{dr}H(\varphi_r q)
   =dH|_{\varphi_r q}\,\tfrac{d}{dr} \varphi_r q
   =\INNER{\Babla{}H|_{\varphi_r q}}{V|_{\varphi_r q}}=1,
\]
show that
\begin{equation}\label{eq:normal-form-H}
   H(\varphi_r q)=r
\end{equation}
for every $(q,r)\in\Sigma\times(-\delta,\delta)$.
Thus, for every map $u\colon\R\to M$ that takes values in the image of
the flow diffeomorphism $\varphi$, there are maps
$\qfrak\colon\R\to\Sigma$ and $r\colon\R\to(-\delta,\delta)$, namely
$(\qfrak,r):=\varphi^{-1}(u)$ pointwise, such that
\begin{equation}\label{eq:normal-form-u}
   u=\varphi_r (\qfrak)
   ,\qquad
   r=H(u) ,
\end{equation}
pointwise at $s\in\R$.

\boldmath
%%%%%%%%%%%%%%%%%%%%%%%%%%%%%%%%%%%
%%%%%%%%%%%%%%%%%%%%%%%%%%%%%%%%%%%
\subsubsection{Induced connection}
\unboldmath

The Levi-Civita connections associated to $(M,G)$ and
$(\Sigma,g)$ are maps
\begin{equation*}
   \widebar{\nabla}\colon
      \Xx(M)\times\Xx(M)\to\Xx(M),\qquad
   \nabla\colon
      \Xx(\Sigma)\times\Xx(\Sigma)\to\Xx(\Sigma) .
\end{equation*}
Via vector field extension from the domain $\Sigma$ to $M$ the
connection $\widebar{\nabla}$ gives rise to a map, independent of
the chosen extensions $\widebar{\xi},\widebar{X}$,
the \textbf{induced connection}
\[
   \widebar{\nabla}\colon
   \Xx(\Sigma)\times\widebar{\Xx}(\Sigma)\to\widebar{\Xx}(\Sigma)
   ,\quad
   (\xi,X)\mapsto \Babla{\xi}X:=\Babla{\widebar{\xi}}\widebar{X},
\]
still denoted by the same symbol $\widebar{\nabla}$.

\begin{lemma}\label{le:findim-ind-conn}
The induced connection
satisfies the five axioms that characterize the Levi-Civita
connection on the tangent bundle of a Riemannian manifold:
\begin{itemize}\setlength\itemsep{0ex} 
\item[\rm (i)]
  $C^\infty(\Sigma)$-linear in $\xi$
  \hfill
  $\Babla{f\xi} X=f \Babla{\xi} X$
\item[\rm (ii)]
  $\R$-linear in $X$
  \hfill
  $\Babla{\xi} (\alpha X)=\alpha \Babla{\xi} X$
\item[\rm (iii)]
  Leibniz rule
  \hfill
  $\Babla{\xi} (fX)=(\xi f) X+f \Babla{\xi} X$
\item[\rm (iv)]
  torsion free
  \hfill
  $[\xi,\eta]:=\xi\eta-\eta\xi=\Babla{\xi} \eta-\Babla{\eta} \xi$
\item[\rm (v)]
  metric
  \hfill
  $\xi\inner{X}{Y}=\inner{\Babla{\xi} X}{Y}+\inner{X}{\Babla{\xi} Y}$
\end{itemize}
for all $\alpha\in\R$, $f\in C^\infty(\Sigma)$,
$\xi,\eta\in\Xx(\Sigma)$, and $X,Y\in\widebar{\Xx}(\Sigma)$,
and where $\xi f$ is a convenient shorter way to write $df(\xi)$.
\end{lemma}

\begin{remark}\label{rem:finfim-tan-nor}
If both vector fields $\xi,\eta$ take values in $T\Sigma$, by torsion
freeness the difference $\Babla{\xi} \eta-\Babla{\eta} \xi$
takes values in $T\Sigma$ as well -- the commutator does. 
This is in general not true for the individual terms. 
Via the orthogonal projections~(\ref{eq:findim-orth-split}) one
decomposes the vector field $\Babla{\xi}\eta\in\widebar{\Xx}(\Sigma)$
into a tangent and a normal part
\begin{equation}\label{eq:findim-decomp-II}
   \Babla{\xi}\eta
   =\Nabla{\xi}\eta+\mathrm{II}(\xi,\eta)
\end{equation}
whenever $\xi,\eta\in\Xx(\Sigma)$ and where
\begin{equation}\label{eq:findim-decomp-III}
   \Nabla{\xi}\eta=\tan \Babla{\xi}\eta\in\Xx(\Sigma),\qquad
   \mathrm{II}(\xi,\eta):=\nor\, \Babla{\xi}\eta\in\Xx(\Sigma)^\perp .
\end{equation}
\end{remark}

The \textbf{second fundamental form tensor} $\mathrm{II}$
of the Riemannian submanifold $\Sigma$ of $M$ is
$C^\infty(\Sigma)$-bilinear and symmetric.
In our codimension $1$ case $U$ generates $\Xx(\Sigma)^\perp$,
so $\mathrm{II}(\xi,\eta)$ is a $C^\infty(\Sigma)$-multiple of $U$.
Multiply~(\ref{eq:findim-decomp-II}) by $U$ to~get
\begin{equation}\label{eq:findim-II}
   \mathrm{II}(\xi,\eta)=\mu(\xi,\eta)\cdot U
   =\tfrac{\inner{\Babla{\xi}\eta}{\Babla{}H}}{\abs{\Babla{}H}^2}\, \Babla{}H
   ,\qquad
   \mu(\xi,\eta)=\inner{\Babla{\xi}\eta}{U} .
\end{equation}
The tensor $\mathrm{II}$ appears in the formal adjoint operator
$(D^0_q)^*$, see~(\ref{eq:adjoint0-findim-lin}).
The \textbf{second fundamental form} $B$ and the
\textbf{shape operator} $S$, both associated to the unit normal vector
field $U$, so determined up to sign, are defined~by
\[
   B(\xi,\eta)
   :=\inner{S\xi}{\eta}
   \stackrel{\text{def. $S$}}{=}
   \inner{\mathrm{II}(\xi,\eta)}{U}
   \stackrel{(\ref{eq:findim-II})}{=}
   \inner{\Babla{\xi}\eta}{U}
\]
for all $\xi,\eta\in\Xx(\Sigma)$. But $0=\xi\inner{\eta}{U}
=\inner{\Babla{\xi}\eta}{U}+\inner{\eta}{\Babla{\xi}U}$. Therefore
the shape operator at $q\in\Sigma$ is the symmetric linear map
\[
   S\colon T_q\Sigma\to T_q\Sigma,\quad
  \xi\mapsto-\Babla{\xi}U .
\]
Implicitly this tells that $\Babla{\xi}U$ is tangent to $\Sigma$
(alternatively hit $\inner{U}{U}=1$ by $\xi$).

%\newpage.\newpage
\boldmath
%%%%%%%%%%%%%%%%%%%%%%%%%%%%%%%%%%%
%%%%%%%%%%%%%%%%%%%%%%%%%%%%%%%%%%%
\subsection{Critical points are in canonical bijection}
\unboldmath

Critical points of $f=F|_\Sigma$ satisfy $x\in\Sigma$~and
\begin{equation}\label{eq:df}
   0
   =\Nabla{}f(x)
   \stackrel{(\ref{eq:findim-gradf})}{=}
   \bigl(\Babla{}F+\chi\Babla{}H\bigr)(x)
   \quad\Leftrightarrow\quad
   \bigl(dF+\chi dH\bigr)(x)=0 .
\end{equation}
A point $(p,\tau)\in M\times\R$ is critical for the function
$F_H(p,\tau)=F(p)+\tau H(p)$
%see~(\ref{eq:findim-F_H}), 
iff the derivative vanishes
\begin{equation}\label{eq:dF_H}
   dF_H(p,\tau)\begin{pmatrix}X\\ \ell\end{pmatrix}
   =
   dF(p) X+\tau\cdot dH(p) X+\ell\cdot H(p)
   =0
\end{equation}
for all $X\in T_p M$ and $\ell\in\R$.
Fix $X=0$ to obtain $H(p)=0$, that is $p\in\Sigma$.
Now fix $\ell=0$ and set $x:=p$ to obtain that $(x,\tau)$
is a critical point of $F_H$ iff
\begin{equation*}%\label{eq:Crit-F_H} MULTIPLY DEFINED
   dF(x)+\tau\cdot dH(x)=0,\qquad x\in\Sigma .
\end{equation*}

\boldmath
%%%%%%%%%%%%%%%%%%%%%%%%%%%%%%%%%%%
%%%%%%%%%%%%%%%%%%%%%%%%%%%%%%%%%%%
\subsubsection{Canonical embedding}
\unboldmath

\begin{definition}[Canonical embedding]
The graph map of $\chi\colon\Sigma\to\R$, cf.~(\ref{eq:chi}),
\begin{equation}\label{eq:can-emb}
   i:\Sigma\to M\times\R,\quad
   q\mapsto\left(q,\chi(q)\right)
   {\color{gray}\;=\left(\iota(q),\chi(\iota(q))\right)},
\end{equation}
is called the \textbf{canonical embedding}.
The derivative is denoted and given by
\begin{equation}\label{eq:can-emb-deriv}
   I_q:=di(q)\colon T_q\Sigma\to T_q M\times\R,\quad
   \xi\mapsto\left(\xi,d\chi(q)\xi\right) .
\end{equation}
\end{definition}
For simplicity of notation we usually abbreviate $\iota(q)$ by $q$
and $d\iota(q)\xi$ by $\xi$.
Graph maps of smooth functions are embeddings.
The Lagrange function $F_H(p,\tau)=F(p)+\tau H(p)$ coincides along
the image of $i$ with the restriction $f=F|_\Sigma$ to the
zero level $\Sigma$ of $H$, in symbols
\[
   F_H\circ i=f .
\]

\begin{lemma}\label{le:Crit-F_H}
The critical points of $F_H$ and $f$ are in bijection, more precisely
\begin{equation}\label{eq:Crit-F_H}
\begin{split}
   \Crit  F_H
   &=i(\Crit f)\\
   &=\{(x,\chi(x))\in\Sigma\times\R\mid dF(x)+\chi(x)\cdot dH(x)=0\} .
\end{split}
\end{equation}
In particular, along critical points $x$ both functions coincide
$f(x)=F_H(x,\chi(x))$.
\end{lemma}

\begin{proof}
Compare~(\ref{eq:df}) and~(\ref{eq:Crit-F_H}) where $dH(x)\not=0$
implies $\tau=\chi(x)$.
\end{proof}

\boldmath
%%%%%%%%%%%%%%%%%%%%%%%%%%%%%%%%%%%
%%%%%%%%%%%%%%%%%%%%%%%%%%%%%%%%%%%
\subsubsection{Hessians and Morse indices}
\unboldmath

Suppose $x\in\Sigma$ is a \textbf{non-degenerate} critical point of
$f$, that is $0$ is not an eigenvalue of the Hessian operator,
the covariant derivative of $\Nabla{}f$ at $x$, namely
\[
   A_x^0\colon T_x\Sigma\to T_x\Sigma,\quad
   \xi\mapsto D\Nabla{}f(x)\xi=\Nabla{\xi}\Nabla{} f(x) .
\]
This linear map is symmetric;
see identity 2 and 3 in~(\ref{eq:adjoint0-findim-lin-2}) further below.
In local coordinates $A_x^0$ is represented by the Hessian matrix of
second derivatives $a_x^f=\left(\p_i\p_j f(x)\right)_{i,j=1}^n$. This matriz is
symmetric, hence admits $n$ real eigenvalues, counted with
multiplicities. While the Hessian matrix depends on the choice of
coordinates, the number of negative eigenvalues does not.
The number $k$ of negative eigenvalues, counted with multiplicity, of the
Hessian operator $A_x^0$ or, equivalently, of any Hessian matrix $a_x^f$
is called the \textbf{Morse index} of $x$, in symbols $\IND_f(x)=k$.

In the transition from $f$ to $F_H$, in terms of critical points
from $x\in\Sigma$ to $(x,\chi(x))\in M\times\R$, two new eigenvalues
appear, one is positive and the other one is negative.
This result is due to the first
author~\cite{Frauenfelder:2006a}
where the proof is in local coordinates.
It is easy to obtain such coordinates in our scenario:
for the submanifold $H^{-1}(0)\INTO M$ use submanifold coordinates
and for the orthogonal complement use the local flow generated by
the gradient of $H$ suitably rescaled.

\begin{lemma}[Morse index increases by $1$]
\label{le:findim-Morse-indices}
If $x\in\Crit\, f$ is non-degenerate, then so is $(x,\chi(x))\in\Crit\, F_H$
and the Morse index increases by one, in symbols
\begin{equation*}
\begin{split}
   \IND_{F_H}(x,\chi(x))=\IND_f(x)+1 .
\end{split}
\end{equation*}
\end{lemma}

\begin{remark}
By Lemma~\ref{le:Crit-F_H} and~\ref{le:findim-Morse-indices},
if $f$ is Morse, so is $F_H$. Let $f$ be Morse.
Since the dimension difference $\dim M-\dim\Sigma=2$ is two,
there always arises together with the negative Hessian eigenvalue exactly one
positive eigenvalue. Consequently the Hessian
of $F_H$ at a critical point is never negative (positive\footnote{
  Replace $f$ by $-f$.
  })
definite. Hence critical points of $F_H$ are not minima (maxima), hence not
detectable by direct methods using minimization (maximization).
\end{remark}

\begin{proof}
Given $F,G\colon M\to\R$ with $G\pitchfork 0$, let
$\Sigma:=H^{-1}(0)\subset M$.
Pick a critical point $x$ of $f=F|_\Sigma\colon\Sigma\to\R$.
Choose a local coordinate chart between open subsets
\[
   \phi\colon M\supset V\to U\subset\R^n
   ,\quad p\mapsto\phi(p)=(z_1,\dots,z_{n},r)=(z,r),
\]
which takes $x$ to the origin of $\R^n$ and has the following properties:
\begin{itemize}\setlength\itemsep{0ex} 
\item[a)]
  the part of $\Sigma$ in $V$ corresponds to the part of
  $\R^{n}\times 0$ in $U$;
\item[b)]
  in local coordinates $H$ is given by $(z,r)\mapsto r$.
  \hfill {\small\color{gray} $H(z,r)=r$}
\end{itemize}
Such coordinates exist: By compactness of $\Sigma$ there is
a constant $\delta>0$ such that the vector field
$V=\Babla{}H/\abs{\Babla{}H}^2$ along $M_{\rm reg}$
generates a local flow, notation
$
   \varphi\colon\Sigma\times(-\delta,\delta)\to M
$,
$
  (q,r)\mapsto\varphi_r q
$.
The identities
\[
   H(\varphi_0 q)=0
   ,\qquad
   \tfrac{d}{dr}H(\varphi_r q)
   =dH|_{\varphi_r q}\,\tfrac{d}{dr} \varphi_r q
   =\INNER{\Babla{}H|_{\varphi_r q}}{\tilde U|_{\varphi_r q}}=1,
\]
show that $H(\varphi_r (q))=r$.
Compose $\varphi$ with submanifold coordinates of $\Sigma$ in $M$.

In the following local coordinate representations of maps
are denoted by the same symbols as the maps themselves.
For instance, for $F$ in our local coordinates we write $F(z,r)$.
In these local coordinates we have
\[
   \text{(i) $f(z)=F(z,0)$},\qquad
   \text{(ii) $F_H(z,r,\tau)=F(z,r)+\tau r$}.
\]
The proof proceeds in two steps. First we consider the special case where
$F(z,r)=f(z)$, second we homotop the general case to the special case.

\medskip
\noindent
\textbf{Special case \boldmath$F(z,r)=f(z)$.}
The gradient of $F_H(z,r,\tau)=f(z)+\tau r$
is $\Nabla{}F_H(z,r,\tau)=\left(\Nabla{}f(z),\tau,r\right)$, so
the Hessian at the critical point $(x,0,\chi(x))$ is
\[
   a_0:=a_{(x,0,\chi(x))}^{F_H}=
   \begin{bmatrix}
      a_x^f&0&0\\
      0&0&1\\
      0&1&0
   \end{bmatrix} .
\]
Since the lower $2\times 2$ diagonal block has eigenvalues $-1,+1$ we
are done.

\medskip
\noindent
\textbf{General case \boldmath$F(z,r)$.}
The gradient of $F_H(z,r,\tau)=F(z,r)+\tau r$ is given by
$\Nabla{}F_H(z,r,\tau)=\left(\Nabla{1}F(z,r),\Nabla{2}F(z,r)+\tau,r\right)$,
so the Hessian at $(x,0,\chi(x))$ is the matrix $a_1=a_{(x,0,\chi(x))}^{F_H}$
given by setting $s=1$ in the interpolating family

\[
   a_{\color{red}s}:=
   \begin{bmatrix}
      a_x^f&{\color{red}s}\Nabla{2}\Nabla{1} F(x,0)&0\\
      {\color{red}s}\Nabla{1}\Nabla{2} F|_{(x,0)}
         &{\color{red}s}\Nabla{2}\Nabla{2} F|_{(x,0)}&1\\
      0&1&0
   \end{bmatrix}
   ,\quad s\in[0,1] .
\]
Zero is not an eigenvalue of $a_1$: Let
$(\xi,R,T)\in\ker a_1\subset\R^{n}\times\R\times\R $, then
\[
   \begin{bmatrix}0\\0\\{\color{cyan} 0}\end{bmatrix}
   =a_1\begin{bmatrix}\xi\\R\\T\end{bmatrix}
   =\begin{bmatrix}
      a_x^f\xi+\Nabla{2}\Nabla{1} F(x,0)R\\
      \Nabla{1}\Nabla{2} F|_{(x,0)} \xi+\Nabla{2}\Nabla{2} F|_{(x,0)} R+T\\
      {\color{cyan} R}
   \end{bmatrix}
   =\begin{bmatrix}
      a_x^f\xi\\
      \Nabla{1}\Nabla{2} F|_{(x,0)} \xi+T\\
      0
   \end{bmatrix} .
\]
Thus $R=0$. Since $a_x^f$ does not have eigenvalue zero,
if $a_x^f\xi=0$, then $\xi=0$, so $T=0$.
For any $s\in[0,1)$ the same argument shows that the matrix~$a_s$
does not have eigenvalue $0$. But each eigenvalue depends
continuously on the matrix $a_s$, so $a_1$ and $a_0$ do have the same
number of negative/positive eigenvalues.
\end{proof}

\boldmath
%%%%%%%%%%%%%%%%%%%%%%%%%%%%%%%%%%%
%%%%%%%%%%%%%%%%%%%%%%%%%%%%%%%%%%%
%%%%%%% Section:  %%%%%%%%%%%%%%%%%%%%%
%%%%%%%%%%%%%%%%%%%%%%%%%%%%%%%%%%%
%%%%%%%%%%%%%%%%%%%%%%%%%%%%%%%%%%%
\section{Downward gradient flows}
\label{sec:DGF}
\unboldmath

\boldmath
%%%%%%%%%%%%%%%%%%%%%%%%%%%%%%%%%%%
%%%%%%%%%%%%%%%%%%%%%%%%%%%%%%%%%%%
%%%%%%%%%%%%%%%%%%%%%%%%%%%%%%%%%%%
\subsection[Base downward gradient flow of $f$ on $\Sigma$]
{Base flow}
%%%%%%%%%%%%%%%%%%%%%%%%%%%%%%%%%%%
%%%%%%%%%%%%%%%%%%%%%%%%%%%%%%%%%%%
\label{sec:DGF-base}
\unboldmath

The downward gradient equation on the regular hypersurface
$(\Sigma,g)=(H^{-1}(0),\iota^*G)$ of the restriction
$f=F|_\Sigma\colon\Sigma\to\R$ is given by
\begin{equation}\label{eq:DGF-ff-pre}
\begin{aligned}
   \p_sq
   =-\Nabla{}f(q)
   &\stackrel{(\ref{eq:findim-gradf})}{=}
   -\bigl(\Babla{}F+\chi\Babla{}H\bigr)(q)
   \\
   f&\stackrel{{\color{white}(1.5)}}{=}
   F\circ\iota\colon\Sigma\to\R
\end{aligned}
\end{equation}
for smooth maps $q\colon\R\to \Sigma$ and where $\chi$ is defined
by~(\ref{eq:chi}).

Pointwise evaluation at $s\in\R$ extends the canonical
embedding~(\ref{eq:can-emb}) from points in $\Sigma$ to smooth maps
$q\colon\R\to\Sigma$. The induced embedding, still denoted~by
\begin{equation}\label{eq:finfim-i}
   i(q)=\left(\iota\circ q,\chi\circ\iota\circ q\right)=:(u,\tau)
   ,\qquad q\colon\R\to\Sigma,
\end{equation}
is injective consisting of a pair of maps
$u=\iota\circ q\colon\R\to M$
and $\tau=\chi\circ\iota\circ q\colon\R\to\R$.

Consider the pair of equations
\begin{equation}\label{eq:DGF-ff}
\begin{aligned}
   \p_su
   +\Babla{}F(u)+\tau\Babla{}H(u)
   &=0
   \\
   H\circ u&=0
\end{aligned}
\end{equation}
for smooth maps $(u,\tau)\colon\R\to M\times\R$.

\begin{lemma}[Base equation]\label{def:DGF-ff}
If $q\colon\R\to\Sigma$ solves~(\ref{eq:DGF-ff-pre}), then
$(u,\tau):=i(q)$ as defined by~(\ref{eq:finfim-i})
solves~(\ref{eq:DGF-ff}) and every solution of~(\ref{eq:DGF-ff})
arises this way.
\end{lemma}

\begin{proof}
Identifying domain and image of $\iota\colon\Sigma\to M$
the first lines of~(\ref{eq:DGF-ff-pre}) and of~(\ref{eq:DGF-ff}) are
just the same equation whenever $u=\iota(q)$ and $\tau=\chi(q)$.

Now suppose $(u,\tau)\colon\R\to M\times\R$ solves~(\ref{eq:DGF-ff}).
By the second equation $u$ takes values in $\Sigma$. This has two
consequences. 
  Firstly, we can view $u$ as a map to $\Sigma$, notation
$q_u\colon\R\to\Sigma$.
  Secondly, the derivative $\p_s u$ is tangential to
$\Sigma$, hence so is $-\p_su =\Babla{}F(u)+\tau\Babla{}H(u)$.
Take the inner product with the normal field $\Babla{}H(u)$ to get
$0=\inner{\Babla{}F(u)}{\Babla{}H(u)}+\tau\abs{\Babla{}H(u)}^2$.
By definition~(\ref{eq:chi}) this means that $\tau=\chi(u)=\chi(q_u)$.
Hence $i(q_u)=\left(\iota(q_u),\chi(\iota(q_u))\right)=(u,\tau)$.
\end{proof}

\boldmath
%%%%%%%%%%%%%%%%%%%%%%%%%%%%%%%%%%%
%%%%%%%%%%%%%%%%%%%%%%%%%%%%%%%%%%%
\subsubsection{Base energy $E^0$}
\unboldmath

Given critical points $x^\mp$ of $f\colon \Sigma\to\R$, we impose
on a smooth map $q\colon\R\to \Sigma$ the asymptotic
boundary conditions
\begin{equation}\label{eq:limit-ff}
   \lim_{s\to\mp\infty} q_s=x^\mp .
\end{equation}

\begin{definition}\label{def:energy-ff}
Define the \textbf{base energy} of a smooth map
$q\colon\R\to\Sigma$ by
\begin{equation*}
\begin{split}
   E^0(q)
   &\stackrel{\;\:\text{def.}\;\:}{=}
   \tfrac12\int_{-\infty}^\infty
   \Abs{\p_s q_s}^2
   +\Abs{\Nabla{}f(q_s)}^2 ds\\
   &\stackrel{(\ref{eq:DGF-ff-pre})}{=}
   \tfrac12\int_{-\infty}^\infty
   \Abs{\p_s q_s}^2
   +\Abs{\Babla{} F(q_s)+\chi(q_s)\cdot\Babla{} H(q_s)}^2 ds\\
   &\stackrel{\;\:\text{def.}\;\:}{=}
   E^0(q,\chi(q)) .
\end{split}
\end{equation*}
\end{definition}

\begin{lemma}[Energy identity]
\label{le:energy-identity-0}
Let $q\colon\R\to\Sigma$ be a smooth solution
of~(\ref{eq:DGF-ff-pre}).
Then the energy is bounded by the oscillation of $f$ and there is
the energy identity
\begin{equation}\label{eq:energy-0-osc}
   E^0(q)
   \stackrel{(\ref{eq:DGF-ff-pre})}{=}
   \Norm{\p_sq}^2
   \le\osc f
   :=\max f-\min f<\infty
\end{equation}
where $\norm{\cdot}$ is the $L^2$ norm.
With asymptotic boundary conditions~(\ref{eq:limit-ff}) it holds
\begin{equation}\label{eq:energy-0}
   E^0(q)
   \stackrel{(\ref{eq:DGF-ff-pre})}{=}
   \Norm{\p_sq}^2
   \stackrel{(\ref{eq:limit-ff})}{=}
   f(x^-)-f(x^+)=: c^*.
\end{equation}

\end{lemma}

\begin{proof}
We see that
\begin{equation*}
\begin{split}
   E^0(q)
   &\stackrel{(\ref{eq:DGF-ff-pre})}{=}
   \lim_{T\to\infty}\int_{-T}^T \Abs{\p_sq_s}^2 ds
   \\
   &\stackrel{(\ref{eq:DGF-ff-pre})}{=}
   \lim_{T\to\infty}\int_{-T}^T
   -\inner{\Nabla{}f(q_s)}{\p_sq_s}\, ds
   \\
   &\stackrel{{\color{white}(3.33)}}{=}
   -\lim_{T\to\infty}\int_{-T}^T\tfrac{d}{ds} f(q_s)\, ds
   \\
   &\stackrel{{\color{white}(3.33)}}{=}
   \lim_{T\to\infty} \left(f(q_{-T})-f(q_T)\right) .
%   \\
%   &\stackrel{(\ref{eq:limit-ff})}{=}
%   f(x^-)-f(x^+).
\end{split}
\end{equation*}
Now both, (\ref{eq:energy-0-osc}) and (\ref{eq:energy-0}), are obvious.
\end{proof}

\boldmath
%%%%%%%%%%%%%%%%%%%%%%%%%%%%%%%%%%%
%%%%%%%%%%%%%%%%%%%%%%%%%%%%%%%%%%%
\subsection[Ambient flow deformation by a parameter $\eps$]
{Ambient flow and deformation}
%%%%%%%%%%%%%%%%%%%%%%%%%%%%%%%%%%%
%%%%%%%%%%%%%%%%%%%%%%%%%%%%%%%%%%%
\label{sec:deformation}
\unboldmath

\textbf{Ambient flow.}
We endow the product $M\times\R$ with the product metric
$h^1:=G\oplus 1$ and the associated Levi-Civita connection $\nabla^1$.
The downward gradient equation for the function
$F_H\colon M\times\R\to\R$ from~(\ref{eq:findim-F_H}), namely
$\p_s z=-\nabla^1 F_H(z)$, is according to~(\ref{eq:dF_H}) given by
the pair of equations
\begin{equation}\label{eq:DGF-F_H}
   \begin{pmatrix}
      \p_su\\
      \tau^\prime
   \end{pmatrix}
   =\p_s z=-\nabla^1 F_H(z)
   =-
   \begin{pmatrix}
      \Babla{} F(u)+\tau\Babla{} H(u)\\
      H(u)
   \end{pmatrix}
\end{equation}
for smooth maps $z=(u,\tau)\colon\R\to M\times\R$.
The ambient energy $E^1$ is $E^{\eps=1}$
in Definition~\ref{def:energy-def-I}.

\medskip\noindent
\textbf{Deformed flow.}
For $\eps>0$ consider on $M\times\R$ the rescaled Riemannian metric
and associated Levi-Civita connection
\begin{equation}\label{eq:h^eps}
   h^\eps:=G\oplus\eps^2,\qquad \nabla^\eps.
\end{equation}
Thus the inner product of elements $Z=(X,\ell)$ and
$\tilde Z=(\tilde X,\tilde \ell )$ of $T_uM\times\R$ is
\[
   h^\eps(Z,\tilde Z)
   =\inner{X}{\tilde X}+\eps^2 \ell  \tilde \ell 
   ,\qquad
   \abs{Z}_\eps^2:=h^\eps(Z,Z)
   =\abs{X}^2+\eps^2 \ell ^2.
\]
By~(\ref{eq:dF_H}) the downward $\eps$-gradient equation for the
function $F_H$ on $M\times\R$ is
\begin{equation}\label{eq:DGF-def-I}
   \begin{pmatrix}
      \p_s u\\
      \tau^\prime
   \end{pmatrix}
   =\p_s z=-\nabla^\eps F_H(z)
   =-
   \begin{pmatrix}
      \Babla{}F(u)+\tau\Babla{}H(u)\\
      \eps^{-2} H(u)
   \end{pmatrix}
\end{equation}
for smooth maps $z=(u,\tau)\colon\R\to M\times\R$.

Multiply the second equation by $\eps^2$ and
formally set $\eps=0$ to obtain that $H(u_s)=0$ $\forall s\in\R$.
This suggests that in the limit
$\eps\to 0$ the solutions to~(\ref{eq:DGF-def-I}) converge to a
solution of the base equation~(\ref{eq:DGF-ff}).

\boldmath
%%%%%%%%%%%%%%%%%%%%%%%%%%%%%%%%%%%
%%%%%%%%%%%%%%%%%%%%%%%%%%%%%%%%%%%
\subsubsection{Ambient energy $E^\eps$}\label{sec:E^eps}
\unboldmath

Given critical points $x^\mp$ of $f\colon\Sigma\to\R$, impose
on a smooth map $(u,\tau)\colon\R\to M\times\R$
the asymptotic boundary conditions
\begin{equation}\label{eq:limit-F_H}
   \lim_{s\to\mp\infty} \left(u_s,\tau_s\right)
   =\left(x^\mp,\chi(x^\mp)\right)\stackrel{(\ref{eq:Crit-F_H})}{\in}\Crit F_H.
\end{equation}

\begin{definition}\label{def:energy-def-I}
The \textbf{\boldmath$\eps$-energy} of a smooth map
$z=(u,\tau)\colon\R\to M\times\R$ is
\begin{equation*}
\begin{split}
   E^\eps(u,\tau)
   :&=\tfrac12\int_{-\infty}^\infty
   \Abs{\p_s z_s}^2_{\eps}
   +\abs{\nabla^\eps F_H(z_s)}^2_{\eps} ds \\
   &=\tfrac12\int_{-\infty}^\infty
   \Abs{\p_s u_s}^2+\eps^2{\tau^\prime_s}^2
   +\Abs{\Babla{}F(u_s)+\tau_s\Babla{} H(u_s)}^2
   +\eps^{-2}H(u_s)^2
   ds.
\end{split}
\end{equation*}
\end{definition}

\begin{lemma}[Energy identity]\label{le:energy-identity-def-I}
Given $\eps>0$, let $(u,\tau)\colon\R\to M\times\R$ be a solution
of~(\ref{eq:DGF-def-I}). Then the following is true.
%the energy is bounded by the oscillation of $f$ and 
a) There is the identity
\begin{equation}\label{eq:energy-eps-osc}
   E^\eps(u,\tau)
   \stackrel{(\ref{eq:DGF-def-I})}{=}
   \Norm{\p_su}^2+\eps^2\norm{\tau^\prime}^2
   \in[0,\infty]
%   \le\osc f
%   :=\max f-\min f<\infty
\end{equation}
where $\norm{\cdot}$ denotes $L^2$ norms.
b) If, in addition, the energy $E^\eps(u,\tau)<\infty$ is finite, then
the energy is bounded by the oscillation of $f$, in symbols
\[
   E^\eps(u,\tau)
   \le\max f-\min f=: \osc f<\infty .
\]
c) In case of asymptotic boundary
conditions~(\ref{eq:limit-F_H}) there is the energy identity
\begin{equation}\label{eq:energy-eps}
   E^\eps(u,\tau)
   \stackrel{(\ref{eq:DGF-def-I})}{=}
   \Norm{\p_su}^2+\eps^2\norm{\tau^\prime}^2
   \stackrel{(\ref{eq:limit-F_H})}{=}
   f(x^-)-f(x^+)=: c^* .
\end{equation}
\end{lemma}

\begin{proof}
Fix $\eps>0$.
We see that
\begin{equation*}
\begin{split}
   E^\eps(u,\tau)
   &\stackrel{(\ref{eq:DGF-def-I})}{=}
   \lim_{T\to\infty}\int_{-T}^T
   \Bigl(\Abs{\p_su_s}^2+\eps^2{\tau^\prime_s}^2\Bigr)\, ds
   \\
   &\stackrel{(\ref{eq:DGF-def-I})}{=}
   \lim_{T\to\infty}\int_{-T}^T
   -\INNER{\p_s u_s}{\Babla{}F(u_s)-\tau_s\Babla{}H(u_s)}_G
   -\tau^\prime_s\cdot H(u_s)\, ds
   \\
   &\stackrel{{\color{white}(3.30)}}{=}
   -\lim_{T\to\infty}\int_{-T}^T
   dF|_{u_s} \p_su_s+\tau_s\, dH|_{u_s} \p_su_s 
   +\tau^\prime_s\cdot H(u_s)\, ds
   \\
   &\stackrel{{\color{white}(3.30)}}{=}
   -\lim_{T\to\infty}\int_{-T}^T
   \frac{d}{ds} \Bigl(F(u_s)+\tau_s
     H(u_s)\Bigr)\, ds
   \\
   &\stackrel{{\color{white}(3.30)}}{=}
   \lim_{T\to\infty} \left(F_H(u_{-T},\tau_{-T})
   -F_H(u_T,\tau_T)\right).
%   \\
%   &\stackrel{{\color{white}(3.30)}}{\le}
%   \max f-\min f .
\end{split}
\end{equation*}
This proves a) and also c) since $F_H(x^-,\chi(x^-))=
F(x^-)+\chi(x^-) H(x^-)=f(x^-)\le\max f$ and similarly at $x^+$.
b) That the right hand side of the displayed formula is
$\le \max f-\min f$ will be proved in two steps.

\medskip
\noindent
\textbf{Step 1.}
{\it
Fix $\eps>0$.
For each $\mu>0$ there exists a $\delta=\delta(\mu)>0$ with the
following property. At any point $(p,t)\in M\times\R$ where the
gradient is $\delta$-small
\begin{equation}\label{eq:grad-delta}
   \abs{\nabla^\eps F_H(p,t)}_\eps^2
   =\abs{\Babla{}F(p)+t\Babla{}H(p)}^2+\eps^{-2} H(p)^2
   \le\delta
\end{equation}
the value of the multiplier function lies in the $\mu$-interval
\begin{equation}\label{eq:grad-delta-2}
   \min f-\mu\le F_H(p,t)\le \max f+\mu.
\end{equation}
}

\noindent
To prove Step~1 suppose that a point $(p,t)$
satisfies~(\ref{eq:grad-delta}). Hence $H(p)^2\le\delta\eps^2$.
Since $H$ is locally proper around zero by
Hypothesis~\ref{hyp:uniform-Lag-bound-intro}, it follows
that for any open neighborhood $U$ of $\Sigma$
there exists a $\delta_U>0$ such that
the point $p$ lies in $U$ whenever $(p,t)$ satisfies~(\ref{eq:grad-delta})
for $\delta=\delta_U$.
Otherwise, there would exist a sequence $p_\nu\notin U$ with the
property that $H(p_\nu)\to 0$, as $\nu\to\infty$.
By local properness there is a subsequence $p_{\nu_k}$ which converges
to a point $p_\infty \in \Sigma=H^{-1}(0)$ contradicting
the assumption that none of the $p_\nu$ lies in the open neighborhood
$U$ of the compact set $\Sigma$.
\\
Given $\mu>0$, we choose $U(\mu)$:
Since zero is a regular value of $H$ and $\Sigma=H^{-1}(0)$ is compact
there exists an open neighborhood $U(\mu)$ of $\Sigma$ and constants
$c,C>0$ such that
\[
   c\le \inf_U \abs{\Babla{}H}
   ,\quad
   \sup_U\abs{\Babla{}F}\le C
   ,\quad
   \sup_U F\le\max f+\tfrac{\mu}{2}
   ,\quad
   \inf_U F\ge\min f-\tfrac{\mu}{2} .
\]
We choose $\delta=\delta(\mu)$:
Choose $\delta<\min\{\delta_{U(\mu)},C^2,\frac{\mu^2 c^2}{16\eps^2C^2}\}$.
From~(\ref{eq:grad-delta}) we deduce firstly that $p\in U(\mu)$
and secondly that, together with
$$
   \sqrt{\delta}\ge \abs{\Babla{}F(p)+t\Babla{}H(p)}
   \ge \abs{t\Babla{}H(p)}-\abs{\Babla{}F(p)} ,
$$
we obtain
\[
   \abs{t}
   \le\tfrac{\sqrt{\delta}+\abs{\Babla{}F(p)}}{\abs{\Babla{}H(p)}}
   \le\tfrac{\sqrt{\delta}+C}{c} .
\]
From this we get that
\begin{equation*}
\begin{split}
   F_H(p,t)
   =F(p)+t H(p)
   &\le \max f+\tfrac{\mu}{2}+\tfrac{\sqrt{\delta}+C}{c}\eps\sqrt{\delta}
   \\
   &\le\max f+\tfrac{\mu}{2}
   +\tfrac{2C}{c}\eps\frac{\mu c}{4\eps C}
   \\
   &=\max f+\mu .
\end{split}
\end{equation*}
This proves the upper bound in~(\ref{eq:grad-delta-2}).
The lower bound follows similarly.

\medskip
\noindent
\textbf{Step 2.}
{\it
If $(u,\tau)$ is a finite energy solution of the
$\eps$-equation.
Then
\[
   \min f\le F_H(u_s,\tau_s)\le\max f.
\]
}

\noindent
We prove the upper bound in Step~2, the lower bound follows analogously.
Assume by contradiction that there exists a time $s_0\in\R$
such that $F_H(u_{s_0},\tau_{s_0})>\max f$.
Let $\mu>0$ be determined by the difference
$2\mu:=F_H(u_{s_0},\tau_{s_0})-\max f$.
Let $\delta=\delta(\mu)$ be as in Step~1.
Since $(u,\tau)$ has finite energy there exists $s_1\le s_0$
such that $\abs{\nabla^\eps F_H(u_{s_1},\tau_{s_1})}_\eps^2\le\delta$.
Hence, by~(\ref{eq:grad-delta-2}), we have
\[
   F_H(u_{s_1},\tau_{s_1})
   \le \max f+\mu
   <\max f+2\mu
   =F_H(u_{s_0},\tau_{s_0}) .
\]
However, the action is decreasing along the negative gradient flow.
This contradiction proves the upper bound.
\end{proof}

\boldmath
%%%%%%%%%%%%%%%%%%%%%%%%%%%%%%%%%%%
%%%%%%%%%%%%%%%%%%%%%%%%%%%%%%%%%%%
%%%%%%% Section:  %%%%%%%%%%%%%%%%%%%%%
%%%%%%%%%%%%%%%%%%%%%%%%%%%%%%%%%%%
%%%%%%%%%%%%%%%%%%%%%%%%%%%%%%%%%%%
\section{Linearized operators}
\label{sec:linear-operators}
\unboldmath

\boldmath
%%%%%%%%%%%%%%%%%%%%%%%%%%%%%%%%%%%
%%%%%%%%%%%%%%%%%%%%%%%%%%%%%%%%%%%
\subsection{Base $\Sigma$}
%%%%%%%%%%%%%%%%%%%%%%%%%%%%%%%%%%%
%%%%%%%%%%%%%%%%%%%%%%%%%%%%%%%%%%%
\label{sec:base-linear-op}
\unboldmath

\boldmath
%%%%%%%%%%%%%%%%%%%%%%%%%%%%%%%%%%%
\subsubsection{Hilbert manifold $\Qq$ and moduli space $\Mm^0$}
%%%%%%%%%%%%%%%%%%%%%%%%%%%%%%%%%%%
%%%%%%%%%%%%%%%%%%%%%%%%%%%%%%%%%%%
\label{sec:base-moduli-space}
\unboldmath

Fix two critical points~$x^\mp$ of~$f\colon\Sigma\to\R$.
We denote the Hilbert manifold of all absolutely continuous paths
$q\colon\R\to\Sigma$ from $x^-$ to $x^+$
with square integrable derivative\footnote{
  by absolute continuity the derivative, notation $\p_s q$, exists at
  almost every instant $s\in\R$
  }
by
\[
   \Qq_{x^-,x^+}:=\{q\in W^{1,2}(\R,\Sigma)
   \mid\lim_{s\to\mp\infty}q(s)=x^\mp\}.
\]
We obtain charts for the \textbf{Hilbert manifold} $\Qq_{x^-,x^+}$
as follows. Let $q_T\colon\R\to\Sigma$ be a smooth map
with the property that there is a real $T>0$
such that $q_T(s)=x^-$ for $s\le -T$
and $q_T(s)=x^+$ for $s\ge T$.
Let $U_{q_T}$ be the set of vector fields
$\xi\in W^{1,2}(\R,q_T^*T\Sigma)$
such that at each instant of time $s$ the length of $\xi(s)$ is less
than the injectivity radius of $(\Sigma,g)$.
The exponential map of $(\Sigma,g)$ induces a parametrization,
still denoted $\exp$, of a neighborhood of $q_T$ in $\Qq_{x^-,x^+}$
as follows
\[
   \exp_{q_T}\colon U_{q_T}\to \Qq_{x^-,x^+},\quad
   \xi\mapsto \exp_{q_T}\xi,
   \qquad (\exp_{q_T}\xi)(s):=\exp_{q_T(s)}\xi(s).
\]
Consider the tangent bundle of $\Qq_{x^-,x^+}$, namely
\[
   T \Qq_{x^-,x^+}\to \Qq_{x^-,x^+},\quad
   \Ww_q:=T_q \Qq_{x^-,x^+}=W^{1,2}(\R,q^*T\Sigma),
\]
whose fiber $\Ww_q:=T_q \Qq_{x^-,x^+}$ over a path $q$ are the
$W^{1,2}$ vector fields along $q$ tangent to~$\Sigma$.
Now consider the vector bundle
\begin{equation}\label{eq:Ll-Qq}
   \Ll\to \Qq_{x^-,x^+},\quad
   \Ll_q:=L^2(\R,q^*T\Sigma),
\end{equation}
whose fiber $\Ll_q$ over a path $q$ consists of the $L^2$ vector fields
along $q$ tangent to~$\Sigma$. Corresponding inner products are
defined by
\begin{equation*}
\begin{split}
   \inner{\xi}{\eta}=\inner{\xi}{\eta}_2=\inner{\xi}{\eta}_{\Ll_q}
   :&=\int_{-\infty}^\infty \inner{\xi(s)}{\eta(s)}\, ds
   \\
   \inner{\xi}{\eta}_{1,2}=\inner{\xi}{\eta}_{\Ww_q}
   :&=\int_{-\infty}^\infty \INNER{\xi(s)}{\eta(s)}
   +\INNER{\Nabla{s}\xi(s)}{\Nabla{s}\eta(s)}\, ds
\end{split}
\end{equation*}
for compactly supported smooth vector fields
$\xi,\eta\in C^\infty_0(\R,q^*T\Sigma)$ .
A section of the vector bundle $\Ll\to \Qq_{x^-,x^+}$,
strictly speaking its principal part, is given~by
\begin{equation}\label{eq:Ff^0}
\begin{aligned}
   \Ff^0\colon \Qq_{x^-,x^+}&\to \Ll,
   \\
   q&\mapsto
   \p_sq+\Nabla{}f(q)
   \stackrel{(\ref{eq:DGF-ff-pre})}{=}
   \p_sq+\Babla{}F(q)+\chi(q)\Babla{} H(q) .
\end{aligned}
\end{equation}
The \textbf{base moduli space} is the zero set of
the section $\Ff^0$, in symbols
\begin{equation}\label{eq:findim-lin-Mm0}
   \Mm^0_{x^-,x^+}
%   :=(\Ff^0)^{-1}(0)
   =\left\{q\in\Qq_{x^-,x^+}
   \mid \p_sq+\Babla{}F(q)+\chi(q)\Babla{} H(q)=0\right\} .
\end{equation}

\begin{lemma}[Regularity and finite energy]\label{le:findim-reg0}
Any element $q\in \Mm^0_{x^-,x^+}$ is smooth
and, by~(\ref{eq:energy-0}), of finite energy $E^0(q)=f(x^-)-f(x^+)$.
\end{lemma}

\begin{proof}
Since by assumption $F,H$ are $C^\infty$ smooth and $q$ is continuous,
we see that the derivative
$\p_sq=-\Babla{}F(q)-\chi(q)\cdot\Babla{} H(q)$ is in fact continuous.
So $q\in C^1$. But then the right-hand side, hence $\p_s q$, is $C^1$,
so $q\in C^2$, and so on.
\end{proof}

%%%%%%%%%%%%%%%%%%%%%%%%%%%%%%%%%%%
\subsubsection{Linearization of base equation}
%%%%%%%%%%%%%%%%%%%%%%%%%%%%%%%%%%%
%%%%%%%%%%%%%%%%%%%%%%%%%%%%%%%%%%%
\label{sec:base-linearization}

Linearizing the section $\Ff^0$ at a zero $q\colon\R\to\Sigma$
we obtain the linear operator
\[
   D^0_q:=d\Ff^0(q)\colon
   W^{1,2}(\R,q^*T\Sigma)\to L^2(\R,q^*T\Sigma)
\]
which is of the form
\begin{equation}\label{eq:lin0-findim-lin}
\begin{aligned}
   D^0_q\xi
   &\stackrel{1}{=}
   \Nabla{s}\xi+\Nabla{\xi}\Nabla{}f|_q
%   {\color{white}\Bigl(}
   %
   {\color{gray}\;\stackrel{(\ref{eq:findim-gradf})}{=}
   \Nabla{s}\xi+\Nabla{\xi}\bigl(\Babla{}F|_q+\chi|_q\Babla{} H|_q\bigr)
   }
   \\
   &\stackrel{2}{=}
   \Babla{s}\xi
   +\Babla{\xi}   \bigl(\Babla{}F|_q+\chi|_q\Babla{} H|_q\bigr)
   {\color{gray}\;\stackrel{(\ref{eq:findim-gradf})}{=}
   \Babla{s}\xi+\Babla{\xi}\Nabla{}f(q)}
   \\
   &\stackrel{3}{=}
   \Babla{s}\xi
   +\Babla{\xi}\Babla{}F|_q
   +\chi|_q\Babla{\xi}\Babla{}H|_q
   +\left(d\chi|_q\xi\right)\cdot \Babla{}H|_q .
\end{aligned}
\end{equation}
For general elements $q\in\Mm^0_{x^-,x^+}$ we define $D^0_q$
by~(\ref{eq:lin0-findim-lin}).

Formula 1 arises when linearizing the base formulation
of the section, namely $\Ff^0(q)=\p_s q+\Nabla{}f(q)=0$.

Formula 2 arises when linearizing the ambient formulation of the section,
namely $\Ff^0(q)=\p_s q+\Babla{}F(q)+\chi(q)\cdot\Babla{} H(q)=0$.
Here the second equation in~(\ref{eq:DGF-ff})
imposes the condition that the domain of $D^0_q$ consists of vector fields
$\xi$ along $q$ that must be tangent to $\Sigma$. 
\\
Formula 2 in~(\ref{eq:lin0-findim-lin}) is a sum of vector
fields along $q\colon\SS^1\to\Sigma$ each of which a priori takes
values in $TM$. The sum, however, takes values in $T\Sigma$, indeed
\[
   D^0_q\xi
   =\Babla{s}\xi
   +\Babla{\xi}   \bigl(\Babla{}F|_q+\chi|_q\Babla{} H|_q\bigr)
   \stackrel{(\ref{eq:DGF-ff-pre})}{=}
   \Babla{s}\xi-\Babla{\xi}\p_sq
   =[\p_s q,\xi] ,
\]
but the commutator of vector fields tangent to $\Sigma$ is tangent
to $\Sigma$. The last identity is torsion freeness of the induced
connection $\widebar{\nabla}$, Lemma~\ref{le:findim-ind-conn}~(iv).
The second equation in formula 2 uses the Leibniz rule,
Lemma~\ref{le:findim-ind-conn}~(iii).

\smallskip
{\sc Symmetry} with respect to $g$ of the map
$\xi\mapsto\Nabla{\xi}\Nabla{}f|_q
=\Nabla{\xi}\bigl(\Babla{}F+\chi\Babla{} H\bigr)|_q$,\footnote{
  the map $\xi\mapsto\Babla{\xi}\bigl(\Babla{}F+\chi\Babla{}H\bigr)$
  takes values in $TM$ only, so it cannot be $g$-symmetric
  }
even in the case where $q\in\Sigma$ is a point and
$\xi,\eta\in T_q\Sigma$ vectors, is seen as follows
\begin{equation}\label{eq:adjoint0-findim-lin-2}
\begin{split}
   \inner{\eta}
      {\Babla{\xi}\bigl(\Babla{}F|_q+\chi|_q\Babla{} H|_q\bigr)}_{{\color{red} G}}
   &\stackrel{{\color{gray}\;\;\;\perp\;\;\;}}{=}
   \inner{\eta}{\Nabla{\xi}   \bigl(\Babla{}F|_q+\chi|_q\Babla{} H|_q\bigr)}_g\\
   &\stackrel{(\ref{eq:DGF-ff-pre})}{=}
   \inner{\eta}{\Nabla{\xi}\Nabla{}f|_q}_g\\
   &\stackrel{{\color{gray}\;\;\;\:3\:\;\;\;}}{=}
   \xi \inner{\eta}{\Nabla{}f|_q}_g
   -\inner{\Nabla{\xi}\eta}{\Nabla{}f|_q}_g\\
   &\stackrel{{\color{gray}\;\;\;\:4\:\;\;\;}}{=}
   \bigl(\xi\eta-\Nabla{\xi}\eta\bigr) f|_q\\
   &\stackrel{{\color{gray}\;\;\;\:5\:\;\;\;}}{=}
   \bigl(\eta\xi-\Nabla{\eta}\xi\bigr) f|_q .
\end{split}
\end{equation}
Here step 3 is by metric compatibility of the Levi-Civita connection,
step~4 holds since $\inner{\eta}{\Nabla{} f}
=df(\eta)=\eta f$, and step~5 is torsion freeness of $\nabla$.

\smallskip
{\sc Alternatively}
formula 2 arises from formula 1 by substituting both terms
$ \Nabla{s}\xi$ and $\Nabla{\xi}\Nabla{}f(q)$ by differences
according to~(\ref{eq:findim-decomp-II}):
\begin{equation}\label{eq:adjoint0-findim-lin-alt}
\begin{split}
   \Nabla{s}\xi+\Nabla{\xi}\Nabla{}f|_q
   &\stackrel{\;(\ref{eq:findim-decomp-II})\;}{=}
   \Babla{s}\xi
   -\mathrm{II}(\p_s q,\xi)
   \\
   &\stackrel{{\color{white} (1.15)}}{{\color{white}=}}
   +\Babla{\xi}   \bigl(\Babla{}F|_q+\chi|_q\Babla{} H|_q\bigr)
   -\mathrm{II}\bigl(\xi,   \bigl(\Babla{}F|_q+\chi|_q\Babla{} H|_q\bigr)  \bigr)
   \\
   &\stackrel{(\ref{eq:DGF-ff-pre})}{=}
   \Babla{s}\xi+\Babla{\xi}   \bigl(\Babla{}F|_q+\chi|_q\Babla{} H|_q\bigr).
\end{split}
\end{equation}
To see the second step substitute $\p_sq$, then cancel the two
$\mathrm{II}$-terms by symmetry. Such cancellation will not happen for
the adjoint operator in~(\ref{eq:def-adjoint0-findim-lin}) where
$\Nabla{s}\xi$ appears with the opposite sign, but the other term
keeps its sign.

\begin{lemma}\label{le:findim-ker-D^0}
If $\Ff^0(q)=0$, then the kernel of $D^0_q$ contains the element $\p_s q$.
\end{lemma}

\begin{proof}
Take the covariant derivative $\Nabla{s}$ of the vector field
$\p_sq+\Nabla{}f(q)=0$.
\end{proof}

\boldmath
%%%%%%%%%%%%%%%%%%%%%%%%%%%%%%%%%%%
\subsubsection{Trivialization of base section and derivative}
%%%%%%%%%%%%%%%%%%%%%%%%%%%%%%%%%%%
%%%%%%%%%%%%%%%%%%%%%%%%%%%%%%%%%%%
\label{sec:base-trivialization}
\unboldmath

Given a map $q\in \Qq_{x^-,x^+}$
and a vector field $\xi$ along
$q$, denote (pointwise for $s\in\R$) parallel transport in
$(\Sigma,g)$ along the geodesic $r\mapsto\exp_q(r\xi)$ by
\[
   \phi=\phi(q,\xi)\colon T_q\Sigma\to T_{\exp_q(\xi)}\Sigma.
\]
A trivialization of the base section $\Ff^0$ is given by the map
\[
   \Ff^0_q(\xi)
   :=\phi(q,\xi)^{-1}\Ff^0(\exp_q\xi)
   =\phi(q,\xi)^{-1}\left(\p_s(\exp_q(\xi))+\Nabla{}f(\exp_q(\xi))\right)
\]
defined on a sufficiently small neighborhood of the origin
(so $\exp$ is injective)
in the Hilbert space scale $h=(h_m)_{m\in\N_0}$ where
$h_m=W^{m+1,2}(\R,q^*T\Sigma)$; see~\cite{Hofer:2021a} or the
introduction~\cite{Weber:2022a}.
The derivative at the origin
\[
   d \Ff^0_q(0)\xi
   =\left.\tfrac{d}{dr}\right|_{r=0}\Ff^0_q(r\xi)
   =D^0_q\xi
\]
coincides with the linearization~(\ref{eq:lin0-findim-lin}) of the
section $\Ff^0$ at a zero; details are spelled out, e.g., in the proof of
Theorem~A.3.1 in~\cite{weber:1999a}.

\boldmath
%%%%%%%%%%%%%%%%%%%%%%%%%%%%%%%%%%%
\subsubsection{Formal adjoint}
%%%%%%%%%%%%%%%%%%%%%%%%%%%%%%%%%%%
%%%%%%%%%%%%%%%%%%%%%%%%%%%%%%%%%%%
\label{sec:base-formal=adjoint}
\unboldmath

For $q\in W^{1,2}$ the
\textbf{formal adjoint} $(D^0_q)^*\colon \Ww_q\to\Ll_q$
is determined by
\begin{equation}\label{eq:def-adjoint0-findim-lin}
   \inner{\eta}{D^0_q\xi}_2
   =\inner{(D^0_q)^*\eta}{\xi}_2
   ,\quad \forall \xi,\eta\in\Ww_q=W^{1,2}(\R,q^*T\Sigma),
\end{equation}
and consequently given by the first formula in what follows, namely
\begin{equation}\label{eq:adjoint0-findim-lin}
\begin{split}
   (D^0_q)^*\xi
   &\stackrel{1}{=}-\Nabla{s}\xi+\Nabla{\xi}\Nabla{}f|_q
   \\
   &\stackrel{2}{=}
   -\Babla{s}\xi
   +\mathrm{II}(\p_s q,\xi)
   \\
   &\quad
   +\Babla{\xi}\Nabla{}f|_q
   -\mathrm{II}\bigl(\xi,\Nabla{}f|_q\bigr)
   \\
   &\stackrel{3}{=}
   -\Babla{s}\xi+\Babla{\xi}\bigl(\Babla{}F|_q+\chi|_q\Babla{}H|_q\bigr)
   +2 \mathrm{II}\bigl(\xi, \p_s q\bigr)\\
   &{\color{gray}\,\,=
   -\Babla{s}\xi+\Babla{\xi}\Nabla{}f|_q
   +2 \mathrm{II}\bigl(\xi, \p_s q\bigr)}
\end{split}
\end{equation}
for every $\xi\in\Ww_q$ and where $\mathrm{II}$ is defined
by~(\ref{eq:findim-II}).
Step~3 holds for $0$-solutions $q$.
To see step~1 it suffices to work
in~(\ref{eq:def-adjoint0-findim-lin}) with
the dense subspace $C_0^\infty(\R,q^*T\Sigma)$. That $\Nabla{s}$
becomes $-\Nabla{s}$ follows by partial integration and compact
support. The map $\xi\mapsto\Nabla{\xi}\Nabla{}f$ is symmetric
by~(\ref{eq:adjoint0-findim-lin-2}) and thus it passes from $D^0_q$ to
the adjoint.
\\
To obtain step~2 we substituted each of the two terms tangential to
$\Sigma$, namely $\Nabla{s}\xi$ and $\Nabla{\xi}\Nabla{}f(q)$,
according to~(\ref{eq:findim-decomp-II}).
In step 3 we replaced $\Nabla{} f$ by $\Babla{}F+\chi\Babla{}H$
using~(\ref{eq:findim-gradf})
and in the $\mathrm{II}$-term by $-\p_sq$
using~(\ref{eq:DGF-ff-pre}) and symmetry of $\mathrm{II}$.

\boldmath
%%%%%%%%%%%%%%%%%%%%%%%%%%%%%%%%%%%
\subsubsection{Base linear estimate}
%%%%%%%%%%%%%%%%%%%%%%%%%%%%%%%%%%%
%%%%%%%%%%%%%%%%%%%%%%%%%%%%%%%%%%%
\label{sec:base-lin-est}
\unboldmath

\begin{proposition}\label{prop:findim-BaseEst-prop.D.2}
Let $q\in C^1(\R,\Sigma\times\R)$ such that
$\norm{\p_sq}_\infty<\infty$ is finite.
Then there is a constant $c_b=c_b(\norm{\p_sq}_\infty ,\norm{f}_{C^2(\Sigma)},
\norm{\mathrm{II}}_{L^\infty(\Sigma)})$ such that
\begin{equation}\label{eq:findim-BaseEst-prop.D.2}
   \Norm{\Nabla{s}\xi}+\Norm{\Babla{s}\xi}
   \le c_b\left(\Norm{D_q^0\xi}+\Norm{\xi}\right)
\end{equation}
for all vector fields $\xi\in W^{1,2}(\R,q^*TM)$.
The estimate also holds for $(D^0_q)^*$.
\end{proposition}

\begin{proof}
Expand the square 
$\Norm{D_q^0\xi}^2=\Norm{\Nabla{s}\xi+\Nabla{\xi}\Nabla{}f(q)}^2$
and use Cauchy-Schwarz and Young to get
$
   \norm{\Nabla{s}\xi}^2
   \le 2 \norm{D_q^0\xi}^2
   -2\norm{\nabla\Nabla{}f(q)}_\infty^2\norm{\xi}^2
$.
By~(\ref{eq:findim-decomp-II})
$\norm{\Nabla{s}\xi}^2=\norm{\Babla{s}\xi-\mathrm{II}(\p_s q,\xi)}^2$,
now expand the square. Same for $(D^0_q)^*$.
\end{proof}

\boldmath
%%%%%%%%%%%%%%%%%%%%%%%%%%%%%%%%%%%
\subsubsection{Fredholm property}\label{sec:Fredholm-0}
%%%%%%%%%%%%%%%%%%%%%%%%%%%%%%%%%%%
%%%%%%%%%%%%%%%%%%%%%%%%%%%%%%%%%%%
\unboldmath

Given a path $q\in\Qq_{x^-,x^+}$, it
makes sense to define operators $D^0_q,(D^0_q)^*\colon\Ww_q\to\Ll_q$
by the formulae~(\ref{eq:lin0-findim-lin})
and~(\ref{eq:adjoint0-findim-lin}), respectively.

A continuous linear operator $D$ between Banach spaces is called
\textbf{Fredholm} if kernel and cokernel are finite dimensional.
Finite codimension implies closed image.\footnote{
  Finite codimension of an arbitrary linear subspace $Y$
  does not, in general, imply closedness of $Y$ --
  for an image $Y=\im T$ of a \emph{continuous} operator $T$ it does.
  }
The difference $\dim\ker D-\dim\coker D$ is called the
\textbf{Fredholm index}.

\begin{proposition}
\label{prop:findim-0-Fredholm}
Let $q\in \Qq_{x^-,x^+}$ with $x^\mp\in\Crit f$ non-degenerate.
Then the following is true for the
operators $D^0_q,(D^0_q)^*\colon\Ww_q\to\Ll_q$
defined by~(\ref{eq:lin0-findim-lin})
and~(\ref{eq:adjoint0-findim-lin}).

\begin{itemize}\setlength\itemsep{0ex} 
\item[{\rm\bf (Exp.\,decay)}]
  Any kernel element $\xi=\xi(s)$ of $D^0_q$ or $(D^0_q)^*$
  is $C^\infty$ smooth and decays exponentially
  with all derivatives, as $s\to\mp\infty$.
  Hence $\norm{\xi},\norm{\xi}_\infty<\infty$.

\item[{\rm\bf (Fredholm)}]
  Both operators $D^0_q$ and $(D^0_q)^*$ are Fredholm and the
  Fredholm indices are the Morse index differences, namely
  \[
     \INDEX D^0_q
  %   :=\dim\ker D^0_q-\dim\coker D^0_q
     =\IND_f(x^-)-\IND_f(x^+)
     =-\INDEX (D^0_q)^* .
  \]
\end{itemize}
\end{proposition}

\begin{proof} [Proof of Proposition~\ref{prop:findim-0-Fredholm}]
That an operator $\frac{d}{ds}+A(s)$ with invertible
asymptotics $A(\mp\infty)$ has exponentially decaying kernel elements,
that it is Fredholm, and that the index is the asymptotics' Morse
index difference is well known, see e.g.~\cite{schwarz:1993a}.
In suitable trivializations both $D^0_q$ and $(D^0_q)^*$ are of such form.

That the formal adjoint is Fredholm whenever $D^0_q$ is (and of the same
Fredholm index times $-1$) follows immediately from the two vector
space equalities
\begin{equation}\label{eq:ker-coker}
   \ker (D^0_q)^*=\coker D^0_q
   \,\,{\color{gray}:=\left(\im D^0_q\right)^\perp}
   ,\qquad
   \coker (D^0_q)^*=\ker D^0_q .
\end{equation}
Vector space equality one.
`$\subset$'
Pick $\eta\in\ker (D^0_q)^*$.
By definition~(\ref{eq:def-adjoint0-findim-lin}) of $(D^0_q)^*$ we
have $\inner{\eta}{D^0_q\xi}=0$ for every $\xi\in W^{1,2}$.
But this means that $\eta\in\left(\im D^0_q\right)^\perp$.
\\
`$\supset$'
Pick $\eta\in \left(\im D^0_q\right)^\perp\subset L^2$.
Then
\begin{equation*}
\begin{split}
   0
   =
   \inner{\eta}{D^0_q\xi}
   &\stackrel{(\ref{eq:lin0-findim-lin})}{=}
   \inner{\eta}{\Nabla{s}\xi}
   +\inner{\eta}{\Nabla{\xi}\Nabla{}f(q)}\\
   &\stackrel{{\color{white}(1.32)}}{=}
   \inner{\eta}{\Nabla{s}\xi}
   +\inner{\Nabla{\eta}\Nabla{}f(q)}{\xi}
\end{split}
\end{equation*}
for every $\xi\in W^{1,2}$. But this is the definition of weak
derivative. So $\eta$ admits a weak derivative,
again denoted by $\Nabla{s}\eta$, and it is given by
\[
   \Nabla{s}\eta
   =\Nabla{\eta}\Nabla{}f(q)
   =D\Nabla{}f(q)\,\eta
   \in L^2.
\]
Indeed the last term lies in $L^2$, because $\eta$ does
and since $D\Nabla{}f(q)$ is of class $C^\infty$ (as $f$ is and by
Lemma~\ref{le:findim-reg0}) and decays exponentially with all
derivatives: indeed $\Nabla{}f(q)=-\p_s q\in\ker D^0_q$ is a kernel
element by Lemma~\ref{le:findim-ker-D^0}. Thus $\eta\in W^{1,2}$.
Now we can use the defining
identity~(\ref{eq:def-adjoint0-findim-lin}) of the adjoint to get that
\[
   0=\inner{(D^0_q)^*\eta}{\xi}
   =\inner{(D^0_q)^*\eta}{\xi}_g
\]
for every $\xi\in W^{1,2}(\R,q^*T\Sigma)$.
Thus $(D^0_q)^*\eta=0$ by nondegeneracy of $g$.

The proof of vector space equality two is analogous.
\end{proof}

\boldmath
%%%%%%%%%%%%%%%%%%%%%%%%%%%%%%%%%%%
%%%%%%%%%%%%%%%%%%%%%%%%%%%%%%%%%%%
\subsection{Ambience $M\times\R$}
%%%%%%%%%%%%%%%%%%%%%%%%%%%%%%%%%%%
%%%%%%%%%%%%%%%%%%%%%%%%%%%%%%%%%%%
\label{sec:ambient-lin-operator}
\unboldmath

\boldmath
%%%%%%%%%%%%%%%%%%%%%%%%%%%%%%%%%%%
\subsubsection{Hilbert manifold $\Zz$ and moduli space $\Mm^\eps$}
%%%%%%%%%%%%%%%%%%%%%%%%%%%%%%%%%%%
%%%%%%%%%%%%%%%%%%%%%%%%%%%%%%%%%%%
\label{sec: ambient-moduli-space}   %\label{sec:ff-lin-eps}
\unboldmath

Fix two critical points $x^\mp$ of $f=F|_\Sigma$.
So $(x^\mp,\chi(x^\mp))\in\Crit F_H$, by Lemma~\ref{le:Crit-F_H}.
We denote the Hilbert manifold of absolutely
continuous paths $z=(u,\tau)\colon\R\to M\times\R$ from $z^-$ to $z^+$
with square integrable derivative by
\[
   \Zz_{x^-,x^+}
   ,\qquad
   x^\mp\in\Crit f,\quad
   \tau^\mp:=\chi(x^\mp),\quad
   z^\mp:=(x^\mp,\tau^\mp)\in\Crit F_H.
\]
The tangent space at an element $z=(u,\tau)$ are the pairs $Z=(X,\ell)$
consisting of a $W^{1,2}$ vector field $X$ along $u$ and a $W^{1,2}$
function $\ell \colon\R\to\R$, in symbols
\[
   \Ww_{u,\tau}:=T_{(u,\tau)} \Zz_{x^-,x^+}
   =W^{1,2}(\R,u^*TM\oplus\R) .
\]
We use the same symbol $\Ll$ as in~(\ref{eq:Ll-Qq}) also for the
vector bundle
\[
   \Ll\to \Zz_{x^-,x^+} ,\qquad
   \Ll_{u,\tau}:=L^2(\R,u^*TM\oplus\R)
\]
whose fiber $\Ll_{u,\tau}$ over a path in $M\times\R$ are
the $L^2$ vector fields along $(u,\tau)$.

Given a parameter value $\eps>0$, a section of the vector bundle
$\Ll\to \Zz_{x^-,x^+}$ is defined by
\begin{equation}\label{eq:Ff-eps}
   \Ff^\eps(u,\tau)
   :=\p_s(u,\tau)+\nabla^\eps F_H(u,\tau)
   \stackrel{(\ref{eq:DGF-def-I})}{=}
   \begin{pmatrix}
      \p_su+\Babla{} F|_u+\tau\Babla{} H|_u\\
      \tau^\prime+\eps^{-2}H\circ u
   \end{pmatrix} .
\end{equation}
By definition the zero set is called the \textbf{ambient} or
\textbf{\boldmath$\eps$-moduli space}, notation
\[
   \Mm^\eps_{x^-,x^+}
   :=\{\Ff^\eps=0\}
   \subset \Zz_{x^-,x^+} .
\]

\boldmath
%%%%%%%%%%%%%%%%%%%%%%%%%%%%%%%%%%%
\subsubsection{Linearization of ambient equation}
%%%%%%%%%%%%%%%%%%%%%%%%%%%%%%%%%%%
%%%%%%%%%%%%%%%%%%%%%%%%%%%%%%%%%%%
\label{sec:ambient-linearization}   %\label{sec:ff-lin-eps}
\unboldmath

Linearizing the section $\Ff^\eps$ at a zero
$z=(u,\tau)\colon\R\to M\times\R$
provides the operator
\[
   D^\eps_{u,\tau}:=d\Ff^\eps(u,\tau)\colon
   \Ww_{u,\tau}\to\Ll_{u,\tau}
\]
given by
$
   D^\eps_{u,\tau }Z
   =\Nabla{s}^{\,\eps} Z+\Nabla{Z}^{\,\eps}\Nabla{}^{\,\eps} F_H(u,\tau)
$
or, equivalently, given by
\begin{equation}\label{eq:lin-eps-findim-lin}
\begin{aligned}
   D^\eps_{u,\tau}
   \begin{pmatrix} X\\\ell\end{pmatrix}
   &=\begin{pmatrix}
     \Babla{s} X
     +\Babla{X}\Babla{} F|_u+\tau\Babla{X}\Babla{} H|_u
     +\ell \Babla{} H|_u\\
     \ell^\prime+\eps^{-2}dH|_u X
   \end{pmatrix}.
\end{aligned}
\end{equation}
for $Z=(X,\ell)\in W^{1,2}(\R,u^*TM\oplus\R)$.
For $(u,\tau)\in\Zz_{x^-,x^+}$ define $D^\eps_{u,\tau}$
by~(\ref{eq:lin-eps-findim-lin}).

\boldmath
%%%%%%%%%%%%%%%%%%%%%%%%%%%%%%%%%%%
\subsubsection{Trivialization of ambient section and derivative}
%%%%%%%%%%%%%%%%%%%%%%%%%%%%%%%%%%%
%%%%%%%%%%%%%%%%%%%%%%%%%%%%%%%%%%%
\label{sec:ambient-trivialization}
\unboldmath

Pick a map $(u,\tau)\in \Zz_{x^-,x^+}$
and a vector field $(X,\ell)$ along it.
Denote parallel transport in $(M,G)$ along the geodesic
$
   r\mapsto \Exp_u(rX)
$
by
\[
   \Phi=\Phi(u,X)\colon T_uM\to T_\Gamma M
   ,\qquad
   \Gamma:=\Exp_u(X),
\]
pointwise for $s\in\R$.
A trivialization of the ambient section $\Ff^\eps$ is defined by
\begin{equation}\label{eq:eps-triv}
   \Ff^\eps_{u,\tau}(X,\ell)
   =\begin{pmatrix}
      \Phi^{-1}\left(
      \p_s\Gamma+\Babla{} F|_{\Gamma}
      +(\tau+\ell)\Babla{} H|_{\Gamma}\right)
   \\
      (\tau+\ell)^\prime+\eps^{-2}H|_{\Gamma}
   \end{pmatrix}
\end{equation}
for every vector field $(X,\ell)$ in a sufficiently small (so $\Exp$ is
injective) ball $\Oo$ about the origin of $W^{1,2}(\R,u^*TM\oplus\R)$.

To calculate the derivative at the origin
we utilize the facts about covariant
derivation and exponential maps collected
in~\cite[appendix~A]{weber:1999a} where the details of essentially
the same linearization are spelled out.
Abbreviate $\Phi_r:=\Phi(u,r  X)$ and $\Gamma_r:=\Exp_u(r X)$, then
$\left.\frac{d}{dr}\right|_0 \Gamma_r=  X$ and
\begin{equation}\label{eq:deriv-triv-origin}
\begin{split}
   &d \Ff^\eps_{u,\tau} (0,0)
   \begin{pmatrix} X\\\ell\end{pmatrix}
   :=\tfrac{d}{dr}\bigr|_0
   \Ff^\eps_{u,\tau} (r X,r\ell)
\\
   &\stackrel{{\color{gray} 1}}{=}
   \frac{d}{dr}\Bigr|_0
   \begin{pmatrix}
   \Phi_r^{-1}\left(\p_s(\Gamma_r)+\Babla{} F|_{\Gamma_r}\right)
   +(\tau+r\ell) \Phi_r^{-1}\Babla{} H|_{\Gamma_r}
   \\
     (\tau+r\ell)^\prime+\eps^{-2}H|_{\Gamma_r}
   \end{pmatrix}
\\
   &\stackrel{{\color{gray} 2}}{=}
   \begin{pmatrix}
   \left.\frac{d}{dr}\right|_0
   \left(\Phi_r^{-1}
   \left(\p_s(\Gamma_r)+\Babla{} F|_{\Gamma_r}\right)\right)
   +\ell\Babla{} H|_u
   +\tau\left.\frac{d}{dr}\right|_0
   \left(\Phi_r^{-1}\Babla{} H|_{\Gamma_r}\right)
   \\
   \left.\frac{d}{dr}\right|_0\left(
   (\tau+r\ell)^\prime+\eps^{-2}H|_{\Gamma_r}\right)
   \end{pmatrix}
\\
   &\stackrel{{\color{gray} 3}}{=}
   \begin{pmatrix}
      \Babla{s}  X
      +\Babla{  X}\Babla{}  F|_u
     +\tau\Babla{  X}\Babla{}  H|_u
     +\ell\Babla{} H|_u
    \\
      \ell^\prime
      +\eps^{-2}dH|_u  X
   \end{pmatrix}
   \stackrel{(\ref{eq:lin-eps-findim-lin})}{=}
   D^\eps_{u,\tau}
   \begin{pmatrix} X\\\ell\end{pmatrix} .
\end{split}
\end{equation}
Step~1 is by definition of $\Ff^\eps_{u,\tau}$ and linearity of parallel transport.
Step~2 uses the Levi-Civita connection $\widebar \nabla$ of $(M,G)$
and the Leibniz rule.
Step~3 holds by Theorem~A.3.1 in~\cite{weber:1999a},
more precisely by terms 1 and 3 in the proof.

\boldmath
%%%%%%%%%%%%%%%%%%%%%%%%%%%%%%%%%%%
\subsubsection{Formal adjoint and Fredholm property}
%%%%%%%%%%%%%%%%%%%%%%%%%%%%%%%%%%%
%%%%%%%%%%%%%%%%%%%%%%%%%%%%%%%%%%%
\label{sec:ambient-formal=adjoint}
\unboldmath

The \textbf{formal adjoint}
$(D^\eps_{u,\tau})^*\colon \Ww_{u,\tau}\to\Ll_{u,\tau}$ with
respect to the $(0,2,\eps)$ inner product associated to the
$(0,2,\eps)$ norm, defined in~(\ref{eq:findim-0-2-eps}) below,
is determined~by
\begin{equation}\label{eq:def-adjointeps-findim-lin}
   \INNER{\tilde Z}{D^\eps_{u,\tau}Z}_{0,2,\eps}
   =\INNER{(D^\eps_{u,\tau})^*\tilde Z}{Z}_{0,2,\eps}
   ,\quad \forall Z,\tilde Z\in\Ww_{u,\tau}.
\end{equation}
The formal $(0,2,\eps)$ adjoint is then given by the formula
\begin{equation}\label{eq:lin-epsadjoint-findim-lin}
\begin{split}
   (D^\eps_{u,\tau})^*
   \begin{pmatrix} X\\\ell\end{pmatrix}
   &=
   (D^\eps_z)^*Z\\
   &\stackrel{{\color{gray} 2}}{=}
   -\Nabla{s}^{\,\eps} Z+\Nabla{Z}^{\,\eps}\Nabla{}^{\,\eps} F_H|_z\\
   &\stackrel{{\color{gray} 3}}{=}
   \begin{pmatrix}
     -\Babla{s} X
     +\Babla{X}\Babla{} F|_u+\tau\Babla{X}\Babla{} H|_u
     +\underline{\ell \Babla{} H|_u}\\
     -\ell^\prime+\underline{\eps^{-2}dH|_u X}
   \end{pmatrix}
\end{split}
\end{equation}
for every $Z=(X,\ell)\in\Ww_{u,\tau}=W^{1,2}(\R,u^*TM\oplus\R)$.
Concerning identity~2, an $s$-derivative turns, by partial integration,
into minus an $s$-derivative and the operator
$Z\mapsto \Nabla{Z}^{\,\eps}\Nabla{}^{\,\eps} F_H$
is symmetric by an argument analoguous
to~(\ref{eq:adjoint0-findim-lin-2}).
Alternatively, analyze~(\ref{eq:def-adjointeps-findim-lin})
term by term. Apart from the two arguments we just gave,
the two underlined terms in~(\ref{eq:lin-epsadjoint-findim-lin})
satisfy the identity
\begin{equation}\label{eq:Hess-eps-symmetry}
   \inner{\tilde X}{\underline{\ell\Babla{}H|_u}}
   +\eps^2\inner{\tilde \ell}{\underline{\eps^{-2}dH|_u X}}
   =\inner{\tilde \ell\Babla{}H|_u}{X}
   +\eps^2\inner{\eps^{-2}dH|_u\tilde X}{\ell}.
\end{equation}
Mind the tildes.
To see the equality write out the inner products as integrals.

\begin{proposition}[Fredholm property]
\label{prop:findim-eps-Fredholm}
For a path $z=(u,\tau)\in\Zz^\eps_{x^-,x^+}$ with non-degenerate
boundary conditions $x^\mp\in\Crit f$ the following is true. Both
operators $D^\eps_{u,\tau},(D^\eps_{u,\tau})^*\colon\Ww_{u,\tau}\to\Ll_{u,\tau}$
are Fredholm and
\begin{equation}\label{eq:ker-coker-eps}
   \ker (D^\eps_{u,\tau})^*=\coker D^\eps_{u,\tau}
   \,\,{\color{gray}:=\left(\im D^\eps_{u,\tau}\right)^\perp}
   ,\qquad
   \coker (D^\eps_{u,\tau})^*=\ker D^\eps_{u,\tau} .
\end{equation}
The Fredholm and Morse indices are related by
\begin{equation}\label{eq:F-index-eps}
   \INDEX D^\eps_{u,\tau}
   =\IND_f(x^-)-\IND_f(x^+)
   =-\INDEX (D^\eps_{u,\tau})^* .
\end{equation}
\end{proposition}

\begin{proof}
Analogous to
Proposition~\ref{prop:findim-0-Fredholm};
use in addition Lemma~\ref{le:findim-Morse-indices}.
\end{proof}

\boldmath
%%%%%%%%%%%%%%%%%%%%%%%%%%%%%%%%%%%
\subsubsection{Suitable $\eps$-dependent norms}
%%%%%%%%%%%%%%%%%%%%%%%%%%%%%%%%%%%
\unboldmath

To obtain uniform estimates for the right inverse with constants
independent of $\eps>0$ small, we must work
with $\eps$-dependent norms which are suggested
on $L^2$ by the energy identity~(\ref{eq:energy-eps})
and on $W^{1,2}$ by the fundamental estimate~(\ref{eq:amblinest}).
For compactly supported smooth vector fields
$Z=(X,\ell)$ along $(u,\tau)$ define
\begin{equation}\label{eq:findim-0-2-eps}
\begin{split}
   \norm{Z}_{0,2,\eps}
   :&=\left(\norm{X}^2+\eps^2\norm{\ell }^2\right)^{1/2}\\
   &\le\norm{X}+\eps\norm{\ell }
\\
   \norm{Z}_{0,\infty,\eps}
   :&=\norm{X}_\infty+\eps\norm{\ell}_\infty
\\
   \norm{Z}_{1,2,\eps}
   :&=\left(\norm{X}^2+\eps^2\norm{\ell }^2
   +\eps^2\norm{\Nabla{s} X}^2+\eps^4\norm{\ell ^\prime}^2\right)^{1/2}\\
   &\le\norm{X}+\eps\norm{\ell }
   +\eps\norm{\Babla{s} X}+\eps^2\norm{\ell ^\prime}
   \stackrel{(\ref{eq:sum-square})}{\le}
   2^{\frac{3}{2}}\norm{Z}_{1,2,\eps} .
\end{split}
\end{equation}

\begin{lemma}\label{le:0-infty-eps}
Let $(u,\tau)\in W^{1,2}(\R,M\times\R)$ and $\eps>0$. Then
there is the estimate
\begin{equation}\label{eq:cor:infty}
   \eps^{1/2}\norm{Z}_{0,\infty,\eps}
   \le 3\norm{Z}_{1,2,\eps}
\end{equation}
for every $Z=(X,\ell)\in W^{1,2}(\R,u^*TM\oplus\R)$.
\end{lemma}

\begin{proof}
For $v\colon\R\to\R$ of class $C^1$ and compactly supported
it holds that
\[
   \abs{v(s)}\cdot v(s)
   =\int_{-\infty}^s\underbrace{\tfrac{d}{d\sigma}
   \left(\abs{v(\sigma)}\cdot v(\sigma)\right)}_{=2 \abs{v(\sigma)} v^\prime(\sigma)} d\sigma
   =2\inner{\abs{v}}{v^\prime}_{L^2}
   \le 2\norm{v}\cdot\norm{v^\prime}
   \le \norm{v}_{1,2}
\]
where the last step is by Young $ab\le a^2/2+b^2/2$ and
$\norm{v}_{1,2}^2:=\norm{v}^2+\norm{v^\prime}^2$. So
\begin{equation}\label{eq:Linfty-est}
   \norm{v}_\infty\le \norm{v}_{1,2} .
\end{equation}
Use that $C^1_0$ is dense in $W^{1,2}$ on the domain $\R$,
then~(\ref{eq:Linfty-est})
provides the Cauchy property of the approximating
sequence, so~(\ref{eq:Linfty-est}) remains true for $v\in W^{1,2}$.

Now we rescale. For $\beta\in\R$ and $\eps>0$ define
$v_\beta\colon \R\to\R$ by $v_\beta(s):=v(\eps^{2\beta} s)$.
Note that $\norm{v_\beta}_\infty=\norm{v}_\infty$, but the $L^2$-norms
behave as follows
\begin{equation*}
\begin{split}
   \norm{v_\beta}^2
   &=\int_{-\infty}^\infty v(\underbrace{\eps^{2\beta} s}_{\sigma(s)})^2\, ds
   =\eps^{-2\beta}\int_{-\infty}^\infty v(\sigma)^2\, d\sigma
   =\eps^{-2\beta}\norm{v}^2 ,
   \\
   \norm{v_\beta^\prime}^2
   &=\int_{-\infty}^\infty
   (v^\prime(\underbrace{\eps^{2\beta} s}_{\sigma(s)})\eps^{2\beta})^2\, ds
   =\eps^{-2\beta}\eps^{4\beta}\int_{-\infty}^\infty (v^\prime(\sigma))^2\, d\sigma
   =\eps^{2\beta}\norm{v^\prime}^2 .
\end{split}
\end{equation*}
Now square~(\ref{eq:Linfty-est}) to $v_\beta$ to see that
\begin{equation*}
\begin{split}
   \norm{v}_\infty^2
   =\norm{v_\beta}_\infty^2
   \stackrel{(\ref{eq:Linfty-est})}{\le} \norm{v_\beta}^2+\norm{v_\beta^\prime}^2
   &\le\left(\eps^{-\beta}\norm{v}\right)^2
   +\left(\eps^\beta\norm{v^\prime}\right)^2
   \\
   &\le\left(\eps^{-\beta}\norm{v}+\eps^\beta\norm{v^\prime}\right)^2
\end{split}
\end{equation*}
whenever $\beta\in\R$ and $\eps>0$.
Take the square root, then multiply by $\eps^\beta$ to obtain
\begin{equation}\label{eq:Linfty-est-beta}
   \eps^\beta\norm{v}_\infty
   \le\norm{v}+\eps^{2\beta}\norm{v^\prime} .
\end{equation}
With $\beta=\frac12$ apply~(\ref{eq:Linfty-est-beta}) for
$v(s)=\abs{X(s)}=\abs{X(s)}_G$ and $v(s)=\ell(s)$ to obtain
\[
   \sqrt{\eps}\norm{Z}_{0,\infty,\eps}
   \stackrel{(\ref{eq:findim-0-2-eps})}{=}
   \sqrt{\eps}\norm{X}_\infty+\sqrt{\eps}\eps\norm{\ell}_\infty
   \stackrel{(\ref{eq:Linfty-est-beta})}{\le} \norm{X}+\eps\norm{X^\prime}
   +\eps\norm{\ell}+\eps^2\norm{\ell^\prime} .
\]
Now the square root of the inequality for non-negative reals
\begin{equation}\label{eq:sum-square}
   (a_1+\dots+a_k)^2\le 2^{k-1}(a_1^2+\dots+a_k^2)
\end{equation}
in case $k=4$ completes the proof of Lemma~\ref{le:0-infty-eps}.
\end{proof}

\boldmath
%%%%%%%%%%%%%%%%%%%%%%%%%%%%%%%%%%%
\subsubsection{Ambient linear estimate along maps $i(q)$}
\label{sec:amblinest}
%%%%%%%%%%%%%%%%%%%%%%%%%%%%%%%%%%%
\unboldmath

The most important uniform linear estimates in an adiabatic limit
are the fundamental estimate, in our case the ambient linear estimate,
Theorem~\ref{thm:amblinest} below,\footnote{
  In PDE cases, such as~\cite{salamon:2006a},
  the ambient linear estimate is often much weaker than
  in our ODE case, so it must be improved to what we
  refer to as the fundamental estimate.
  }
and the key estimate, Theorem~\ref{thm:KeyEst-thm.3.3}.

In the following we consider maps $q$ that take values in the compact
regular hypersurface $\Sigma$. Thus we can work directly with the
(positive) minimal length $m_H:=\min_\Sigma\abs{\Babla{}H}>0$ along
$\Sigma$, instead of invoking part~(ii) of Theorem~\ref{thm:apriori-intro}
which only works for small $\eps>0$.
In fact, Section~\ref{sec:amblinest} can be generalized to
maps $(u,\tau)\in C^1(\R,M\times \R)$ with
$\norm{\p_su}_\infty+\norm{\tau}_\infty< c_w$ and for $\eps>0$ small.

\begin{theorem}\label{thm:amblinest}
Let $q\in C^1(\R,\Sigma)$. Let $\norm{\p_sq}_\infty< c_w$
be bounded by a constant. Then there is a constant
$c_a=c_a(m_H,c_w,\norm{H}_{C^2(\Sigma)},\norm{F}_{C^2(\Sigma)})>0$
such that
\begin{equation}\label{eq:amblinest}
   \eps^{-1}\norm{dH|_qX}
   +\norm{\ell }
   +\norm{\Babla{s} X}
   +\eps\norm{\ell ^\prime}
   \le c_a\left(
   \norm{D^\eps_q Z}_{0,2,\eps}
   +\norm{X}
   \right)
\end{equation}
for all $\eps>0$ and $Z=(X,\ell)\in W^{1,2}(\R,q^*TM\oplus\R)$.
The estimate continues to hold for $(D^\eps_q)^*$.
The constants $c_a$ is invariant under $s$-shifts of $q$.
\end{theorem}

\begin{proof}
Fix $Z=(X,\ell)$ in the dense subset $C^\infty_0 (\R,q^*TM\oplus\R)$.
Take the square
\begin{equation*}
\begin{split}
   \norm{D^\eps_q Z}_{0,2,\eps}^2
   &=\norm{\Babla{s} X+\Babla{X}\bigl(\Babla{} F|_q
   +\chi|_q\Babla{} H|_q\bigr)+\ell  \Babla{} H|_q}_{L^2_q}^2\\
   &\quad+\eps^2\norm{\ell ^\prime+\eps^{-2}dH|_q X}_{L^2(\R)}^2.
\end{split}
\end{equation*}
Consider the first term in the sum. Expand the square to get
\begin{equation*}
\begin{split}
   &\norm{\Babla{s} X+\Babla{X}\bigl(\Babla{} F|_q
   +\tau\Babla{} H|_q\bigr)+\ell  \Babla{} H|_q}_{L^2_q}^2
\\
   &=\norm{\Babla{s} X}_{L^2_q}^2
   +\norm{\Babla{X}\bigl(\Babla{} F|_q+\tau\Babla{} H|_q\bigr)}_{L^2_q}^2
   +\norm{\ell  \Babla{} H|_q}_{L^2_q}^2
   \\
   &\quad
   +2\INNER{\sqrt{2}\Babla{X}\bigl(\Babla{} F|_q+\tau\Babla{} H|_q\bigr)}
   {\tfrac{1}{\sqrt{2}} \ell  \Babla{} H|_q}_{L^2_q}\\
   &\quad
   +2\INNER{\tfrac{1}{\sqrt{2}}\Babla{s} X}
   {\sqrt{2}\Babla{X}\bigl(\Babla{} F|_q+\tau\Babla{} H|_q\bigr)}_{L^2_q}
   +2\INNER{\Babla{s} X}{\ell  \Babla{} H|_q}_{L^2_q}
\\
   &\ge\tfrac{1}{2}\norm{\Babla{s} X}_{L^2_q}^2
   +\tfrac{1}{2}\norm{\ell  \Babla{} H|_q}_{L^2_q}^2
   -3\norm{\Babla{X}\bigl(\Babla{} F|_q+\tau\Babla{}
     H|_q\bigr)}_{L^2_q}^2
%   \\
%   &\quad
   +2\INNER{\Babla{s} X}{\ell  \Babla{} H|_q}_{L^2_q}
\\
   &\ge\tfrac{1}{2}\norm{\Babla{s} X}_{L^2_q}^2
   +\tfrac{{m_H}^2}{2}\norm{\ell }_{L^2(\R)}^2
   -3\left(\norm{F}_{C^2(\Sigma)}
   +\norm{\tau}_\infty\norm{H}_{C^2(\Sigma)}\right)
   \norm{X}_{L^2_q}^2\\
   &\quad+\underline{2\INNER{\Babla{s} X}{\ell  \Babla{} H|_q}_{L^2_q}} .
\end{split}
\end{equation*}
Here we also used Cauchy-Schwarz followed by Young's
inequality, then we pulled out the $L^\infty$ norms.
Next consider the second term in the sum. Expand
the square and integrate by parts to get
\begin{equation*}
\begin{split}
   &\eps^2\norm{\ell ^\prime+\eps^{-2}dH|_q X}_{L^2(\R)}^2
   \\
   &=\eps^2\norm{\ell ^\prime}_{L^2(\R)}^2+\eps^{-2}\norm{dH|_qX}_{L^2(\R)}^2
   +2\INNER{\ell ^\prime}{\inner{\Babla{}H|_q}{X}_G}_{L^2(\R)}
   \\
   &=\eps^2\norm{\ell ^\prime}_{L^2(\R)}^2+\eps^{-2}\norm{dH|_qX}^2_{L^2(\R)}\\
   &\quad
   -\INNER{\tfrac{m_H}{\sqrt{2}}\ell}{2\tfrac{\sqrt{2}}{m_H}\inner{\Babla{s}\Babla{}H|_q}{X}_G}_{L^2(\R)}
   -2\INNER{\ell }{\inner{\Babla{}H|_q}{\Babla{s} X}_G}_{L^2(\R)}
   \\
   &\ge
   \eps^2\norm{\ell ^\prime}_{L^2(\R)}^2+\eps^{-2}\norm{dH|_qX}^2_{L^2(\R)}
   -\tfrac{{m_H}^2}{4}\norm{\ell }_{L^2(\R)}^2
   \\
   &\quad
   -\tfrac{4 \norm{\p_su}_\infty^2 \norm{H}^2_{C^2(\Sigma)}}{{m_H}^2}\,
   \norm{X}^2_{L^2_q}
   -\underline{2\INNER{\ell }{\inner{\Babla{}H|_q}{\Babla{s} X}_G}_{L^2(\R)}}.
\end{split}
\end{equation*}
To obtain the inequality we used Cauchy-Schwarz followed by Young's
inequality, then we pulled out the $L^\infty$ norms.
Adding the two estimates the underlined terms cancel
and we obtain the estimate~(\ref{eq:amblinest}).

The estimate for the formal adjoint follows exactly the same way.
Here the derivative terms show up with a minus sign. The
underlined terms now show up with a factor $-1$ and so they still cancel.
\end{proof}

\begin{remark}\label{rem:amblinest}
Under the hypotheses of Theorem~\ref{thm:amblinest}
there are $C,\eps_0>0$ with
\begin{equation*}%\label{eq:amblinest-rem}
   \eps^{-1}\norm{dH|_qX}
   +\norm{\ell }
   +\norm{\Babla{s} X}
   +\eps\norm{\ell ^\prime}\\
   \le C\left(
   \norm{D^\eps_q Z}_{0,2,\eps}+\norm{\tan X}
   \right)
\end{equation*}
for every $Z=(X,\ell)\in W^{1,2}(\R,q^*TM\oplus\R)$ and
whenever $\eps\in(0,\eps_0]$. Similarly for $(D^\eps_q)^*$
and the constants $C,\eps_0$ are invariant under $s$-shifts of $q$.

To see this decompose $X=\tan X+\nor\, X$ on the right
of~(\ref{eq:amblinest}) to obtain
\[
   \Norm{X}\le\Norm{\tan X}+\Norm{\nor\, X}
  \stackrel{(\ref{eq:findim-gradf})}{\le}
   \Norm{\tan X}+\tfrac{\eps}{m_H}\eps^{-1}\Norm{dH|_q X}.
\]
Incorporate the last summand into the left-hand side
of~(\ref{eq:amblinest}) for small $\eps$.
\end{remark}

\begin{corollary}\label{cor:amblinest}
Let $q\in C^1(\R,\Sigma)$. Let $\norm{\p_sq}_\infty< c_w$
be bounded by a constant. 
Let $\eps_0$ be the constant in Remark~\ref{rem:amblinest}.
Then there is a constant $C_a>0$ with
\begin{equation}\label{eq:cor-amblinest}
   \tfrac{1}{3}\eps^{1/2}\norm{Z}_{0,\infty,\eps}
   \le\Norm{Z}_{1,2,\eps}
   \le\eps C_a\norm{D^\eps_q Z}_{0,2,\eps}+\norm{\tan X}
\end{equation}
for all $\eps\in(0,\eps_0]$ and $Z=(X,\ell)\in W^{1,2}(\R,q^*TM\oplus\R)$.
The estimate also holds for $(D^\eps_q)^*$.
The constants $C_a,\eps_0$ are invariant under $s$-shifts of $q$.
\end{corollary}

\begin{proof}
By definition~(\ref{eq:findim-0-2-eps}) of the $(1,2,\eps)$-norm,
by writing $X=\tan X+\nor\, X$, and since
$\Norm{\nor\, X}\le\frac{1}{m_H}\Norm{dH|_q X}$
by~(\ref{eq:findim-gradf}), we get that
\begin{equation*}
\begin{split}
   \norm{Z}_{1,2,\eps}
   &\le\norm{\tan X}+\norm{\nor\, X}+\eps\norm{\ell }
   +\eps\norm{\Babla{s} X}+\eps^2\norm{\ell ^\prime}
   \\
   &\le\norm{\tan X}
   +\eps\cdot \tfrac{\max\{m_H,1\}}{m_H}
   \left(
      \eps^{-1}\norm{dH|_qX}+\norm{\ell }+\norm{\Babla{s} X}
      +\eps\norm{\ell^\prime}
   \right) .
\end{split}
\end{equation*}
Now apply Remark~\ref{rem:amblinest}.
Inequality~(\ref{eq:cor:infty}) concludes the proof.
\end{proof}

\boldmath
%%%%%%%%%%%%%%%%%%%%%%%%%%%%%%%%%%%
%%%%%%%%%%%%%%%%%%%%%%%%%%%%%%%%%%%
%%%%%%% Section:  %%%%%%%%%%%%%%%%%%%%%
%%%%%%%%%%%%%%%%%%%%%%%%%%%%%%%%%%%
%%%%%%%%%%%%%%%%%%%%%%%%%%%%%%%%%%%
\section{Linear estimates}
\label{sec:linear-estimates}
\unboldmath

Throughout Section~\ref{sec:linear-estimates} we study linearized operators
along maps $q$ which take values in the compact hypersurface $\Sigma$.
Thus we can work with the constant
\[
   m_H:=\min_\Sigma \Abs{\Babla{} H}>0,
\]
see~(\ref {eq:findim-gradf}), which does not impose restrictions on
the values of $\eps>0$, in sharp contrast to the constant  $c_\kappa$
that appears in part~(ii), see~\cite{Frauenfelder:2022g}, of the
a priori Theorem~\ref{thm:apriori-intro}
requiring a small parameter interval $(0,\eps_\kappa]$.

\boldmath
%%%%%%%%%%%%%%%%%%%%%%%%%%%%%%%%%%%
%%%%%%%%%%%%%%%%%%%%%%%%%%%%%%%%%%%
\subsection{Canonical embedding and orthogonal projection}
%%%%%%%%%%%%%%%%%%%%%%%%%%%%%%%%%%%
\label{sec:ff-comp-eqs}
\unboldmath

The elements $q$ of the Hilbert manifold $\Qq_{x^-,x^+}$ are paths
that take values in the regular level set $\Sigma=H^{-1}(0)$
along which the map $\chi$ defined by~(\ref{eq:chi}) is well defined. 
By~(\ref{eq:can-emb}) and~(\ref{eq:finfim-i}) there is the
\textbf{canonical embedding}
\[
   i\colon \Qq_{x^-,x^+}\to \Zz_{x^-,x^+}
   ,\quad
   q\mapsto(q,\chi(q)) ,
\]
which is useful to compare the base solutions $q$
and the $\eps$-solutions $(u,\tau)$. At a path $q\in \Qq_{x^-,x^+}$
the linearization of the natural embedding is given by
\begin{equation*}
\begin{split}
   {\color{gray}T_q \Qq_{x^-,x^+}}
   &{\color{gray}\,\to T_{i(q)}i(\Qq_{x^-,x^+})
   \qquad\;\subset T_{i(q)}\Zz_{x^-,x^+}}
   \\
   I_q:=di|_q\colon
   W^{1,2}(\R,q^*T\Sigma)
%   \overbrace{W^{1,2}(\R,q^*T\Sigma)}^{T_q \Qq_{x^-,x^+}}
   &\to W^{1,2}(\R,q^*T{\color{red}\Sigma}\oplus\R)
%   \overbrace{W^{1,2}(\R,q^*T{\color{red}\Sigma}\oplus\R)}
%      ^{T_{i(q)}i(\Qq_{x^-,x^+})}
   \subset 
%\overbrace{W^{1,2}(\R,q^*T{\color{cyan}M}\oplus\R)}^{T_{i(q)}\Zz_{x^-,x^+}} .
   W^{1,2}(\R,q^*T{\color{cyan}M}\oplus\R)
   \\
   \xi&\mapsto\left(\xi,d\chi|_q\xi\right) .
\end{split}
\end{equation*}

\begin{definition}[Orthogonal projection]\label{def:orth-proj}
At $q\in \Qq_{x^-,x^+}$ the $(0,2,\eps)$-ortho\-gonal projection
onto the image of the linearized embedding $I_q$
is the composition
\[
   \Pi^\perp_\eps= I_q\pi^\perp_\eps\colon T_{i(q)} \Zz_{x^-,x^+}
   =W^{1,2}(\R,q^*TM\oplus\R)\to W^{1,2}(\R,q^*TM\oplus\R)
\]
whose value on
$Z=(X,\ell)\in W^{1,2}(\R,q^*TM\oplus\R)$ is determined by
\begin{equation}\label{eq:pi_eps-def-ff}
   \INNER{Z- I_q\pi^\perp_\eps Z}{ I_q\xi}_{0,2,\eps}
   =0
\end{equation}
for every vector field $\xi\in T_q\Qq_{x^-,x^+}=W^{1,2}(\R,q^*T\Sigma)$.
\end{definition}

\begin{lemma}
a)~The linear map $\pi^\perp_\eps\colon
T_{i(q)}\Zz_{x^-,x^+}\to T_q\Qq_{x^-,x^+}$ is given by
\begin{equation}\label{eq:pi_eps-ff}
   \pi^\perp_\eps (X,\ell)
   =\left(\1+\eps^2\mu^2\, P\right)^{-1}
   \left(\tan X+\eps^2 \ell \Nabla{}\chi|_q\right)
   ,\qquad
   \mu:=\Abs{\Nabla{}\chi(q)} ,
\end{equation}
for every pair $Z=(X,\ell)\in W^{1,2}(\R,q^*TM\oplus\R)$. Here
$\Nabla{} \chi$ is the gradient in $(\Sigma,g)$
and $P$ is the pointwise orthogonal projection\,\footnote{
  if $\nabla\chi(q(s)=0$ vanishes at some $s$,
  then $\mu_{q(s)}^2 P_{q(s)}=0$ is the zero map at that $s$
  }
\begin{equation}\label{eq:findim-P}
\begin{split}
   P=P_q\colon T_q\Sigma
   &\to V_q:=\R\Nabla{}\chi|_q\subset T_q\Sigma
   \\
   \xi
   &\mapsto \frac{\INNER{\Nabla{}\chi|_q}{\xi}}{\mu^2}\,
   \Nabla{}\chi|_q,
\end{split}
\end{equation}
where $q$ actually abbreviates $q(s)$ for $s\in\R$. By compactness of
$\Sigma$ the constant
$\mu_\infty:=\max\{1,\Norm{\Nabla{}\chi}_{L^\infty(\Sigma)}\}$ is finite.
b)~It holds that $\pi^\perp_\eps I_q=\1$, so
$(\Pi^\perp_\eps)^2=\Pi^\perp_\eps$.
\end{lemma}

\begin{proof}
a) Let $\xi_0:=\pi^\perp_\eps(X,\ell)$.
By~(\ref{eq:pi_eps-def-ff}) the vector field $\xi_0$
lives in $T\Sigma$ and
\begin{equation*}
\begin{split}
   0
   &=\INNER{X-\xi_0}{\xi}_G+\eps^2\left(\ell -d\chi|_q\xi_0\right)d\chi|_q\xi
   \\
   &=\INNER{\tan X-\xi_0
   +\eps^2\left(\ell -\INNER{\nabla{}\chi}{\xi_0}\right)
   \nabla{}\chi}{\xi}
\end{split}
\end{equation*}
pointwise at $s\in\R$ and for every $\xi\in T_q \Qq_{x^-,x^+}$.
We wrote $X=\tan X+\nor\, X$, we used that $\xi\perp\nor\, X$, and we
replaced the metric $G$ by $g$.
By non-degeneracy
\begin{equation*}%\label{eq:proj-2-ff}
\begin{split}
   \tan X+\eps^2\ell \Nabla{}\chi
   &=\xi_0
   +\eps^2\INNER{\Nabla{}\chi}{\xi_0}
   \Nabla{}\chi
   =\xi_0+\eps^2\mu^2 P\xi_0
\end{split}
\end{equation*}
and so $\pi^\perp_\eps (X,\ell)=\xi_0
=\left(\1+\eps^2\INNER{\Nabla{}\chi}{\1}
   \Nabla{}\chi\right)^{-1}
   \left(\tan X+\eps^2\ell \Nabla{}\chi\right)$.

b)~Apply the isomorphism in~(\ref{no-eq:findim-B}) to the
desired identity $\xi=\pi^\perp_\eps I_q\xi$ to get equivalently
$\xi+\eps^2\mu^2 P\xi=\xi+\eps^2 (d\chi|_q\xi)\Nabla{}\chi$
which is true by definition of $P$.
\end{proof}

%\newpage
\boldmath
%%%%%%%%%%%%%%%%%%%%%%%%%%%%%%%%%%%
%%%%%%%%%%%%%%%%%%%%%%%%%%%%%%%%%%%
\subsubsection{Ansatz for a suitable projection}
%%%%%%%%%%%%%%%%%%%%%%%%%%%%%%%%%%%
\label{no-sec:ff-est-can-proj}
\unboldmath

In previous adiabatic limits~\cite{dostoglou:1994a,gaio:1999a,weber:1999a, Gaio:2005a,salamon:2006a}
-- where the spatial part involves differential
equations, so the flow equation is a PDE and not just an ODE as in
the present article --
it was crucial for the functioning of the Newton iteration
not to choose the operator $\pi^\perp_\eps$ associated to
the orthogonal projection $\Pi^\perp_\eps=I_q\pi^\perp_\eps$.
There the natural orthogonal choice did produce an abundance of powers of
$\eps$ in one component, but a lack in the other one.
To balance this out one can introduce parameters $\alpha,\beta>0$
and make the Ansatz
\begin{equation}\label{no-eq:pi_eps-ff-Ansatz}
   \pi_\eps (X,\ell)
   :=\left(\1+\eps^\alpha\mu^2 P\right)^{-1}
   \left(\tan X+\eps^\beta \ell \Nabla{}\chi|_q\right) .
\end{equation}
It seems a common principle that the epsilon power $\beta=2$
that shows up in the \emph{orthogonal} projection~(\ref{eq:pi_eps-ff})
and also in the $\eps$-equation~(\ref{eq:DGF-def-I}), is the right
value of $\beta$.
Usually the value $\beta=2$ is suggested, too, when
comparing the linear operators $D^0_q$ and $D^\eps_q$,
see the proof of Proposition~\ref{le:4.1.3}.
In the present article the choice $\beta=2$ also
optimizes the Uniqueness Theorem~\ref{thm:uniqueness-findim},
see~(\ref{eq:beta=2}).
For $\alpha=1$ the operator comparison estimate~(\ref{eq:4.1.3})
has a nicely equilibrated right hand side, but the orthogonal
choice $\alpha=2$ works as well.

\begin{lemma}[Le. 4.1.5]\label{no-le:4.1.5-findim}
Let $q\in W^{1,2}(\R,\Sigma)$ and $\alpha\in\R$.
Then
\begin{equation}\label{no-eq:4.1.5-findim}
\begin{aligned}
   \Norm{(\1+\eps^\alpha\mu^2 P)^{-1}\xi}
   &\le\Norm{\xi}
\\
   \Norm{(\1+\eps^\alpha\mu^2 P)^{-1} P\xi}
   &\le\Norm{\xi}
   &&&
   (\1+\eps^\alpha\mu^2 P)^{-1} P
   &=\tfrac{P}{1+\eps^\alpha\mu^2}
\\
   \Norm{(\1+\eps^\alpha\mu^2 P)^{-1}\eps^{\alpha/2}\mu P\xi}
   &\le\tfrac12\Norm{\xi}
   &&&
   \tfrac{\eps^{\alpha/2}}{1+\eps^\alpha\mu(s)^2}&\le\tfrac{1}{2\mu(s)}
\\
   \Norm{(\1+\eps^\alpha\mu^2 P)^{-1}\eps^\alpha\mu^2 P\xi}
   &\le\Norm{\xi}
   &&&
   \tfrac{\eps^{\alpha}}{1+\eps^\alpha\mu(s)^2}&\le\tfrac{1}{\mu(s)^2}
\end{aligned}
\end{equation}
for all constants $\eps>0$, vector fields
$\xi\in W^{1,2}(\R,q^*T\Sigma)$, and reals $s\in\R$.
\end{lemma}

Recall that $P^2=P$, pointwise at $q(s)$, is a projection, an
orthogonal one, hence of norm $1$. So estimate one with $\xi$
replaced by $P\xi$ implies estimate two.
Note that estimate two in the lemma allows for removing the square root
$\mu P$, at cost $\eps^{\alpha/2}$, of the operator $(\mu P)^2=\mu^2 P$
that appears in $(\1+\eps^\alpha\mu^2 P)^{-1}$, whereas removing
$(\mu P)^2=\mu^2 P$ itself has cost $\eps^\alpha\mu^2$. These facts
are somewhat hidden since $P^2=P$.
As it turns out only estimates one and two in
Lemma~\ref{no-le:4.1.5-findim} are of significance in the
present ODE adiabatic limit. In sharp contrast,
the refined estimates three and four were foundational in the
PDE adiabatic limit~\cite{salamon:2006a} where $P=\Nabla{t}$ is one
spatial derivative.
At present the finer estimate three in the lemma can still be used
for cosmetics, for example to get constant~$2$ in estimate three
in~(\ref{eq:4.1.4}), as opposed to a factor involving $\mu_\infty$,
see~(\ref{no-eq:findim-norm-xi_0}).

\begin{proof}
Let $\eps>0$ and $\xi\in W^{1,2}(\R,q^*T\Sigma)$. Pick $s\in\R$.
The operator
\begin{equation}\label{no-eq:findim-B}
   B(s):=\1+\eps^\alpha\mu_{q(s)}^2 P_{q(s)}
   \colon T_{q(s)}\Sigma\to T_{q(s)}\Sigma
\end{equation}
is symmetric since the projection is orthogonal, thus the eigenvalues
are real. The eigenvalues of $B(s)$ are positive: The projection $P_{q(s)}$,
defined by~(\ref{eq:findim-P}), has
eigenvalue $0$ on $V_{q(s)}^\perp$ and $1$ on the line
$V_{q(s)}=\R\Nabla{}\chi|_{q(s)}$.
Thus the operator $B(s)$ has eigenvalue
$1$ on $V_{q(s)}^\perp$ and $1+\eps^\alpha\mu_{q(s)}^2$ on $V_{q(s)}$.
Hence $B(s)$ is invertible.
The inverse $B(s)^{-1}$ has spectrum $\{1,(1+\eps^\alpha\mu_{q(s)}^2)^{-1}\}$,
thus norm $1$. Hence
\[
   \Norm{(\1+\eps^\alpha\mu^2 P)^{-1} \xi}
   =\Norm{B(s)^{-1}\xi}
   \le\Norm{\xi}.
\]
This proves estimate one.
For estimate two replace $\xi$ by $P\xi$
and use that by orthogonality $\abs{P_{q(s)}\xi(s)}\le\abs{\xi(s)}$
at any $s\in\R$.
The symmetric operator
\begin{equation*}%\label{no-eq:findim-B-1}
   (\1+\eps^\alpha\mu_{q(s)}^2 P_{q(s)})^{-1}P_{q(s)}
   \colon T_{q(s)}\Sigma\to T_{q(s)}\Sigma
\end{equation*}
has eigenvalue $0$ on $V_{q(s)}^\perp$
and $1/(1+\eps^\alpha\mu_{q(s)}^2)$ on
$V_{q(s)}=\im P_{q(s)}=\R\Nabla{}\chi(q(s))$.
This proves in~(\ref{no-eq:4.1.5-findim}) the identity in line two.
By Young $1\cdot\eps^{\alpha/2}\mu\le(1^2+(\eps^{\alpha/2}\mu)^2)/2$,
hence $\eps^{\alpha/2}\mu/(1+\eps^\alpha\mu^2)\le 1/2$ and this implies
estimate three. Clearly $\eps^\alpha\mu^2/(1+\eps^\alpha\mu^2)\le 1$
and this implies estimate four.
\end{proof}

\boldmath
%%%%%%%%%%%%%%%%%%%%%%%%%%%%%%%%%%%
\subsubsection{Component estimates}
%%%%%%%%%%%%%%%%%%%%%%%%%%%%%%%%%%%
\label{sec:components}
\unboldmath

As discussed prior to Lemma~\ref{no-le:4.1.5-findim}
we already choose $\beta=2$.

\begin{lemma}\label{le:4.1.4}
Let $q\in W^{1,2}(\R,\Sigma)$. In $\pi_\eps$ let
$\alpha\in[1,2]$ and $\beta=2$.
Then
\begin{equation}\label{eq:4.1.4}
\begin{aligned}
%   \Norm{X_1}=
   \Norm{X-\pi_\eps Z}
%1
   &\le\tfrac{1}{m_H}\Norm{dH|_q X}
   +\eps^\alpha\mu_\infty^2\Norm{P\tan X}
   +\eps^2\mu_\infty\Norm{\ell}
\\
%   \Norm{T_1}=
   \Norm{\ell -d\chi|_q\pi_\eps Z}
%2
   &\le
   \mu_\infty\Norm{P\tan X}
   +2\Norm{\ell }
\\
   \Norm{Z-I_q\pi_\eps Z}_{0,2,\eps}
%3
   &\le
   \tfrac{1}{m_H}\Norm{dH|_q  X}
   +2\mu_\infty^2\eps\Norm{P\tan X}
   +4\mu_\infty\eps\Norm{\ell}
\\
   \Norm{\pi_\eps Z}
%4
   &\le\Norm{I_q\pi_\eps Z}_{0,2,\eps}
   \le2\Norm{Z}_{0,2,\eps}
\end{aligned}
\end{equation}
for all constants $\eps\in(0,1]$ and pairs
$Z=(X,\ell)\in W^{1,2}(\R,q^*TM\oplus\R)$ where
\begin{equation}\label{eq:mu-infty}
   m_H:=\min_\Sigma \Abs{\Babla{} H}>0
   ,\qquad
   \mu_\infty:=\max\{1,\Norm{\Nabla{}\chi}_{L^\infty(\Sigma)}\}
   \in[1,\infty) .
\end{equation}
\end{lemma}

\begin{proof}
Given $q$ and $Z=(X,\ell)$, we denote
\begin{equation*}%\label{no-eq:findim-xi_0}
   \xi_0:=\pi_\eps Z
   =(\1+\eps^\alpha\mu^2 P)^{-1}
   \left(\tan X+\eps^2 \ell \Nabla{}\chi\right).
\end{equation*}
Write $X=\nor\, X+B^{-1}(B\tan X)$, with $B$ given
by~(\ref{no-eq:findim-B}), in order to obtain
\begin{equation*}%\label{no-eq:findim-X_1}
   X_1:=X-\xi_0
   =\nor\, X
   +(\1+\eps^\alpha\mu^2 P)^{-1}\left(\eps^\alpha\mu^2 P\tan X
   -\eps^2 \ell \Nabla{}\chi\right)
\end{equation*}
pointwise at $s\in\R$. 
By~(\ref{eq:findim-gradf}) and Lemma~\ref{no-le:4.1.5-findim}, we get
\begin{equation*}%\label{no-eq:findim-norm-X_1-id3}
   \Norm{X_1}
   \le\tfrac{1}{m_H}\Norm{dH|_q X}
   +\eps^\alpha\mu_\infty^2 \Norm{P\tan X}
   +\eps^2\mu_\infty\Norm{\ell}
\end{equation*}
Similarly, we get
\begin{equation*}%\label{no-eq:findim-T_1}
\begin{split}
   \ell _1:
%1
   &=\ell -d\chi|_q \xi_0\\
%2
   &=\ell -d\chi|_q\left(\1+\eps^\alpha\mu^2 P\right)^{-1}
   \left(\tan X+\eps^2 \ell \Nabla{}\chi\right)\\
%3
   &
   =\ell -\INNER{\Nabla{}\chi}
   {\left(\1+\eps^\alpha\mu^2 P\right)^{-1}
   \left(P\tan X+(\1-P)\tan X+\eps^2 \ell \Nabla{}\chi\right)}
   \\
%4
   &
   \stackrel{4}{=}
   \ell -\tfrac{\inner{\Nabla{}\chi}{P\tan X}}{1+\eps^\alpha\mu^2}-0
   -\tfrac{\eps^2\mu^2}{1+\eps^\alpha\mu^2} \ell 
\end{split}
\end{equation*}
By Lemma~\ref{no-le:4.1.5-findim} we get
\begin{equation*}%\label{no-eq:findim-T-dchi}
   \Norm{\ell _1}
   \le\mu_\infty\Norm{P\tan X}+2\Norm{\ell}.
\end{equation*}
For later use in~(\ref{no-eq:findim-norm-xi_0}),
note that by equality~4 above
\begin{equation*}%\label{no-eq:findim-dchi_xi_0}
   d\chi(q) \xi_0
%   =\tfrac{d\chi|_q\tan X+\eps^2\mu^2 \ell }{1+\eps^\alpha\mu^2}
   =
   \tfrac{\inner{\Nabla{}\chi}{P\tan X}}{1+\eps^\alpha\mu^2}
   +\tfrac{\eps^\alpha\mu^2}{1+\eps^\alpha\mu^2} \eps^{2-\alpha} \ell .
\end{equation*}
Take the sum of the estimates for $X_1$ and $\ell _1$ to obtain
\begin{equation*}
\begin{split}
   \Norm{Z-I_q\pi_\eps Z}_{0,2,\eps}
   &\le\Norm{X_1}+\eps\Norm{\ell _1}
\\
   &\le\tfrac{1}{m_H}\Norm{dH|_q X}
   +\mu_\infty^2\eps(1+\eps^{\alpha-1})\Norm{P\tan X}
   \\
   &\quad+2\mu_\infty\eps(1+\eps)\Norm{\ell}.
\end{split}
\end{equation*}
Now use the hypotheses $\alpha\ge 1$ and $\eps\le 1$.
By Lemma~\ref{no-le:4.1.5-findim}, also using the finer third estimate,
applied to the earlier identity for $\xi_0$, and for $d\chi(q)\xi_0$, we get
\begin{equation}\label{no-eq:findim-norm-xi_0}
\begin{aligned}
   \Norm{\xi_0}
   &\le\Norm{\tan X}+\tfrac12\eps^{2-\frac{\alpha}{2}}\Norm{\ell },
   \\
   \Norm{d\chi|_q\xi_0}
   &\le\tfrac12\eps^{-\frac{\alpha}{2}}\Norm{\tan X}
   +\eps^{2-\alpha}\Norm{\ell}.
\end{aligned}
\end{equation}
Square these two inequalities and take the sum to obtain
\begin{equation*}
\begin{split}
   \Norm{I_q\pi_\eps Z}_{0,2,\eps}^2
   &=\Norm{\xi_0}^2+\eps^2 \Norm{d\chi|_q\xi_0}^2\\
   &\le2(1+\tfrac{1}{4}\eps^{2-\alpha})\Norm{\tan X}^2
   +2\eps^{2-\alpha}(\tfrac{1}{4}+\eps^{2-\alpha})\eps^2\Norm{\ell }^2\\
   &\le3\left(\Norm{\tan X}^2+\eps^2\Norm{\ell }^2\right).
\end{split}
\end{equation*}
Note that $\Norm{\tan X}\le \Norm{X}$ since $\tan$ is an orthogonal
projection.
The proof of Lemma~\ref{le:4.1.4} is complete.
\end{proof}

\boldmath
%%%%%%%%%%%%%%%%%%%%%%%%%%%%%%%%%%%
%%%%%%%%%%%%%%%%%%%%%%%%%%%%%%%%%%%
\subsection{Comparing the base and ambient linear operators}
%%%%%%%%%%%%%%%%%%%%%%%%%%%%%%%%%%%
\label{sec:difference}
\unboldmath

We keep focusing on the special class of the ambient linear operators,
see~(\ref{eq:lin-eps-findim-lin}),
along the canonical embedding $i\colon q\mapsto (q,\chi(q))$.
The aim of this section is to control, downstairs in $q$-space,
the difference between the base linear operator along $q$ and the
ambient linear operator along $i(q)$.

For $q\in C^1(\R,\Sigma)$ denote the ambient linear operators along
the graph of $\chi$ over $q$ by $D^\eps_q:=D^\eps_{q,\chi(q)}$ and
$(D^\eps_q)^*:=(D^\eps_{q,\chi(q)})^*$. These operators
have the form
\begin{equation}\label{eq:D^eps_q}
\begin{split}
   D^\eps_q
   \begin{pmatrix}X\\\ell\end{pmatrix}
   &\stackrel{(\ref{eq:lin-eps-findim-lin})}{=}
   \begin{pmatrix}
      \Babla{s} X+\Babla{X}\Babla{} F|_q
      +\chi(q) \Babla{X}\Babla{} H|_q+\ell \Babla{} H|_q
   \\
      \ell^\prime+\eps^{-2} dH|_q X
   \end{pmatrix}
\\
   (D^\eps_q)^*
   \begin{pmatrix}X\\\ell\end{pmatrix}
   &\stackrel{(\ref{eq:lin-epsadjoint-findim-lin})}{=}
   \begin{pmatrix}
      -\Babla{s} X+\Babla{X}\Babla{} F|_q
      +\chi(q) \Babla{X}\Babla{} H|_q+\ell \Babla{} H|_q
   \\
      -\ell^\prime+\eps^{-2} dH|_q X
   \end{pmatrix}
\end{split}
\end{equation}
for every $Z=(X,\ell)\in W^{1,2}(\R,q^*TM\oplus\R)$.

\begin{proposition}\label{le:4.1.3}
In $\pi_\eps$ let $\alpha>0$ and $\beta=2$.
Let $q\in C^1(\R,\Sigma)$ be a map with bounded
derivative $\p_s q$. Then there is a constant $c_d>0$
such that
\begin{equation}\label{eq:4.1.3}
   \bigl\|(D^0_q)^*\pi_\eps Z
   -\pi_\eps (D^\eps_q)^*Z\Bigr\|
   \le
   \eps c_d\left(
   \tfrac{1}{\eps}\Norm{dH|_q X}
   +\eps^{\alpha-1}\Norm{\tan X}+\eps\Norm{\ell}
   \right)
\end{equation}
for every $Z=(X,\ell )\in W^{1,2}(\R,q^*TM\oplus\R)$ whenever $\eps\in(0,1]$.
The same is true for $D^0_q\pi_\eps-\pi_\eps D^\eps_q$. The constant
$c_d$ is invariant under $s$-shifts of $q$.
\end{proposition}

Note that for $\alpha=1$ all three terms on the right hand side
of~(\ref{eq:4.1.3}) are of the same quality in terms of powers of $\eps$
as in the ambient linear estimate~(\ref{eq:amblinest}).

\boldmath
%%%%%%%%%%%%%%%%%%%%%%%%%%%%%%%%%%%
\subsubsection{Commutators along $\Sigma$}
%%%%%%%%%%%%%%%%%%%%%%%%%%%%%%%%%%%
\label{sec:comutators}
\unboldmath

The proof of Proposition~\ref{le:4.1.3} below
suggests the value $\beta=2$.
For better reading we set $\beta=2$ already now. Let $\alpha\in\R$.

A commutator with the inverse operator
$\left(\1+\eps^\alpha\mu^2\, P\right)^{-1}$
should be rewritten in terms of a commutator with the operator itself.
The reason is that commutators are additive and the first summand of
$\1+\eps^\alpha\mu^2\, P$ commutes with anybody, thus disappears, and the
second summand then brings in the precious factor $\eps^\alpha$.

Here is an example of this technique, further below
in~(\ref{no-eq:findim-comm-2}) there will be another one.
In preparation to prove Proposition~\ref{le:4.1.3}
note that along $\Sigma$ it holds
\begin{equation*}
\begin{split}
   [\Nabla{s}, \left(\1+\eps^\alpha\mu^2\, P\right)^{-1}]
   &=\left(\1+\eps^\alpha\mu^2\, P\right)^{-1}
   [\1+\eps^\alpha\mu^2\, P,\Nabla{s}]
   \left(\1+\eps^\alpha\mu^2\, P\right)^{-1}\\
   &=\eps^\alpha\left(\1+\eps^\alpha\mu^2\, P\right)^{-1}
   [\mu^2\, P,\Nabla{s}]
   \left(\1+\eps^\alpha\mu^2\, P\right)^{-1}
\end{split}
\end{equation*}
where, by definition~(\ref{eq:findim-P}) of $P$,
the last commutator has the form
\[
    [\mu^2\, P,\Nabla{s}]\xi
   =-\INNER{\Nabla{s}\Nabla{}\chi}{\xi}\Nabla{}\chi
   -\INNER{\Nabla{}\chi}{\xi}\Nabla{s}\Nabla{}\chi
\]
for every $\xi\in W^{1,2}(\R,q^*T\Sigma)$. Thus, abbreviating
$B\stackrel{(\ref{no-eq:findim-B})}{:=} \1+\eps^\alpha\mu^2\, P$, we get
\begin{equation}\label{no-eq:findim-commut}
   [\Nabla{s},B^{-1}]\cdot
   =-\eps^\alpha B^{-1}
   \Bigl(\INNER{\Nabla{s}\Nabla{}\chi}{B^{-1}\cdot}\Nabla{}\chi
   +\INNER{\Nabla{}\chi}{B^{-1}\cdot}\Nabla{s}\Nabla{}\chi\Bigr).
\end{equation}

\begin{proof}[Proof of Proposition~\ref{le:4.1.3}]
Let $Z=(X,\ell ) \in W^{1,2}(\R,q^*TM\oplus\R)$.
We abbreviate $\xi_0:=\pi_\eps Z$ and write the operator $\pi_\eps$ in
the general form
\begin{equation*}
\begin{split}
   \xi_0:=\pi_\eps Z
   =B^{-1} \left(\tan X+\eps^\beta \ell \Nabla{}\chi\right)
   ,\qquad
   B\stackrel{(\ref{no-eq:findim-B})}{:=} \1+\eps^\alpha\mu^2\, P,
\end{split}
\end{equation*}
in order to identify how the natural choice $\beta=2$
arises. For simplicity of reading we mainly omit arguments $q$ and
$q(s)$. 
By~(\ref{eq:adjoint0-findim-lin}) the adjoint of $D^0_q$ is given by
\begin{equation*}
\begin{split}
   (D^0_q)^*\pi_\eps Z
   &\stackrel{(\ref{eq:adjoint0-findim-lin})}{=}
   -\Nabla{s}\xi_0+\Nabla{\xi_0}\Nabla{} f
\\
   &\stackrel{\;\;\,\xi_0\;\;\,}{=}
   -B^{-1}\Nabla{s} \left(\tan X+\eps^\beta \ell \Nabla{}\chi\right)
   -[\Nabla{s},B^{-1}] \left(\tan X+\eps^\beta \ell \Nabla{}\chi\right)
\\
   &\qquad\; +\Nabla{B^{-1}(\tan X+\eps^\beta \ell \Nabla{}\chi)}\Nabla{} f
\\
   &\stackrel{(\ref{no-eq:findim-commut})}{=}
   -B^{-1}
   \Bigl(
   \underline{\Nabla{s}\tan X}
   +\underline{\eps^\beta \ell^\prime\Nabla{}\chi}
   +\eps^\beta \ell \Nabla{s}\Nabla{}\chi
   \Bigr)
\\
   &\qquad\;
   +\eps^\alpha B^{-1}
   \Bigl(\INNER{\Nabla{s}\Nabla{}\chi}{\xi_0}\Nabla{}\chi
   +\INNER{\Nabla{}\chi}{\xi_0}\Nabla{s}\Nabla{}\chi\Bigr)
\\
   &\qquad\; +\Nabla{B^{-1}\tan X}\Nabla{} f
   +\eps^\beta \ell\Nabla{B^{-1}\Nabla{}\chi}\Nabla{} f .
\end{split}
\end{equation*}
The underlined terms annihilate their twins below when we take the difference.
We write $(D^\eps_q)^*Z=:(X^*,\ell ^*)$, where $(D^\eps_q)^*$ is given
by~(\ref{eq:D^eps_q}), then
\begin{equation*}
\begin{split}
   \pi_\eps(D^\eps_q)^*Z
   &=
   \pi_\eps (X^*,\ell ^*)
\\
   &=
   B^{-1}
   \left(\tan X^*+\eps^\beta \ell ^*\Nabla{}\chi\right)
\\
   &\stackrel{3}{=}
   B^{-1} \tan
   \Bigl(-\Babla{s} X
   +\Babla{X}\bigl(\Babla{} F|_q+\chi|_q\Babla{} H|_q\bigr)
   -(d\chi|_q X)\Babla{}H
   +\ell \Babla{}H
   \Bigr)
   \\
   &\quad
   +B^{-1}\left(
      -\eps^\beta \ell ^\prime
      +\eps^{\beta-2}dH|_q X
   \right) \Nabla{}\chi
\\
   &\stackrel{4}{=}
   -B^{-1}\left(
      \underline{\Nabla{s} \tan X}
         +\tan\Babla{s}\nor\, X
      -\Babla{\tan X} \Nabla{}f -\tan\Babla{\nor\, X} \Nabla{}f
   \right)
   \\
   &\quad
   -B^{-1}\left(
      \underline{\eps^\beta \ell ^\prime\Nabla{}\chi}
      -\eps^{\beta-2}(dH|_q X)\Nabla{}\chi
   \right) .
\end{split}
\end{equation*}
In identity~3 we pulled out the term $\Babla{X}$
from the sum of two terms whereby the extra term
$-(d\chi|_q X)\Babla{} H$ arises.
Identity~4 substitutes $\Babla{} F|_q+\chi|_q\Babla{} H|_q$ for $\Nabla{}f$,
by~(\ref{eq:findim-gradf}), and uses that
$\tan\Babla{}H=0=\tan\mathrm{II}$ and that
$\tan\Babla{}\chi=\Nabla{}\chi$, by~(\ref{eq:findim-gradf}).
We wrote $\Babla{s}X=\Babla{s}(\tan X+\nor\, X)$ and
$\Babla{X}\Nabla{}f=\Babla{\tan X}\Nabla{}f+\Babla{\nor\, X}\Nabla{}f$,
then we used~(\ref{eq:findim-decomp-II}) and that
normal parts $\mathrm{II}$ vanish under tangential projection.

Take the difference, so the $s$-derivatives (underlined) disappear,
and utilize~(\ref{eq:findim-decomp-II}), to obtain (the lower signs are
for $D^0_q\pi_\eps-\pi_\eps D^\eps_q$)
\begin{equation}\label{no-eq:findim-diff-adj}
\begin{split}
   &(D^0_q)^*\pi_\eps Z-\pi_\eps(D^\eps_q)^*Z\\
   &=
   -\eps^{\beta-2}(dH|_q X) B^{-1}\Nabla{}\chi
   \mp\eps^\beta \ell 
   \left( B^{-1}\Nabla{s}\Nabla{}\chi
   -\tfrac{1}{1+\eps^\alpha\mu^2}\,\Nabla{\Nabla{}\chi}\Nabla{} f\right)
\\
   &\quad
   \pm\eps^\alpha B^{-1}
   \bigl(\INNER{\Nabla{s}\Nabla{}\chi}{\xi_0}\Nabla{}\chi
   +\INNER{\Nabla{}\chi}{\xi_0}\Nabla{s}\Nabla{}\chi\bigr)
\\
   &\quad
   +\Nabla{B^{-1}\tan X}\Nabla{} f-B^{-1}\Nabla{\tan X}\Nabla{} f
\\
   &\quad
   \pm B^{-1}\tan\Babla{s}\nor\, X
   \mp B^{-1}\tan\Babla{\nor\, X} \Nabla{}f .
\end{split}
\end{equation}

To finish the proof it remains to inspect line by line the $L^2$ norm
of these four lines, denoted by $L_1,\dots,L_4$.
The coefficient $\eps^{\beta-2}$ suggests to choose $\beta\ge 2$.
In view of line four, see analysis below, choosing $\beta>2$
does not improve the overall estimate for the term $dH|_q X$.
So the value $\beta=2$ that appears in the orthogonal projection
will be just fine.\footnote{
  We do not see here the phenomenon that the two most unpleasant terms,
  here $dH|_q X$, appear
  with opposite signs, one with~$\eps^0$ and one with~$\eps^{\beta-2}$
  thereby \emph{enforcing} the choice $\beta=2$,
  as opposed to~\cite[p.\,1132, formula for $\pi_\eps\Dd_u^\eps\zeta$,
  unpleasant terms $\Nabla{t}\eta$ already cancelled]{salamon:2006a}.
}

\smallskip
To estimate line one $L_1$ we use that $\Norm{B^{-1}}\le1$,
by~(\ref{no-eq:4.1.5-findim}),
to obtain
\[
   \Norm{L_1}
   \le\mu_\infty \eps^{\beta-2} \Norm{dH|_q X}
   +c_a \eps^\beta\Norm{\ell }
\]
where $c_a$ depends on $\Norm{\p_s q}_\infty$, the $C^2(\Sigma)$-norms
of $\chi$ and $f$, and on $\mu_\infty$.

\smallskip
Concerning line two $L_2$, by definition of $\xi_0$ and since
$\Norm{B^{-1}}\le 1$, we obtain
\[
   \Norm{L_2}
   \le C\eps^\alpha \Norm{\xi_0}
   ,\qquad
   \Norm{\xi_0}
   \le\Norm{\tan X}+\mu_\infty\eps^\beta\Norm{\ell} ,
\]
where $C$ depends on $\Norm{\p_s q}_\infty$, the $C^2(\Sigma)$-norm
of $\chi$, and $\mu_\infty$.

\smallskip
Line three $L_3$ in~(\ref{no-eq:findim-diff-adj}) is of the form
\[
   [\Phi,B^{-1}]
   =B^{-1}[B,\Phi]B^{-1}
   =B^{-1}[\1+\eps^\alpha\mu^2 P,\Phi]B^{-1}
   =\eps^\alpha\mu^2 B^{-1}[P,\Phi]B^{-1}
\]
where $\Phi\colon W^{1,2}(\R,q^*TM)\to W^{1,2}(\R,q^*TM)$ is given by
$\Phi\xi=\Nabla{\xi}\Nabla{} f$.
Thus
\begin{equation}\label{no-eq:findim-comm-2}
\begin{split}
   \Norm{L_3}
   &=\Norm{[\Phi,B^{-1}]\tan X}\\
   &=\Norm{\eps^\alpha\mu^2B^{-1}
   \left(
      P\Nabla{B^{-1}\tan X}\Nabla{} f-\Nabla{PB^{-1}\tan X}\Nabla{} f
   \right)
   }
   \\
   &\le\eps^\alpha\mu_\infty^2\Norm{f}_{C^2(\Sigma)}\Norm{\tan X}
\end{split}
\end{equation}
since $\Norm{B^{-1}}\le 1$, by~(\ref{no-eq:4.1.5-findim}),
and since orthogonal projection have $\Norm{P}=1$.

\smallskip
Line four $L_4$ in~(\ref{no-eq:findim-diff-adj}):
For the first summand, by~(\ref{eq:findim-orth-split})
and Leibniz, we get
\begin{equation*}
\begin{split}
   \Babla{s}\nor\, X
   &=
   \left(\tfrac{\inner{\Babla{} H}{X}}{\abs{\Babla{} H}^2}\right)^\prime
   \Babla{} H
   +\tfrac{\inner{\Babla{} H}{X}}{\abs{\Babla{} H}^2}
   \Babla{s}\Babla{} H .
\end{split}
\end{equation*}
Now use orthogonality $\Babla{} H\perp \tan X$
and write $X=\tan X+\nor\, X$ to obtain
\begin{equation*}
\begin{split}
   \tan\Babla{s}\nor\, X
   &=\tfrac{\inner{\Babla{} H}{\nor\, X}}{\abs{\Babla{} H}^2}
   \tan\Babla{s}\Babla{} H
\end{split}
\end{equation*}
where the right-hand side is linear in $\nor\, X$.
Use this formula to get the estimate
\begin{equation}\label{no-eq:findim-tan-nor}
   \Norm{\tan\Babla{s}\nor\, X}
   \le\Norm{\tfrac{\tan\Babla{s}\Babla{} H}{\abs{\Babla{} H}}}_\infty
   \Norm{\nor\, X}
   \stackrel{(\ref{eq:findim-gradf})}{\le}
   \tfrac{\norm{\Babla{\cdot}\Babla{} H}_\infty\norm{\p_sq}_\infty}{m_H^2}
   \Norm{dH|_q X}
\end{equation}
where $\norm{\Babla{\cdot}\Babla{} H}_\infty$ is over the compact $\Sigma$.
For the second summand of $L_4$ we get
\begin{equation*}
   \Norm{\tan\Babla{\nor\, X} \Nabla{}f}
   \le\Norm{\Babla{\nor\, X} \Nabla{}f}
   \le\Norm{\Babla{\cdot}\Nabla{}f}_\infty
   \Norm{\nor\, X}
   \stackrel{(\ref{eq:findim-gradf})}{\le}
   \tfrac{\Norm{\Babla{\cdot}\Nabla{}f}_\infty}{m_H}
   \Norm{dH|_q X} .
\end{equation*}
For $\alpha>0$, $\beta=2$, and $\eps>0$
the estimates together prove the $L^2$ bound~(\ref{eq:4.1.3}).
All estimates are invariant under $s$-shifts of $q$,
because all constants depend on the $L^\infty$
norm of $\p_s q$.
The proof of Proposition~\ref{le:4.1.3} is complete.
\end{proof}

\boldmath
%%%%%%%%%%%%%%%%%%%%%%%%%%%%%%%%%%%
%%%%%%%%%%%%%%%%%%%%%%%%%%%%%%%%%%%
\subsection{Right inverse -- key estimate}
%%%%%%%%%%%%%%%%%%%%%%%%%%%%%%%%%%%
%%%%%%%%%%%%%%%%%%%%%%%%%%%%%%%%%%%
\label{sec:ff-right-inverse}
\unboldmath

In this section we show that if the base flow is Morse-Smale,
then so is the ambient $\eps$-flow for all $\eps>0$ small,
see Theorem~\ref{thm:KeyEst-thm.3.3}.

\boldmath
%%%%%%%%%%%%%%%%%%%%%%%%%%%%%%%%%%%
%%%%%%%%%%%%%%%%%%%%%%%%%%%%%%%%%%%
\subsubsection*{Definition of right inverse}
%%%%%%%%%%%%%%%%%%%%%%%%%%%%%%%%%%%
\unboldmath

Suppose that $q\in\Mm^0_{x^-,x^+}$. By Morse-Smale the linear operator
\[
   D^0_q\colon W^{1,2}(\R,\Sigma)\to L^2(\R,\Sigma) 
\]
is surjective. By~(\ref{eq:ker-coker}) this is equivalent to
injectivity of the adjoint $(D^0_q)^*$.
Here the Fredholm operator property of $D^0_q$ and $(D^0_q)^*$ enters
which holds true, see Proposition~\ref{prop:findim-0-Fredholm},
since  Morse-Smale implies Morse.

\smallskip
The main result of this section, Theorem~\ref{thm:KeyEst-thm.3.3},
tells that surjectivity of $D^0_q$ implies, for $\eps>0$ small,
surjectivity of $D^\eps_q$, equivalently injectivity of $(D^\eps_q)^*$.
As $\ker D^\eps_q=\im (D^\eps_q)^*$, by analogy to~(\ref{eq:ker-coker}),
the composition $D^\eps_q{D^\eps_q}^*\colon W^{2,2}\to L^2$ is a
bijection and, as a composition of bounded operators, it is bounded.
So $D^\eps_q{D^\eps_q}^*$ has a bounded inverse by the open
mapping theorem. Then the operator
\begin{equation}\label{eq:Reps}
   R^\eps_q
   :=(D^\eps_q)^*\left(D^\eps_q(D^\eps_q)^*\right)^{-1}
   \colon L^2
   \stackrel{(...)^{-1}}{\longrightarrow} W^{2,2}
   \stackrel{(D^\eps_q)^*}{\longrightarrow} W^{1,2}
\end{equation}
is bounded and a right inverse of the operator $D^\eps_q$
given by~(\ref{eq:D^eps_q}).

\smallskip
Boundedness of $R^\eps_q$ is not enough to get a bijection
$\Tt^\eps\colon \Mm^0_{x^-,x^+}\to\Mm^\eps_{x^-,x^+}$
between base and ambient moduli spaces
for every parameter value $\eps>0$ small.
To achieve this via the Newton method,
what we need is a \emph{uniform} bound that works for every $\eps>0$ small.
Uniform boundedness of the right inverse amounts to establishing
uniform estimates for $D^\eps_q$ along the image of the formal adjoint.
This is also part of Theorem~\ref{thm:KeyEst-thm.3.3}.
To have a chance to obtain uniform bounds in~$\eps$
one works with Sobolev norms $\norm{\cdot}_{0,2,\eps}$
and $\norm{\cdot}_{1,2,\eps}$ weighted by suitable powers of $\eps$,
see~(\ref{eq:findim-0-2-eps}).
The weights are suggested by, respectively, the $\eps$-energy identity
and the ambient linear estimate.

\boldmath
%%%%%%%%%%%%%%%%%%%%%%%%%%%%%%%%%%%
%%%%%%%%%%%%%%%%%%%%%%%%%%%%%%%%%%%
\subsubsection{The Fredholm operator interchange estimate}
%%%%%%%%%%%%%%%%%%%%%%%%%%%%%%%%%%%
%%%%%%%%%%%%%%%%%%%%%%%%%%%%%%%%%%%
\label{sec:gen-Fred}
\unboldmath

In adiabatic limit analysis when one proves the key estimates
for the linearized operator along the image of the adjoint
(in the present article Theorem~\ref{thm:KeyEst-thm.3.3})
one needs to interchange the base and ambient operators at some point.
For future reference we include the proof of an abstract
version of~\cite[Le.\,D.7]{salamon:2006a} for Fredholm operators
$D$ and $D^\prime$. In practice $D^\prime$ is the formal adjoint
of $D$, so the isomorphism hypothesis on the maps $A$ and $B$ is
satisfied automatically.

\begin{lemma}\label{le:4.4.5}
Let $D,D^\prime\colon W\to E$ be Fredholm operators between Banach spaces
such that $W$ is contained and dense in $E$ and such that the maps defined by
\begin{equation*}
\begin{aligned}
   A\colon\ker D
   &\stackrel{\simeq}{\to}
   coker D^\prime:=\tfrac{E}{\im D^\prime},
   &&&
   B\colon\ker D^\prime
   &\stackrel{\simeq}{\to}
   \coker D :=\tfrac{E}{\im D},
\\
   \xi
   &\mapsto \xi+\im D^\prime
  &&&
  \eta
   &\mapsto \eta+\im D
\end{aligned}
\end{equation*}
are isomorphisms.
Let $D$ be surjective.
Then there is a constant $c$ such that
\begin{equation}\label{eq:4.4.5}
\begin{split}
   \norm{\eta}_W
   &\le c
   \norm{D^\prime\eta}_E
   \\
   \norm{\xi}_W
   &\le c
   \left(
   \norm{\xi-D^\prime\eta}_E+\norm{D\xi}_E
   \right)
\end{split}
\end{equation}
for all $\xi,\eta\in W$.
\end{lemma}

\begin{proof}[Proof of Lemma~\ref{le:4.4.5}]
Since $D$ is surjective $D^\prime$ is injective as the isomorphism $B$ shows.
Hence estimate one in~(\ref{eq:4.4.5}) follows from the open
mapping theorem; see e.g.~\cite[Thm.\,4.13]{Rudin:1991b}

The linear map $P\colon E\to E/\im D^\prime$, defined by
$\xi\mapsto \xi+\im D^\prime$, is continuous since the target space is of
finite dimension. The operator
\[
   T\colon W\to E\oplus\tfrac{E}{\im D^\prime},\quad
   \xi\mapsto\left(D\xi, P\xi\right) ,
\]
is an injective Fredholm operator: Linearity is clear
and continuity holds by continuity of $D$ and of $P$.
Note that $\ker T\subset \ker D$.
For injectivity let $\xi\in\ker T$, then $D\xi=0$ and $0=P\xi=A\xi$. But
then $\xi=0$ since $A$ is an isomorphism.
The image of $T$ is closed, since so is the image of $D$
and since the dimension of $\ker D$ is finite.
The image of $T$ has finite codimension, since so has $D$ and
since $\frac{E}{\im D^\prime}$ is of finite  dimension.

By injectivity and closed range the operator $T$, as a map
$W\to\im T$, is a bijection between Banach spaces.
Thus by the open mapping theorem,
see e.g.~\cite[Cor.\,2.12\,(c)]{Rudin:1991b},
there is a constant $c>0$ such that
\[
   \norm{\xi}_W
   \le c\norm{T\xi}
   =c\left(\norm{D\xi}_E+\norm{P\xi}_{E/\im D^\prime}\right)
\]
for every $\xi\in W$. Given $\eta\in W$, then
$
   D^\prime\eta\in\im D^\prime=\ker P
$.
Thus, by continuity of $P$ with constant $C$, we get
 $\norm{P\xi}=\norm{P(\xi- D^\prime\eta)}
\le C\norm{\xi- D^\prime\eta}_E$.
\end{proof}

\boldmath
%%%%%%%%%%%%%%%%%%%%%%%%%%%%%%%%%%%
%%%%%%%%%%%%%%%%%%%%%%%%%%%%%%%%%%%
%%%%%%%%%%%%%%%%%%%%%%%%%%%%%%%%%%%
\subsubsection{Weak injectivity estimate of $(D^\eps_q)^*$}
\unboldmath

To show injectivity of $(D^\eps_q)^*\colon W^{1,2}\to L^2$ amounts to
prove the last estimate in~(\ref{eq:prop:4.1.2-findim}) with the
$(1,2,\eps)$-norm on the left-hand side. In this section we aim for
the weaker $(0,2,\eps)$-norm and this is why we use the term weak
injectivity.

\begin{proposition}[Weak injectivity of adjoint $(D_q^\eps)^*$]
\label{prop:4.1.2}
In $\pi_\eps$ let $\alpha\in[1,2]$ and $\beta=2$.
Let $x^\mp\in\Crit f$ be non-degenerate and
$q\in\Mm^0_{x^-,x^+}$ a connecting base trajectory
such that $D^0_q\colon W^{1,2}\to L^2$ is surjective.
Then there are constants $c>0$ and $\eps_0\in(0,1]$
such that for any parameter value $\eps\in(0,\eps_0]$ it holds that
\begin{equation}\label{eq:prop:4.1.2-findim}
\begin{aligned}
   \norm{X}
   &\le c\left(
   \eps\norm{(D^\eps_q)^* Z }_{0,2,\eps}
   +\norm{\pi_\eps (D^\eps_q)^* Z }
   \right)
   \\
   \norm{dH(u)X}+\eps\norm{\ell}
   &\le c\left(
   \eps\norm{(D^\eps_q)^*  Z }_{0,2,\eps}
   +{\color{red}\eps}\norm{\pi_\eps (D^\eps_q)^* Z }
   \right)
   \\
   \norm{ Z }_{0,2,\eps}
   &\le c\left(
   \eps\norm{(D^\eps_q)^* Z }_{0,2,\eps}
   +\norm{\pi_\eps (D^\eps_q)^* Z }
   \right)
   \\
   \norm{ Z }_{0,2,\eps}
   &\le c
   \norm{(D^\eps_q)^*  Z }_{0,2,\eps}
   \quad{\color{gray}\text{\small\rm(weak injectivity estimate)}}
\end{aligned}
\end{equation}
for every $Z=(X,\ell)\in W^{1,2}(\R,q^*TM\oplus\R)$.
\end{proposition}

\begin{proof}
Let $\eps\in(0,1]$.
A base connecting trajectory $q\in\Mm^0_{x^-,x^+}$
is smooth, by Lemma~\ref{le:findim-reg0},
and $\norm{\p_sq}\le\osc f$ is finite, by the energy
identity~(\ref{eq:energy-0-osc}).
So the difference Proposition~\ref{le:4.1.3} applies.
By Lemma~\ref{le:4.4.5}, which applies due to the Fredholm
Proposition~\ref{prop:findim-0-Fredholm},
there is a constant $c_0>0$ such that
\begin{equation*}%\label{eq:findim-D0-inj}
   \norm{\xi}
   \le c_0\norm{(D_q^0)^*\xi}
\end{equation*}
for every $\xi\in W^{1,2}(\R,q^*T\Sigma)$.
The inequality for $\xi=\pi_\eps Z$ is used in step 2 of what follows.
In step~1 and~3 add zero and use the triangle inequality to get
\begin{equation*}
\begin{split}
   \norm{X}
%1
   &\stackrel{{\color{white} (2.34)}}{\le}
   \norm{X-\pi_\eps Z}
   +\norm{\pi_\eps Z}
   \\
%2
   &\stackrel{\text{comps.}\atop(\ref{eq:4.1.4})}{\le}
   \tfrac{1}{m_H}\Norm{dH|_q X}
   +\eps^\alpha\mu_\infty^2\Norm{P\tan X}
   +\eps^2\mu_\infty\Norm{\ell}
   +c_0\norm{(D_q^0)^*\pi_\eps Z}
   \\
%3
   &\stackrel{{\color{white} (2.34)}}{\le}
   \tfrac{1}{m_H}\Norm{dH|_q X}
   +\eps^\alpha\mu_\infty^2\Norm{P\tan X}
   +\eps^2\mu_\infty \Norm{\ell}
   +c_0\norm{\pi_\eps  (D^\eps_q)^*Z}
   \\
   &\qquad\;+c_0\norm{(D_q^0)^*\pi_\eps Z
   -\pi_\eps  (D^\eps_q)^*Z}
   \\
%4
   &\stackrel{\text{diff.}\atop(\ref{eq:4.1.3})}{\le}
   \tfrac{1}{m_H}\Norm{dH|_q X}
   +\eps^\alpha\mu_\infty^2\Norm{P\tan X}
   +\eps^2\mu_\infty\Norm{\ell}
   +c_0\norm{\pi_\eps  (D^\eps_q)^*Z}
   \\
   &\qquad\;
   +c_0c_d\left(
   \Norm{dH|_q X}
   +\eps^\alpha\Norm{\tan X}+\eps^2\Norm{\ell}
   \right)
   \\
%5
   &\stackrel{{\color{white} (2.34)}}{\le}
   \eps \left(\tfrac{1}{m_H}+c_0c_d+\mu_\infty^2\right)
   \left(
   \tfrac{1}{\eps}\Norm{dH|_q X}
   +\eps\Norm{\ell}
   +\eps^{\alpha-1}\Norm{X}
   \right)
   \\
   &\qquad\;
   +c_0\norm{\pi_\eps  (D^\eps_q)^*Z}
   \\
%6
   &\stackrel{\text{amb.}\atop(\ref{eq:amblinest})}{\le}
   \eps (c_a+1)\left(\tfrac{1}{m_H}+c_0c_d+\mu_\infty^2\right)
   \left(\norm{(D^\eps_q)^*Z}_{0,2,\eps}+\norm{X}\right)
   \\
   &\qquad\;
   +c_0\norm{\pi_\eps  (D^\eps_q)^*Z} .
\end{split}
\end{equation*}
Here~(\ref{eq:4.1.4}) requires $\alpha\in[1,2]$, the last step $\alpha\ge 1$.
Choose $\eps_0>0$ so small that
\[
   \eps_0 C:=\eps_0 (c_a+1)\left(\tfrac{1}{m_H}+c_0c_d+\mu_\infty^2\right)
   \le\tfrac12 .
\]
Then we can incorporate the term $\Norm{X}$ into the left-hand side
and get that
\begin{equation}\label{eq:findim-111}
   \norm{X}
   \le 
   2C %c_a\left(\tfrac{1}{m_H}+c_0c_d+\mu_\infty^2\right)
   \eps\norm{(D^\eps_q)^*Z}_{0,2,\eps}
   +2c_0\norm{\pi_\eps  (D^\eps_q)^*Z} .
\end{equation}
Multiply by $\eps$ the ambient estimate~(\ref{eq:amblinest})
for $(D^\eps_q)^*$ with constant $c_a$ to obtain
\begin{equation*}
\begin{split}
   \norm{dH(u)X}+\eps\norm{\ell}
%   &+\eps\norm{\Babla{s} X}
%   +\eps^2\norm{\ell ^\prime}\\
   &\stackrel{\text{amb.}\atop(\ref{eq:amblinest})}{\le}
   \eps c_a\Bigl(
   \norm{(D^\eps_q)^* Z}_{0,2,\eps}+\norm{X}
   \Bigr)
   \\
   &\stackrel{(\ref{eq:findim-111})}{\le}
   \eps c_a \Bigl((1+2\eps C)\norm{(D^\eps_q)^*Z}_{0,2,\eps}
   +2c_0\norm{\pi_\eps  (D^\eps_q)^*Z}\Bigr) .
\end{split}
\end{equation*}
The previous two estimates provide inequality two in the following
\begin{equation*}
\begin{split}
   \norm{Z}_{0,2,\eps}
   &\stackrel{(\ref{eq:findim-0-2-eps})}{\le}
   \norm{X}+\eps\norm{\ell}
   \\
   &\stackrel{{\color{white} (4.42)}}{\le}
   \eps(2C+c_a(1+2\eps C)) \norm{(D^\eps_q)^*Z}_{0,2,\eps}
   +2c_0(1+c_a\eps)\norm{\pi_\eps  (D^\eps_q)^*Z}
   \\
   &\stackrel{(\ref{eq:4.1.4})}{\le}
   \eps(2C+c_a(1+2\eps C)) \norm{(D^\eps_q)^*Z}_{0,2,\eps}
   +4c_0(1+\eps c_a) \norm{(D^\eps_q)^*Z}_{0,2,\eps}
\end{split}
\end{equation*}
where the last step uses the last estimate in~(\ref{eq:4.1.4}).
This proves the final assertions three and four of
Proposition~\ref{prop:4.1.2} whose proof is thereby complete.
\end{proof}

\boldmath
%%%%%%%%%%%%%%%%%%%%%%%%%%%%%%%%%%%
%%%%%%%%%%%%%%%%%%%%%%%%%%%%%%%%%%%
\subsubsection{Surjectivity of $D^\eps_q$ and key estimate}
\label{sec:key-estimate}
\unboldmath

\begin{theorem}[Surjectivity and key estimates for $D^\eps_q$ on
image of $(D^\eps_q)^*$]
\label{thm:KeyEst-thm.3.3}
In $\pi_\eps$ let $\alpha\in[1,2]$ and $\beta=2$.
Let $x^\mp\in\Crit f$ be non-degenerate and
$q\in\Mm^0_{x^-,x^+}$ a connecting base trajectory
such that $D^0_q\colon W^{1,2}\to L^2$ is surjective.
Then there are positive constants $c$ and $\eps_0$
(invariant under $s$-shifts of $q$) such that, for every
$\eps\in(0,\eps_0]$, the following is true.
The operator $D_q^\eps\colon W^{1,2}\to L^2$ is onto and along the
image of the to $W^{2,2}$ restricted adjoint, that is for every pair
\[
   Z^*:=(X^*,\ell^*)
   \in\im (D^\eps_q)^*|_{W^{2,2}} \subset W^{1,2}(\R,q^*TM\oplus\R),
\]
there are the key estimates
\begin{equation}\label{eq:thm:4.4.4}
\begin{aligned}
   {\color{gray}\norm{X^*}\le\;}
   \norm{ Z^*}_{1,2,\eps}
   &\le c\left(\eps
   \norm{D^\eps_q  Z^*}_{0,2,\eps}
   +\norm{\pi_\eps (D^\eps_q  Z^*)}
   \right)
\\
   \eps^{1/2}\norm{ Z^*}_{0,\infty,\eps}
   +\norm{ Z^*}_{1,2,\eps}
   &\le c
   \norm{D^\eps_q  Z^*}_{0,2,\eps}
\\
   \norm{dH|_qX^*}+\eps\norm{\ell^*}
   &+\eps\norm{\Babla{s} X^*}
   +\eps^2\norm{(\ell^*)^\prime}\\
   &\le c\left(\eps\norm{D^\eps_q  Z^*}_{0,2,\eps}
   +{\color{red}\eps}\norm{\pi_\eps (D^\eps_q  Z^*)}
   \right)
\\
   &\le 3c\eps
   \norm{D^\eps_q  Z^*}_{0,2,\eps} .
\end{aligned}
\end{equation}
\end{theorem}

\begin{proof}%[Proof of Theorem~\ref{thm:4.4.4}]
A base connecting trajectory $q\in\Mm^0_{x^-,x^+}$
is smooth, by Lemma~\ref{le:findim-reg0},
and $\norm{\p_sq}\le\osc f$ is finite, by the energy
identity~(\ref{eq:energy-0-osc}).
So we are in position to apply the difference
Proposition~\ref{le:4.1.3} with constant $c_d$
and the weak injectivity Proposition~\ref{prop:4.1.2} which provides
a constant $\eps_0\in(0,1]$. Let $\eps\in(0,\eps_0]$.

\smallskip
To see surjectivity of the Fredholm operator $D_q^\eps$ or,
equivalently, injectivity of $(D^\eps_q)^*$, 
pick $Z=(X,\ell)\in W^{1,2}(\R,q^*TM\oplus\R)$.
Use consequence~(\ref{eq:cor-amblinest})
of the ambient linear estimate with
constant $C_a$ (shrink $\eps_0>0$ if necessary) to obtain
\begin{equation}\label{eq:le:4.1.6-1,2,eps-proof}
\begin{split}
   {\color{gray}\norm{X}\le\;}
   \norm{ Z}_{1,2,\eps}
   &\le\eps C_a\norm{(D^\eps_q)^* Z}_{0,2,\eps}+\norm{\tan X}
   \\
   &\le (\eps C_a+c_w) \norm{(D^\eps_q)^* Z}_{0,2,\eps} .
\end{split}
\end{equation}
In the second step we used
$\norm{\tan X}\le\norm{X}\le \norm{ Z}_{0,2,\eps}$, then we
applied the weak injectivity estimate~(\ref{eq:prop:4.1.2-findim}) with
constant $c_w$. Thus $(D^\eps_q)^*$ is injective.

\smallskip
Now pick $Z=(X,\ell)\in W^{2,2}(\R,q^*TM\oplus\R)$
and set $Z^*:=(D^\eps_q)^* Z$.
To prove the first two lines in~(\ref{eq:thm:4.4.4})
let $c_F$ be the constant of the Fredholm interchange Lemma~\ref{le:4.4.5}. 
%Let $c_d$ be the constant of the operator difference Proposition~\ref{le:4.1.3}.
By~(\ref{eq:4.4.5}) in Lemma~\ref{le:4.4.5}, with $\xi=\pi_\eps Z^*$ and
$\eta=\pi_\eps Z$, we have
\begin{equation*}
\begin{split}
   \norm{\pi_\eps Z^*}   %\\
%1
   &\stackrel{(\ref{eq:4.4.5})}{\le}
   c_F\norm{\pi_\eps Z^*-(D^0_q)^*\pi_\eps Z}
   +c_F\norm{D^0_q\pi_\eps Z^*}
   \\
%2
   &\stackrel{\text{add $0$}}{\le}
   c_F\left(
{\color{brown}
   \norm{\pi_\eps(D^\eps_q)^* Z-(D^0_q)^*\pi_\eps Z}
}
   +
{\color{cyan}
   \norm{D^0_q\pi_\eps Z^*-\pi_\eps D^\eps_q Z^*}
}
   +\norm{\pi_\eps D^\eps_q Z^*}
   \right)
   \\
%3
   &\stackrel{\text{diff.}\atop(\ref{eq:4.1.3})}{\le}
   c_F
{\color{brown}\,
   c_d\eps \left(
   \tfrac{1}{\eps}\Norm{dH|_q X}
   +\eps^{\alpha-1}\Norm{\tan X}+\eps\Norm{\ell}
   \right)
}
   +c_F\norm{\pi_\eps D^\eps_q Z^*}
   \\
   &\qquad
   +c_F
{\color{cyan}\,
   c_d\eps \left(
  \tfrac{1}{\eps}\Norm{dH|_q X^*}
   +\eps^{\alpha-1}\Norm{\tan X^*}+\eps\Norm{\ell^*}
   \right)
}
   \\
%4
   &\stackrel{\alpha\in[1,2]}{\le}
   c_Fc_d\eps c_a\left(\underline{
   {\color{brown}\norm{ Z^*}_{0,2,\eps}+\norm{X}}}
   +{\color{cyan}\norm{D^\eps_q Z^*}_{0,2,\eps}
   +\,}\underline{{\color{cyan}\norm{X^*}}}
   \right)
   +c_F\norm{\pi_\eps D^\eps_q Z^*}
   \\
%5
   &\stackrel{(\ref{eq:le:4.1.6-1,2,eps-proof})}{\le}
   c_1 \underline{\eps\norm{ Z^*}_{0,2,\eps}}
   +c_Fc_dc_a \eps\norm{D^\eps_q Z^*}_{0,2,\eps}
   +c_F\norm{\pi_\eps D^\eps_q Z^*}
\end{split}
\end{equation*}
where $c_1=c_Fc_dc_a(2+\eps C_a+c_w)$.
In step~4 we used twice the ambient linear
estimate~(\ref{eq:amblinest})
with constant $c_a$,
once for ${\color{brown} (D^\eps_q)^*}$ and once for
${\color{cyan} D^\eps_q}$.
In the final step (underlined terms) we estimate $\norm{X}$
by~(\ref{eq:le:4.1.6-1,2,eps-proof}) and $\norm{X^*}$ by
$\norm{Z^*}_{0,2,\eps}$.

Now add zero and use the formula for the linearized injection $I_q$
prior to Definition~\ref{def:orth-proj}, then
apply estimate three of the component Lemma~\ref{le:4.1.4} to~get
\begin{equation*}
\begin{split}
   &\norm{ Z^*}_{0,2,\eps}
   \\
   &\stackrel{{\color{white}(2.37)}}{\le}
    \norm{ Z^*-I_q\pi_\eps Z^*}_{0,2,\eps}
   +{\color{brown}\norm{(\pi_\eps Z^*,d\chi|_q \pi_\eps Z^*)}_{0,2,\eps}}
   \\
   &\stackrel{\text{comps.}\atop(\ref{eq:4.1.4})}{\le}
   3\mu_\infty^2\eps
   \left(
   \tfrac{\eps^{-1}}{m_H}\norm{dH|_qX^*}+\norm{\tan X^*}+\norm{\ell^*}
   \right)
   +{\color{brown}\norm{\pi_\eps Z^*}
   +\eps\norm{d\chi|_q \pi_\eps Z^*}}
   \\
   &\stackrel{\text{amb.}\atop(\ref{eq:amblinest})}{\le}
   \eps c_2
   \left(\norm{D^\eps_q Z^*}_{0,2,\eps}+\norm{X^*}\right)
   +\;{\color{brown}(1+\mu_\infty\eps)\norm{\pi_\eps Z^*}}
   \\
   &\stackrel{{\color{white}(2.37)}}{\le}
   (c_2+c_3c_Fc_dc_a)\eps\norm{D^\eps_q Z^*}_{0,2,\eps}
   +(c_2+c_3c_1)\eps\norm{ Z^*}_{0,2,\eps}
   +c_3 c_F\norm{\pi_\eps D^\eps_q Z^*}
\end{split}
\end{equation*}
where $c_2=\tfrac{3\mu_\infty^2\max\{1,m_H\}}{m_H} c_a$
and $c_3=(1+\mu_\infty\eps)$.
Inequality three uses the ambient linear estimate~(\ref{eq:amblinest})
and definition~(\ref{eq:mu-infty}) of the constant $\mu_\infty\ge 1$.
The final inequality four uses that $\norm{X^*}\le\norm{Z^*}_{0,2,\eps}$
and the previously established estimate for $\norm{\pi_\eps Z^*}$.
Choosing $\eps_0>0$ sufficiently small, we obtain
\begin{equation}\label{eq:zeta^*}
   {\color{gray}\norm{\tan X^*}\le\norm{X^*}\le\;}
   \norm{ Z^*}_{0,2,\eps}
   \le c_4\eps\norm{D^\eps_q Z^*}_{0,2,\eps}
   +2c_3c_F\norm{\pi_\eps D^\eps_q Z^*} .
\end{equation}
By the ambient linear estimate consequence~(\ref{eq:cor-amblinest})
for $D^\eps_q$, constant $C_a$, we have
\begin{equation*}
\begin{split}
   \norm{ Z^*}_{1,2,\eps}
   &\le \eps C_a\norm{D^\eps_q Z^*}_{0,2,\eps}+\norm{\tan X^*} .
\end{split}
\end{equation*}
Combining this with~(\ref{eq:zeta^*}) proves inequality one
in~(\ref{eq:thm:4.4.4}). 
Inequality two, second summand $\norm{Z^*}_{1,2,\eps}$, follows from line
one via the last estimate in~(\ref{eq:4.1.4}) with constant $2$.
To incorporate the first summand $\eps^{1/2}\norm{Z^*}_{0,\infty,\eps}$
simply use~(\ref{eq:cor:infty}).

To prove inequality three in~(\ref{eq:thm:4.4.4})
multiply the ambient linear estimate~(\ref{eq:amblinest}),
for $D^\eps_q$, by $\eps$
to obtain that
\begin{equation*}
\begin{split}
   \norm{dH|_qX^*}+\eps\norm{\ell^*}
   +\eps\norm{\Babla{s} X^*}+\eps^2\norm{(\ell^*)^\prime}
   &\stackrel{(\ref{eq:amblinest})}{\le}
   \eps c_a\norm{D^\eps_q Z^*}_{0,2,\eps}
   +\eps c_a\norm{X^*} .
\end{split}
\end{equation*}
Combining this with~(\ref{eq:zeta^*}) proves inequality three
in~(\ref{eq:thm:4.4.4}). Inequality four holds by
estimate four in~(\ref{eq:4.1.4}).
This concludes the proof of Theorem~\ref{thm:KeyEst-thm.3.3}.
\end{proof}

\boldmath
%%%%%%%%%%%%%%%%%%%%%%%%%%%%%%%%%%%
%%%%%%%%%%%%%%%%%%%%%%%%%%%%%%%%%%%
%%%%%%%%%%%%%%%%%%%%%%%%%%%%%%%%%%%
%%%%%%%%%%%%%%%%%%%%%%%%%%%%%%%%%%%
%%%%%%%%%%%%%%%%%%%%%%%%%%%%%%%%%%%
\section[Implicit function theorem I -- detect ambient solutions]
{Implicit function theorem I -- Ambience}
\label{sec:IFT}
\unboldmath

\begin{theorem}[IFT I -- Existence]\label{thm:existence-findim}
Assume $(f,g)$ is Morse-Smale.
Then there are constants $c>0$ and $\eps_0\in(0,1]$
such that the following holds.
For every $\eps\in(0,\eps_0]$, every pair $x^\mp\in\Crit f$
of index difference one, and every $q\in\Mm^0_{x^-,x^+}$,
there exists a pair $(u^\eps,\tau^\eps)\in\Mm^\eps_{x^-,x^+}$ of the form
\[
   u^\eps=\Exp_q X,\qquad
   \tau^\eps=\chi(q)+\ell,\qquad
   (X,\ell)\in\im(D_q^\eps)^* ,
\]
where the difference $Z=(X,\ell)\in C^\infty(\R,q^*TM\oplus\R)$
is smooth and bounded by
\begin{equation}\label{eq:existence-p}
   \norm{Z}_{1,2,\eps}
   \le\norm{X}+\eps\norm{\ell }
   +\eps\norm{\Babla{s} X}+\eps^2\norm{\ell ^\prime}
   \le c\eps^2
\end{equation}
and by
\begin{equation}\label{eq:existence-infty}
   \norm{X}_\infty\le c\eps^{3/2},\qquad
   \norm{\ell}_\infty\le c\eps^{1/2} .
\end{equation}
\end{theorem}

\begin{theorem}[IFT I -- Uniqueness]\label{thm:uniqueness-findim}
Assume $(f,g)$ is Morse-Smale.
% Fix a constant $C>0$.
Then there are constants $\delta_0,\eps_0\in(0,1]$
such that, for any $\eps\in(0,\eps_0]$,
any pair $x^\mp\in\Crit f$ of index difference one,
and any $q\in\Mm^0_{x^-,x^+}$ the following holds. If
\[
   (X_i,\ell_i)\in\im(D_q^\eps)^*,\qquad
   \norm{X_i}_\infty\le\delta_0\sqrt{\eps} ,
%   ,\qquad
%   \norm{\ell_i}_\infty<\infty ,
\]
for $i=1,2$ and both pairs of maps $(u_1^\eps,\tau_1^\eps)$ and
$(u_2^\eps,\tau_2^\eps)$ defined by
% (both based at the same $q$)
\begin{equation}\label{eq:uniqueness-hyp}
%   (u_1^\eps, \tau_1^\eps)
%   =\left(\Exp_q X_1, \chi(q)+\ell_1\right)
%   ,\quad
%   (u_2^\eps, \tau_2^\eps)
%   =\left(\Exp_q X_2, \chi(q)+\ell_2\right)
%
   u_i^\eps:=\Exp_q X_i,\qquad
   \tau_i^\eps:=\chi(q)+\ell_i ,
\end{equation}
belong to the moduli space $\Mm^\eps_{x^-,x^+}$, then they are equal
$(u_1^\eps,\tau_1^\eps)=(u_2^\eps,\tau_2^\eps)$.
\end{theorem}

Observe that each pair $(X_i,\ell_i)$ is smooth by
hypothesis~(\ref{eq:uniqueness-hyp}). Hence, by
exponential decay of the derivatives of $(u_i^\eps,\tau_i^\eps)$,
each pair $(X_i,\ell_i)$ belongs to $W^{k,2}(\R,q^*TM\oplus\R)$
for every integer $k\ge 0$.

\begin{definition}\label{thm:Teps-findim}
Assume $(f,g)$ is Morse-Smale.
Choose constants $\eps_0,\delta_0\in(0,1]$ and $c>0$
such that the assertions of Theorem~\ref{thm:existence-findim}
and~\ref{thm:uniqueness-findim} hold with these constants.
Shrink $\eps_0$ so that $c\eps_0<\delta_0$.
Given a pair $x^\mp\in\Crit f$ of index difference one, define
for $\eps\in(0,\eps_0)$ the map
\begin{equation}\label{eq:Teps-findim}
   \Tt^\eps\colon\Mm^0_{x^-,x^+}\to\Mm^\eps_{x^-,x^+},\quad
   q\mapsto (u^\eps,\tau^\eps)
   :=\bigl(\Exp_q X,\chi(q)+\ell\bigr) ,
\end{equation}
where the pair $(X,\ell)\in\im(D_q^\eps)^*$
is chosen such that~(\ref{eq:existence-p})
and~(\ref{eq:existence-infty}) are satisfied
and $(\Exp_qX,\chi(q)+\ell)\in\Mm^\eps_{x^-,x^+}$.
Such a pair exists, by Theorem~\ref{thm:existence-findim},
and is unique, by Theorem~\ref{thm:uniqueness-findim}.
The map $\Tt^\eps$ is time shift equivariant.
\end{definition}

\begin{lemma}[Injectivity]\label{le:injectivity}
Assume $(f,g)$ is Morse-Smale.
Then there is a constant $\eps_0\in(0,1]$, such that
for every $\eps\in(0,\eps_0]$ and every pair $x^\mp\in\Crit f$ of
index difference one, the map
$\Tt^\eps\colon\Mm^0_{x^-,x^+}\to\Mm^\eps_{x^-,x^+}$ is injective.
\end{lemma}

\begin{proof}
As $\Sigma$ is compact, the index difference is $1$, and the metric
is Morse-Smale, the moduli space $\widetilde\Mm^0_{\mp}
:=\Mm^0_{x^-,x^+}/\R$ is a finite set.
So the smallest distance
\[
   d_{\rm min}
   :=\min_{[q_1]\not=[q_2]\in \widetilde\Mm^0_{\mp}}
   \;\sup_{s\in\R}\;\inf_{t\in\R}\;
   \dist(q_1(s),q_2(t)))
   >0
\]
is positive.
Choose the constant $\eps_0>0$ in Theorem~\ref{thm:existence-findim}
smaller if necessary such that $2c{\eps_0}^{3/2}<d_{\rm min}$.
By construction of $\Tt^\eps$, for $\eps\in(0,\eps_0)$,
an element $\Tt^\eps(q_1)=\Tt^\eps(q_2)$
lies in both radius $c\eps^{3/2}$ balls, the one about $q_1$ and the one
about $q_2$. Thus we must have $[q_1]=[q_2]$
since otherwise these two balls, by definition of $d_{\rm min}$,
would be disjoint.
But $[q_1]=[q_2]$ means that there exists $\sigma\in\R$ such that
$q_1=\sigma_* q_2:=q_2(\cdot+\sigma)$.
Since $\Tt^\eps$ is time shift invariant
we have $\Tt^\eps(q_1)=\sigma_*\Tt^\eps(q_2)=\sigma_*\Tt^\eps(q_1)$.
This implies $\sigma=0$, hence $q_1=q_2$.
\end{proof}

To prove Theorem~\ref{thm:existence-findim} we carry out a modified
Newton iteration to detect a zero of $\Ff^\eps$ near an approximate zero
for which we choose the pair $(q,\chi(q))$ with $q\in\Mm^0_{x^-,x^+}$.
The first step is to define a suitable map between Banach spaces
for which we choose the local trivialization
$\Ff^\eps_q:=\Ff^\eps_{q,\chi(q)}$, see~(\ref{eq:eps-triv}).
In this model the origin corresponds to our approximate zero.
One finds a true zero nearby if three conditions are satisfied.
Firstly, a small initial value $\Ff^\eps_q(0)$ where smallness will be taken care
of by the weights in the $(0,2,\eps)$ norm.
Secondly, a uniformly bounded right inverse $R^\eps_q$ of
$D^\eps_q=d\Ff^\eps_q(0)$ which holds due to the key
estimate~(\ref{eq:thm:4.4.4}).
Thirdly, we need quadratic estimates to gain control on the
variation of the derivative $d\Ff^\eps_q(Z)$ for $Z$ near the origin.

\boldmath
%%%%%%%%%%%%%%%%%%%%%%%%%%%%%%%%%%%
%%%%%%%%%%%%%%%%%%%%%%%%%%%%%%%%%%%
%%%%%%%%%%%%%%%%%%%%%%%%%%%%%%%%%%%
\subsection{Quadratic estimates}
\unboldmath

Pick a map $q\in W^{1,2}(\R,\Sigma)$.
Consider the map $z=(q,\chi(q))\in W^{1,2}(\R,M\times\R)$ and
let $Z=(X,\ell) \in W^{1,2}(\R,q^*\Oo\oplus\R)$ be a vector field
along it.\footnote{
  For $q\in\Sigma$ let $\Oo_q$ be the maximal
  domain of the exponential map $\Exp_q\colon T_qM\to M$.
  The subset $\Oo_q$ is open and star-shaped about $0$;
  see e.g.~\cite[\S 5 4.\,Cor.]{oneill:1983a}.
  The maximal domain of $\Exp\colon T_\Sigma M\to M$
  is an open neighborhood $\Oo\subset T_\Sigma M$ of the zero section
  with $\Oo\cap T_q M=\Oo_q$.
  }
Denote parallel transport in $(M,G)$ along the geodesic
$
   r\mapsto \Exp_{q(s)}(rX(s))
$
by
\begin{equation}\label{eq:par-trans}
   \Phi=\Phi_q(X)\colon T_qM\supset\Oo_q\to T_{E(q,X)} M
   ,\qquad
   \Gamma_0=\Exp_q(X),
\end{equation}
pointwise for $s\in\R$. A 
trivialization of the ambient section $\Ff^\eps$ is defined by
\begin{equation}\label{eq:triv-section-eps}
   \Ff^\eps_q(X,\ell)
   =\begin{pmatrix}
      \Phi_q^{-1}(X)\left(
      \p_s \Gamma_0+\Babla{} F|_{\Gamma_0}
      +(\chi(q)+\ell)\Babla{} H|_{\Gamma_0}\right)
   \\
      (\chi(q)+\ell)^\prime+\eps^{-2}H|_{\Gamma_0}
   \end{pmatrix}
\end{equation}
for every vector field $(X,\ell)\in W^{1,2}(\R,q^*\Oo\oplus\R)$.
To compute the derivative of the trivialization $\Ff^\eps_q$
at a point $Z=(X,\ell)$ in direction $\zeta=(\hat X,\hat\ell)$
abbreviate
\[
   \Phi_r:=\Phi_q(X+r\hat X),\qquad
   \Gamma_r:=E(q,X+r\hat X).
\]
Then $\left.\frac{d}{dr}\right|_0 \Gamma_r=E_2(q,X)\hat X$ and
the derivative is given by
\begin{equation*}
\begin{split}
   &d \Ff^\eps_q (X,\ell)
   \begin{pmatrix} \hat X\\\hat \ell\end{pmatrix}
   :=\tfrac{d}{dr}\bigr|_0
   \Ff^\eps_q (X+r\hat X,\ell+r\hat \ell)
\\
   &\stackrel{{\color{gray} 1}}{=}
   \frac{d}{dr}\Bigr|_0
   \begin{pmatrix}
   \Phi_r^{-1}\left(\p_s\Gamma_r+\Babla{} F|_{\Gamma_r}\right)
   +(\chi(q)+\ell+r\hat\ell) \Phi_r^{-1}\Babla{} H|_{\Gamma_r}
   \\
     (\chi(q)+\ell+r\hat\ell)^\prime+\eps^{-2}H|_{\Gamma_r}
   \end{pmatrix}
\\
   &\stackrel{{\color{gray} 2}}{=}
   \begin{pmatrix}
   \left.\frac{d}{dr}\right|_0
   \left(\Phi_r^{-1}
   \left(\p_s\Gamma_r+\Babla{} F|_{\Gamma_r}\right)\right)
   +\hat\ell\Phi_0^{-1}\Babla{} H|_{\Gamma_0}
   +(\chi(q)+\ell)\left.\frac{d}{dr}\right|_0
   \left(\Phi_r^{-1}\Babla{} H|_{\Gamma_r}\right)
   \\
   \hat\ell^\prime+\eps^{-2}dH|_{\Gamma_0} E_2(q,X)\hat X
   \end{pmatrix}
\\
   &\stackrel{{\color{gray} 3}}{=}
   \begin{pmatrix}
   \left.\frac{d}{dr}\right|_0 \Phi_r^{-1}\p_s\Gamma_r
   +\left.\frac{d}{dr}\right|_0 \Phi_r^{-1}\Babla{} F|_{\Gamma_r}
   +(\chi(q)+\ell) \left.\frac{d}{dr}\right|_0 \Phi_r^{-1}\Babla{} H|_{\Gamma_r}
   +\hat\ell\Phi_0^{-1}\Babla{} H|_{\Gamma_0}
   \\
   \hat\ell^\prime+\eps^{-2}dH|_{\Gamma_0} E_2(q,X)\hat X
   \end{pmatrix}
\end{split}
\end{equation*}
where step~1 is by definition of $\Ff^\eps_q$ and step~3 by
linearity of parallel transport.

\begin{proposition}[Quadratic estimate I]
\label{prop:quadest-I}
There is a constant $\delta\in(0,1]$ with the following significance.
For every $c_0>0$ there is a constant $c>0$ such~that the following is true.
Let $q\in W^{1,2}(\R,\Sigma)$ be a map and
$Z=(X,\ell)$, $\zeta=(\hat X,\hat\ell)\in W^{1,2}(\R,q^*TM\times\R)$
be two vector fields along $z=(q,\chi(q))$ such~that
\[
   \norm{\p_s q}_\infty+\norm{\chi(q)}_\infty\le c_0
   ,\qquad
   \norm{X}_\infty+\norm{\hat X}_\infty\le \delta .
\]
Then the components $F$ and $f$ of the vector field along $z$,
defined by
\begin{equation}\label{eq:quadest-I-diff}
   \Ff^\eps_q(Z+\zeta)-\Ff^\eps_q (Z)-d\Ff^\eps_q (Z)\zeta
   =:\begin{pmatrix} F\\f \end{pmatrix} ,
\end{equation}
satisfy the inequalities
\begin{equation}\label{eq:quadest-I}
\begin{aligned}
   \norm{F}
   &\le 
   c\norm{\hat X}_\infty\left(\norm{\hat X}
      +\norm{\hat\ell}
      +\norm{\Babla{s}\hat X}\cdot\norm{\hat X}_\infty\right)
   \\
   &\quad
   +c\norm{X}_\infty
   \left(
   \norm{\hat X}
   +\norm{\Babla{s}\hat X}\cdot\norm{X}_\infty
   \right)
   +c\norm{\ell}_\infty \norm{\hat X}_\infty \norm{\hat X}
   \\
   &\quad
   +c \norm{\hat X}_\infty \norm{\Babla{s} X}
      \left(\norm{\hat X}_\infty+\norm{X}_\infty\right)
\\
   \eps\norm{f}
   &\le c\eps^{-1}\norm{\hat X}_\infty\norm{\hat X}
\end{aligned}
\end{equation}
whenever $\eps>0$.
\end{proposition}

By compactness of $\Sigma$ the injectivity radius of the Riemannian
vector bundle $(T_\Sigma M,G)$ is positive.
The choice $\delta=\iota(T_\Sigma M)/2>0$ takes care that $X$ and
$\hat X$ are in the domain of $\Exp$.

\begin{proposition}[Quadratic estimate II]
\label{prop:quadest-II}
There is a constant $\delta\in(0,1]$ with~the following
significance.  For
any $c_0>0$ there is a constant $c>0$~such that the following is true.
Let $q\in W^{1,2}(\R,\Sigma)$ be a map and
$Z=(X,\ell)$, $\zeta=(\hat X,\hat\ell)\in W^{1,2}(\R,q^*TM\times\R)$
be two vector fields along $z=(q,\chi(q))$ such~that
\[
   \norm{\p_s q}_\infty+\norm{\chi(q)}_\infty\le c_0
   ,\qquad
   \norm{X}_\infty\le \delta .
\]
Then the components $F$ and $f$ of the vector field along $z$,
defined by
\begin{equation}\label{eq:quadest-II-diff}
   d\Ff^\eps_q (Z)\zeta-d\Ff^\eps_q (0)\zeta
   =:\begin{pmatrix} \FF\\\ff \end{pmatrix} ,
\end{equation}
satisfy the inequalities
\begin{equation}\label{eq:quadest-II}
\begin{aligned}
   \norm{\FF}
   &\le c\norm{X}_\infty
   \left(
   \norm{\hat X}+\norm{\hat\ell}
   +\norm{\Babla{s}\hat X}\cdot\norm{X}_\infty
   \right)
   \\
   &\quad
   +c \norm{\ell}_\infty \norm{\hat X}
   +c\norm{X}_\infty\norm{\hat X}_\infty\norm{\Babla{s} X}
\\
   \eps\norm{\ff}
   &\le c\eps^{-1} \norm{X}_\infty\norm{\hat X}
\end{aligned}
\end{equation}
whenever $\eps>0$.
\end{proposition}

\boldmath
%%%%%%%%%%%%%%%%%%%%%%%%%%%%%%%%%%%
%%%%%%%%%%%%%%%%%%%%%%%%%%%%%%%%%%%
\subsubsection*{Tools}
\unboldmath

\begin{theorem}[Exponential map -- derivatives]
\label{thm:exp-map}
Let $u$ be a point in a Riemannian manifold $M$ and $X\in\Oo_u$
a tangent vector. Then there are linear maps
$$
  E_i (u,X):T_uM\to T_{\Exp_uX}M ,\qquad
  E_{ij}(u,X):T_uM\times T_uM\to T_{\Exp_uX}M
$$
for $i,j\in\{1,2\}$ such that
the following is true.
If $u:\R\to M$ is a smooth curve
and $X,Y$ are smooth vector fields 
along $u$ with $X(s)\in\Oo_{u(s)}$
for every $s$, then the maps $E_i$ and $E_{ij}$
are characterized (uniquely determined)
by the identities
\begin{equation*} %\label{eq:exponential-identity}
\begin{split}
     \frac{d}{ds}\Exp_u(X)
    &=E_1(u,X)\p_su
     +E_2(u,X)\Babla{s}X
    \\
     \Babla{s}\left( E_1(u,X)Y\right)
    &=E_{11}(u,X)\left(Y,\p_s u\right)
      +E_{12}(u,X)\left(Y,\Babla{s}X\right)
      +E_1(u,X)\Babla{s}Y
    \\
     \Babla{s}\left( E_2(u,X)Y\right)
    &=E_{21}(u,X)\left(Y,\p_s u\right)
      +E_{22}(u,X)\left(Y,\Babla{s}X\right)
      +E_2(u,X)\Babla{s}Y.
\end{split}
\end{equation*}
Here $\Babla{}$ is the Levi-Civita connection.\footnote{
  Our convention for derivatives, example $\p_j E_i$, is to put both,
  the derivative index $j$ and the arising new linear factor to the
  right. This way index order and linear factor order coincide,
  example $\p_j (E_i(x_i,x_j) X_i)=E_{ij} (x_i,x_j)\left(X_i,X_j\right)$.
  }
Furthermore, there are the identities
\begin{equation}\label{eq:E_ij(0)}
     E_1(u,0)=E_2(u,0)=\1
     ,\quad
     E_{11}(u,0)=E_{21}(u,0)=E_{22}(u,0)=0 .
\end{equation}
For 
all $u\in M$, $X\in\Oo_u$, and $Y,Z\in T_uM$
there are the symmetry properties
\begin{equation*} %\label{eq:E12-symmetry}
\begin{split}
     E_{12}(u,X)\left(Y,Z\right)
     =E_{21}(u,X)\left(Z,Y\right)
   \quad
     E_{22}(u,X)\left(Y,Z\right)
     =E_{22}(u,X)\left(Z,Y\right)
\end{split}
\end{equation*}
and the identity
%\begin{equation}\label{eq:E11-E11}
$
     E_{11}(u,X)\left(Y,Z\right)
     -E_{11}(u,X)\left(Z,Y\right)
     =E_{2}(u,X) \widebar R(Y,Z)X
$
%\end{equation}
where $\widebar R$ is the Riemannian curvature operator.
\end{theorem}

\begin{proof}
El\u{\i}asson~\cite{Eliasson:1967a}.
For details see also~\cite[sec.~3.1.1]{gaio:1999a}
or~\cite{Weber:2022a}.
\end{proof}

The following lemma is a major technical tool in the proof of the
pointwise quadratic estimates. The proof is standard, for details see
e.g.~\cite[Le.\,5.0.9]{weber:1999a}.
Note that the lemma remains valid for covariant derivatives
$\widebar D=d+\Gamma\,\hat X$ since the Christoffel symbol
$\Gamma$ arrives together with the direction $\hat X$.

\begin{lemma}\label{le:5.0.9}
Let $m,n\in\N$ and $ h\in C^2(\R^m,\R^n)$.
Then for any $\delta>0$ there exists a continuous function
$c_\delta\in C^0(\R^m,\R^+)$ such that
\begin{itemize}\setlength\itemsep{0ex} 
\item[\rm i)]
   $\abs{ h(X+\hat X)- h(X)}\le c_\delta(\hat X)\abs{\hat X}$
\item[\rm ii)]
   $\abs{ h(X+\hat X)- h(X)-d h(X)\,\hat X}\le c_\delta(\hat X)\abs{\hat X}^2$
\end{itemize}
for all $X\in\R^m$ with $\abs{X}\le\delta$ and all $\hat X\in\R^m$.
\end{lemma}

\boldmath
%%%%%%%%%%%%%%%%%%%%%%%%%%%%%%%%%%%
%%%%%%%%%%%%%%%%%%%%%%%%%%%%%%%%%%%
\subsubsection*{Proofs}
\unboldmath

\begin{proof}[Proof of Proposition~\ref{prop:quadest-I}.]
Write $F=F_1+F_2+F_3+F_4$ and $f=f_1+f_2$
where the summands $F_i$ and $f_j$ are defined now.
The summand $F_1$ is defined by
\begin{equation*}
\begin{split}
   F_1:
% 1
   &=\Phi_q^{-1}(X+\hat X)\tfrac{d}{ds} E(q,X+\hat X)
   -\Phi _q^{-1}(X)\tfrac{d}{ds} E(q,X)
   \\
   &\quad
   -\left(\left.\tfrac{\widebar D}{dr}\right|_0 \Phi_q^{-1}(X+r\hat X)\right)
   \tfrac{d}{ds}E(q,X)
   -\Phi_q^{-1}(X)\left.\tfrac{\widebar D}{dr}\right|_0
      \tfrac{d}{ds}E(q,X+r\hat X)
\\
%2
   &\stackrel{{\color{gray} 2}}{=}
   \Phi_q^{-1}(X+\hat X)
   \left( E_1(q,X+\hat X)\p_sq+E_2(q,X+\hat X)
   \bigl(\Babla{s} X+\Babla{s}\hat X\bigr)\right)
   \\
   &\quad-\Phi_q^{-1}(X) \left(E_1(q,X)\p_sq + E_2(q,X)\Babla{s} X\right)
   \\
   &\quad-\widebar D\Phi_q^{-1}|_X\left(E_1(q,X)\p_sq, \hat X\right)
   -\widebar D\Phi_q^{-1}|_X\left(E_2(q,X)\Babla{s} X, \hat X\right)
   \\
   &\quad-\Phi_q^{-1}(X)\left(E_{12}(q,X)\bigl(\p_s q,\hat X\bigr)
   +E_{22}(q,X)\bigl(\Babla{s} X,\hat X\bigr)
   +E_2(q,X)\Babla{s}\hat X\right)
\\
% 3
   &\stackrel{{\color{gray} 3}}{=}
   \Phi_q^{-1} (X+\hat X) E_1(q,X+\hat X)\p_sq
   -\Phi_q^{-1}(X) E_1(q,X)\p_sq
   \\
   &\qquad\quad
   -\widebar D\Phi_q^{-1}|_X\left(E_1(q,X)\p_sq, \hat X\right)
   \\
   &\quad+\Phi_q^{-1}(X+\hat X) E_2(q,X+\hat X)\Babla{s} X
   -\Phi_q^{-1}(X) E_2(q,X)\Babla{s} X
   \\
   &\qquad\quad
   -\widebar D\Phi_q^{-1}|_X\left(E_2(q,X)\Babla{s} X, \hat X\right)
   \\
   &\quad
   -\Phi_q^{-1}(X) E_{12}(q,X)\bigl(\p_sq,\hat X\bigr)
   -\Phi_q^{-1}(X) E_{22}(q,X)\bigl(\Babla{s} X,\hat X\bigr)
   \\
   &\quad +\left({\color{gray}\Phi_q^{-1}(X+\hat X) E_2(q,X+\hat X)-\1\;}
   +\1-\Phi_q^{-1}(X) E_2(q,X)\right)\Babla{s}\hat X .
\end{split}
\end{equation*}
To get identity 2 we carried out the derivatives with respect to $s$
and $r$ using the characterizing identities from
Theorem~\ref{thm:exp-map}. 
In identity 3 we only reordered the summands.
The estimate for $\norm{F_1}$ is obtained by
applying pointwise Lemma~\ref{le:5.0.9} followed by integration.
More precisely, for the first triple of summands one applies part~ii)
of the lemma, same for the second triple.
To the next two summands apply part~i) individually.
For example define and note that
\[
   h(X):=\Phi_q^{-1}(X) E_{22}(q,X)
   ,\qquad
   h(0)\stackrel{(\ref{eq:E_ij(0)})}{=}0.
\]
Part~ii) also applies to the final line
where we added $-\1+\1=0$.
To deal with the second part of the final line
{\color{gray} (analogously part one)} define and note that
\begin{equation}\label{eq:par-exp-relation}
   h(X):=\Phi_q^{-1}(X) E_2(q,X)-\1
   ,\qquad
   h(0)\stackrel{(\ref{eq:E_ij(0)})}{=}0
   ,\qquad
   \widebar D h(0)X=0.
\end{equation}
It remains to show that the derivative vanishes, indeed
\begin{equation*}%\label{eq:par-exp-relation-2}
\begin{split}
   \widebar D h(0) X
   &=\left.\tfrac{\widebar D}{dr}\right|_0 h(rX)
   \\
   &=\widebar D\Phi_q^{-1}|_0\left(E_2(q,0)\cdot, X\right)
   +\Phi_q^{-1}(0)E_{22}(q,0)\left(\cdot, X\right)
   \\
   &=\left(\widebar D\Phi_q^{-1}|_0
   +E_{22}(q,0) \right) \left(\cdot,X\right)
   \\
   &=0.
\end{split}
\end{equation*}
The last step holds since both summands vanish individually,
namely $E_{22}(q,0)=0$ and a short calculation in local coordinates
shows that
\begin{equation}\label{eq:DPhi=0}
   \left(\left.\tfrac{\widebar D}{dr}\right|_0\Phi_q^{-1}(r\hat X)\right)^k_j
   =\left(\widebar D\Phi_q^{-1}|_0\bigl(\cdot,\hat X\bigr)\right)^k_j
   =\underbrace{\left.\tfrac{d}{dr}\right|_0\Phi_q^{-1}(r\hat X)_j^k}
      _{=-\Gamma_{ij}^k \hat X^i}
   +\Gamma_{ij}^k \hat X^i
   =0
\end{equation}
where the under-braced identity is Lemma~A.1.3 in~\cite{weber:1999a}.
Recall from the primer article
(remark in quadratic estimate section)
that $L^\infty$ norms should go preferably on the base point
$Z=(X,\ell)$, but never on derivatives.
As pointwise estimate for $F_1$, written in the same order as above,
we obtain
\begin{equation*}
\begin{split}
   \abs{F_1}
   &\le c_{\delta,\hat X}\norm{\p_s q}_\infty\abs{\hat X}^2
   +c_{\delta,\hat X}\abs{\hat X}^2\abs{\Babla{s} X}
   +c_{\delta,X}\norm{\p_s q}_\infty\abs{X}\cdot\abs{\hat X}
   \\
   &\quad
   +c_{\delta,X}\abs{\hat X}\cdot\abs{X}\cdot\abs{\Babla{s} X}
   +c_{\delta,X+\hat X}\abs{X+\hat X}^2\abs{\Babla{s}\hat X}
   +c_{\delta,X}\abs{X}^2\abs{\Babla{s}\hat X}
   \\
   &\le \tilde c_1\left(
   \abs{\hat X}^2(1+\abs{\Babla{s} X})
   +\abs{X}\cdot\abs{\hat X}(1+\abs{\Babla{s} X})
   +(\abs{X}^2+\abs{\hat X}^2\abs{\Babla{s}\hat X})
   \right)
\\
   \norm{F_1}
   &\le
   c_1\norm{\hat X}_\infty\left(\norm{\hat X}
      +\norm{\Babla{s}\hat X}\cdot\norm{\hat X}_\infty\right)
   +c_1\norm{X}_\infty\left(\norm{\hat X}
      +\norm{\Babla{s}\hat X}\cdot\norm{X}_\infty\right)
   \\
   &\quad+c_1\norm{\hat X}_\infty
   \left(\norm{\hat X}_\infty+\norm{X}_\infty\right)
   \norm{\Babla{s} X}
\end{split}
\end{equation*}
for suitable positive constants $\tilde c_1$ and $c_1$.
In step~2 of the pointwise estimate we used that
$\abs{X+\hat X}^2\le 2\abs{X}^2+2\abs{\hat X}^2$.
The $L^2$ estimate for $F_1$ follows by squaring the estimate
for $\abs{F_1}$, integrate the result, and pull out $L^\infty$ norms.
The summand $F_2$ is defined and then, via Lemma~\ref{le:5.0.9}~ii),
estimated by
\begin{equation*}
\begin{split}
   F_2:
   &=
   \Phi_q^{-1}(X+\hat X)\Babla{} F|_{E(q,X+\hat X)}
%   \underbrace{\Phi_q^{-1}(X+\hat X)\Babla{} F|_{E(q,X+\hat X)}}_{=:h(\hat X)}
   -\Phi_q^{-1}(X)\Babla{} F|_{E(q,X)}
   -\left.\tfrac{d}{dr}\right|_0\left(\Phi_r^{-1}\Babla{}F|_{\Gamma_r}\right)
   \\
   &\quad
   =h(\hat X)-h(0)-dh(0)\hat X
   ,\quad h(\hat X):=\Phi_q^{-1}(X+\hat X)\Babla{} F|_{E(q,X+\hat X)}
\\
   \norm{F_2}
   &\le %c_{\delta,\hat X,\norm{F}_{C^2}}
   c_2\norm{\hat X}_\infty\norm{\hat X}
\end{split}
\end{equation*}
for suitable $c_2>0$.
Analogous to $F_2$ we define and treat the summand $F_3$ by
\begin{equation*}
\begin{split}
   F_3:
   &=(\chi(q)+\ell)
   \Bigl(
   \Phi_q^{-1}(X+\hat X)^{-1}\Babla{} H|_{E(q,X+\hat X)}
   -\Phi_q^{-1}(X)^{-1}\Babla{} H|_{E(q,X)}
   \\
   &\quad
   -\left.\tfrac{d}{dr}\right|_0
      \left(\Phi_q^{-1}(X+r\hat X)\Babla{} H|_{E(q,X+r\hat X)}\right)
   \Bigr)
\\
   \norm{F_3}
   &\le c_3 %c_{\delta,\hat X,\norm{H}_{C^1}}
   \left(\norm{\hat X}_\infty \norm{\hat X}
   +\norm{\ell}_\infty\norm{\hat X}\cdot\norm{\hat X}_\infty\right) 
\end{split}
\end{equation*}
for suitable $c_3>0$.
For suitable $c_4>0$ we define and treat summand $F_4$ by
\begin{equation*}
\begin{split}
   F_4:
   &=\hat\ell
   \left(
   \Phi_q^{-1}(X+\hat X)\Babla{} H|_{E(q,X+\hat X)}
   -\Phi_q^{-1}(X)^{-1}\Babla{} H|_{E(q,X)}
   \right),
\\
   \norm{F_4}
   &\le c_4 %c_{\delta,\hat X,\norm{H}_{C^1}}
   \norm{\hat X}_\infty \norm{\hat\ell}.
\end{split}
\end{equation*}
Summand $f_1$ is defined by
$
   f_1:
   =(\chi(q)+\ell+\hat\ell)^\prime-(\chi(q)+\ell)^\prime-\hat\ell^\prime
   =0
$
and $f_2$ by
\begin{equation*}
\begin{split}
   f_2:
   &=\eps^{-2}\left(
   H|_{E(q,X+\hat X)}-H|_{E(q,X)}-dH|_{E(q,X)} E_2(q,X)\hat X
   \right),
\\
   \norm{f_2}
   &\le\eps^{-2} c_5 %c_{\delta,\hat X,\norm{H}_{C^1}}
   \norm{\hat X}_\infty\norm{\hat X} .
\end{split}
\end{equation*}
This concludes the proof of Proposition~\ref{prop:quadest-I}
(Quadratic Estimate I).
\end{proof}

\begin{proof}[Proof of Proposition~\ref{prop:quadest-II}]
The derivative of $\Ff^\eps_q$ at $0$ in direction
$\zeta=(\hat X,\hat\ell)$ is
\begin{equation*}
\begin{split}
   &d \Ff^\eps_q (0,0)
   \begin{pmatrix} \hat X\\\hat \ell\end{pmatrix}
\\
   &\stackrel{(\ref{eq:deriv-triv-origin})}{=}
   \begin{pmatrix}
   \left.\frac{d}{dr}\right|_0 \Phi _q^{-1} (r\hat X)
   \left(
   \p_sE(q,r\hat X)+\Babla{} F|_{E(q,r\hat X)}+\chi(q)\Babla{} H|_{E(q,r\hat X)}
   \right)
   +\hat\ell\Babla{} H|_q
   \\
   \hat\ell^\prime+\eps^{-2}dH|_q\hat X
   \end{pmatrix} .
\end{split}
\end{equation*}

Write $F=F_1+F_2+F_3+F_4$ and $f=f_1+f_2$ where the summands
$F_i$ and $f_j$ are defined in what follows. 
The summand $F_1$ is defined by
\begin{equation*}
\begin{split}
   F_1:
   &=
   \left(\left.\tfrac{d}{dr}\right|_0 \Phi _q^{-1}(X+r\hat X)\right)
      \tfrac{d}{ds}E(q,X)
   -\left(\left.\tfrac{d}{dr}\right|_0 \Phi_q^{-1}(r\hat X)\right)
   \tfrac{d}{ds}E(q,0)
   \\
   &\quad
   +\Phi_q^{-1}(X)\left.\tfrac{\widebar D}{dr}\right|_0
      \tfrac{d}{ds}E(q,X+r\hat X)
   -\underline{\Phi_q^{-1}(0)\left.\tfrac{\widebar D}{dr}\right|_0
      \tfrac{d}{ds}E(q,r\hat X)}
\\
   &\stackrel{{\color{gray} 2}}{=}
   \widebar D\Phi_q^{-1}|_X
   \left(E_1(q,X)\p_s q+E_2(q,X)\Babla{s} X, \hat X\right)
   -\widebar D\Phi_q^{-1}|_0\left(\p_sq,\hat X\right)
   -\underline{\Babla{s}\hat X}
   \\
   &\quad+\Phi_q^{-1}(X)
   \left(
   E_{12}(q,X)\bigl(\p_sq,\hat X\bigr)
   +E_{22}(q,X)\bigl(\Babla{s} X,\hat X\bigr)
   +E_2(q,X)\Babla{s}\hat X
   \right)
\\
   &\stackrel{{\color{gray} 3}}{=}
   \widebar D\Phi_q^{-1}|_X\left(E_1(q,X)\p_s q, \hat X\right)
   -\widebar D\Phi_q^{-1}|_0\left(\p_sq,\hat X\right)
   \\
   &\quad
   +\Phi_q^{-1}(X) E_{12}(q,X)\bigl(\p_sq,\hat X\bigr)
  +\left(\Phi_q^{-1}(X) E_2(q,X)-\underline{\1}\right)\Babla{s}\hat X
   \\
   &\quad
   +\widebar D\Phi_q^{-1}|_X \left(E_2(q,X)\Babla{s} X,\hat X\right)
  +\Phi_q^{-1}(X) E_{22}(q,X)\bigl(\Babla{s} X,\hat X\bigr) .
\end{split}
\end{equation*}
To get identity 2 we carried out the derivatives with respect to $s$
and $r$ using the characterizing identities from
Theorem~\ref{thm:exp-map}. 
In identity 3 we only reordered the summands.
The estimate for $\norm{F_1}$ is obtained by
applying pointwise Lemma~\ref{le:5.0.9} followed by integration.
One uses the same techniques as for term $F_1$ in quadratic
estimate~I, in particular~(\ref{eq:par-exp-relation}) and the
identities $E_1(q,0)=\1=E_2(q,0)$ and $E_{12}(q,0)=0=E_{22}(q,0)$.
Note that the last but one term
\[
   g(X)
   :=
   \widebar D\Phi_q^{-1}|_X\left(E_2(q,X)\Babla{s} X,\hat X\right)
   ,\qquad 
   g(0)=0,
\]
vanishes at the origin as we saw earlier in~(\ref{eq:DPhi=0}).
Recall from the primer article
(remark in quadratic estimate section)
that $L^\infty$ norms should go preferably on the base point
$Z=(X,\ell)$, but not on derivatives.
We get the estimate
\begin{equation*}
\begin{split}
   \Norm{F_1}
   \le c_1\left(\norm{\p_sq}_\infty\norm{X}_\infty\norm{\hat X}
   +\norm{X}_\infty^2\norm{\Babla{s}\hat X}
   +\norm{\hat X}_\infty \norm{X}_\infty\norm{\Babla{s} X}
   \right).
\end{split}
\end{equation*}
The summand $F_2$ is defined, and then estimated, by
\begin{equation*}
\begin{split}
   F_2:
   &=
   \left.\tfrac{d}{dr}\right|_0
      \left(\Phi_q^{-1}(X+r\hat X)\Babla{}F|_{E(q,X+r\hat X)}\right)
   -\left.\tfrac{d}{dr}\right|_0
   \left(\Phi_q^{-1}(r\hat X)\Babla{} F|_{E(q,r\hat X)}\right)
   \\
   &=\widebar D\Phi_q^{-1}|_X\left(\Babla{} F|_{E(q,X)},\hat X\right)
   -\widebar D\Phi_q^{-1}|_0\left(\Babla{} F|_q,\hat X\right)
   \\
   &\quad
   +\Phi_q^{-1}(X)\widebar D \Babla{} F|_{E(q,X)}\,E_2(q,X) \hat X
   -\widebar D \Babla{} F|_q\,\hat X ,
\\
   \norm{F_2}
   &\le c_2\norm{X}_\infty\norm{\hat X}.
\end{split}
\end{equation*}
Summand $F_3$ is defined, and then estimated, by
\begin{equation*}
\begin{split}
   F_3:
   &=\ell
   \left.\tfrac{d}{dr}\right|_0
   \left(\Phi_q^{-1}(X+r\hat X)\Babla{} H|_{E(q,X+r\hat X)}\right)
   \\
   &\quad
   +\chi(q) \left.\tfrac{d}{dr}\right|_0\left(
   \Phi_q^{-1}(X+r\hat X)\Babla{} H|_{E(q,X+r\hat X)}
   -
   \Phi_q^{-1}(r\hat X)\Babla{} H|_{E(q,r\hat X)}
   \right),
\\
   &=\ell \widebar D\Phi_q^{-1}|_X
      \left(\Babla{} H|_{E(q,X)},\hat X\right)
   +{\color{red} \ell\,} \Phi_q^{-1}(X)\widebar D \Babla{} H|_{E(q,X)}\,E_2(q,X)
   {\color{red}\, \hat X}
   \\
   &\quad+\chi(q) \widebar D\Phi_q^{-1}|_X
      \left(\Babla{} H|_{E(q,X)},\hat X\right)
   -\chi(q) \widebar D\Phi_q^{-1}|_0\left(\Babla{} H|_q,\hat X\right)
   \\
   &\quad
   +\chi(q) \Phi_q^{-1}(X)\widebar D \Babla{} H|_{E(q,X)}\,E_2(q,X) \hat X
   -\chi(q)\widebar D \Babla{} H|_q\,\hat X ,
\\
   \norm{F_3}
   &\le c_3\Bigl(
   \norm{X}_\infty\norm{\ell}_\infty+{\color{red}\norm{\ell}_\infty\,}
   +\norm{\chi(q)}_\infty \norm{X}_\infty
   \Bigr)\norm{\hat X} .
\end{split}
\end{equation*}
Summand $F_4$ is defined by
\begin{equation*}
\begin{split}
   F_4:
   &=\hat\ell
   \left(\Phi_q^{-1}(X)\Babla{} H|_{E(q,X)}-\Babla{} H|_q\right),
\\
   \norm{F_4}
   &\le c_4\norm{X}_\infty\norm{\hat\ell}.
\end{split}
\end{equation*}
Summand $f_1$ is defined by
$
   f_1:
   =\hat\ell^\prime-\hat\ell^\prime
   =0
$
and $f_2$ by
\begin{equation*}
\begin{split}
   f_2:
   &=\eps^{-2}\left(dH|_{E(q,X)} E_2(q,X)-dH|_q\right) \hat X ,
   \\
%   &=\eps^{-2}\left(dH|_{E(q,X)} E_2(q,X)-dH|_q\right) \nor\,\hat X ,
%\\
   \norm{f_2}
   &\le\eps^{-2} c_5\norm{X}_\infty
   \norm{\hat X} .
\end{split}
\end{equation*}
This concludes the proof of Proposition~\ref{prop:quadest-II}
(Quadratic Estimate II).
\end{proof}

%\newpage %.\newpage
\boldmath
%%%%%%%%%%%%%%%%%%%%%%%%%%%%%%%%%%%
%%%%%%%%%%%%%%%%%%%%%%%%%%%%%%%%%%%
\subsection{Existence -- definition of $\Tt^\eps$}
\label{sec:existence}
\unboldmath

We prove Theorem~\ref{thm:existence-findim}.
Assume the Morse-Smale condition holds true.
Up to time-shift there are only finitely many elements
$q$ of $\Mm^0_{x^-,x^+}$, that is base solutions~$q$
between critical points of $f$ of Morse index difference $1$.
The constant
\[
   c_0:=\max
   \left\{\norm{\p_s q}_\infty\mid q\in\Mm^0_{x^-,x^+}\right\}
   +\norm{\chi}_{L^\infty(\Sigma)}
   <\infty
\]
is finite since the function $\chi$ is bounded along the compact
$\Sigma$ and since $\norm{\p_s q}_\infty$ is finite due to
exponential decay and since, by index difference one, there are only
finitely many $q$'s up to time shift.
Fix $\eps_0>0$ sufficiently small such that the key estimate,
Theorem~\ref{thm:KeyEst-thm.3.3}, applies to all $q\in\Mm^0_{x^-,x^+}$
and $\eps\in(0,\eps_0]$.

Pick $q\in\Mm^0_{x^-,x^+}$. Recall that $\chi$ is defined by~(\ref{eq:chi}).
The trivialized section
along the canonical embedding $i(q)=\left(q,\chi(q)\right)$, namely
$\Ff^\eps_q(X,\ell)$ defined by~(\ref{eq:triv-section-eps}),
acts on the elements $Z=(X,\ell)$
of the Banach space $W^{1,2}(\R,q^*TM\oplus\R)$.
At the origin the first component vanishes
\begin{equation}\label{eq:triv-eps-q}
   \Ff^\eps_q\begin{pmatrix}0\\0\end{pmatrix}
   =
   \begin{pmatrix}
      \p_sq+\Babla{} F(q)+\chi(q)\Babla{} H(q)\\
      (\chi(q))^\prime+\eps^{-2}H(q)
   \end{pmatrix}
   =\begin{pmatrix}
      0\\
      d\chi|_q\p_s q
   \end{pmatrix} 
\end{equation}
since $H(q)\equiv 0$. Therefore for the initial point
\[
   Z_0:=(0,0)
\]
we have
\[
   \Norm{\Ff^\eps_q(Z_0)}_{0,2,\eps}
   =\Norm{\Ff^\eps(q,\chi(q))}_{0,2,\eps}
   =\Norm{(0,d\chi|_q\p_sq)}_{0,2,\eps}
   \le\eps\mu_\infty \sqrt{c^*}
\]
where $\mu_\infty$ is defined by~(\ref{eq:mu-infty}) and
\[
   \Norm{\p_s q}
   \stackrel{(\ref{eq:energy-0})}{=}
   \sqrt{f(x^-)-f(x^+)}=:\sqrt{c^*} .
\]
Now define the initial correction term
$\zeta_0=(\hat X_0,\hat\ell_0)$ by
\[
   \zeta_0
   :=-{D^\eps_q}^*\left(D^\eps_q{D^\eps_q}^*\right)^{-1}
   \Ff^\eps_q(0)
\]
where $D^\eps_q=d \Ff^\eps_q(0,0)$.
Recursively, for $\nu\in\N$, define the sequence 
$\zeta_\nu=(\hat X_\nu,\hat \ell_\nu)$ of correction terms by
\begin{equation}\label{eq:zeta-nu}
\begin{aligned}
   \zeta_\nu=(\hat X_\nu,\hat\ell_\nu)
   &:=-{D^\eps_q}^*\left(D^\eps_q{D^\eps_q}^*\right)^{-1}
   \Ff^\eps_q(Z_\nu) ,
   \\
   Z_\nu=(X_\nu,\ell_\nu)
   &:=\sum_{k=0}^{\nu-1}\zeta_k
   =Z_{\nu-1}+\zeta_{\nu-1} .
\end{aligned}
\end{equation}
We prove by induction that there is a constant $c>0$ such that
\begin{equation}\label{eq:H-nu}\tag{$H_\nu$}
\begin{aligned}
   \eps^{1/2}\norm{\zeta_\nu}_{0,\infty,\eps}+\norm{\zeta_\nu}_{1,2,\eps}
   &\le \frac{c}{2^\nu}\eps^2
   \\
   \norm{\Ff^\eps_q(Z_{\nu+1})}_{0,2,\eps}
   &\le \frac{c}{2^\nu}\eps^{5/2}
\end{aligned}
\end{equation}
for every $\nu\in\N_0$. The $(1,2,\eps)$ and $(0,\infty,\eps)$ norms
were defined in~(\ref{eq:findim-0-2-eps}).

\medskip
\noindent
\textbf{Initial step: \boldmath$\nu=0$.}
By definition of $\zeta_0$ we have
\begin{equation}\label{eq:nu=0}
   D^\eps_q\zeta_0
   =-\Ff^\eps_q(0)
   =
   \begin{pmatrix}
      0\\
      -d\chi|_q\p_s q\\
   \end{pmatrix} .
\end{equation}
Thus, by the key estimate, Theorem~\ref{thm:KeyEst-thm.3.3},
(with constant $c_1>0$) we get
\begin{equation}\label{eq:zeta_0}
\begin{split}
   \norm{\zeta_0}_{1,2,\eps}
   &\stackrel{(\ref{eq:thm:4.4.4})}{\le}   
   c_1\left(
   \eps\norm{(0, d\chi|_q\p_s q)}_{0,2,\eps}
   +\norm{\pi_\eps (0, d\chi|_q\p_s q)}
   \right)
   \\
   &\stackrel{(\ref{eq:pi_eps-ff})}{\le}   
   c_1\left(\eps^2\mu_\infty \norm{\p_s q}
   +\norm{(\1+\eps^2\mu^2P)^{-1}\eps^2(d\chi|_q \p_sq)\Nabla{}\chi}
   \right)
   \\
   &\stackrel{(\ref{no-eq:4.1.5-findim})}{\le}   
   2c_1\mu_\infty^2 \sqrt{c^*} \eps^2
\\
   \norm{\zeta_0}_{0,\infty,\eps}
   &\stackrel{(\ref{eq:cor:infty})}{\le}
   3 \eps^{-1/2}\norm{\zeta_0}_{1,2,\eps}
   \\
   &\stackrel{{\color{white} (4.47)}}{\le}
   6c_1\mu_\infty^2\sqrt{c^*}\eps^{3/2} \le \delta .
\end{split}
\end{equation}
To get the bound $\delta$ (needed by the quadratic estimates
Proposition~\ref{prop:quadest-I} and~\ref{prop:quadest-II})
choose $\eps_0>0$ smaller if necessary.
This proves estimate one in $(H_\nu)$ for $\nu=0$
and with a suitable constant $c>0$ depending only on $c_1$
and the $L^\infty$-norms of $\Nabla{}\chi\colon\Sigma\to T\Sigma$ and $\p_s q$.
To prove estimate two we observe that $Z_1=\zeta_0$ and hence, by
Proposition~\ref{prop:quadest-I} (with constant $c_2>0$), we get
\begin{equation}\label{eq:F(Z_1)}
\begin{split}
   \norm{\Ff^\eps_q(Z_1)}_{0,2,\eps}
   &\stackrel{(\ref{eq:nu=0})}{=}
   \norm{\Ff^\eps_q(\zeta_0)
   \overbrace{-\Ff^\eps_q(0)-D^\eps_u\zeta_0}^{=0}}_{0,2,\eps}
   \\
   &\stackrel{(\ref{eq:quadest-I})}{\le}
   \frac{c_2}{\eps}\left(\norm{\hat X_0}_\infty
   \bigl(
   \norm{\hat X_0}
   +\eps\norm{\hat\ell_0}
   +\eps\norm{\Babla{s}\hat X_0}\cdot
      \underline{\norm{\hat X_0}_\infty}
   \bigr)
%   +\eps\norm{\hat\ell_0}_\infty \norm{\hat X_0}
   \right)
   \\
   &\stackrel{(\ref{eq:quadest-I})}{\le}
   \frac{2c_2}{\eps}\norm{\zeta_0}_{0,\infty,\eps}\,
   \norm{\zeta_0}_{1,2,\eps}
   \\
   &\stackrel{(\ref{eq:zeta_0})}{\le}
   48 c_1^2c_2\mu_\infty^4 c^*\,\eps^{5/2} .
\end{split}
\end{equation}
In step 3 we discarded the underlined term
$\norm{\hat X_0}_\infty\le 1$. Then, up to a factor
$2$, see~(\ref{eq:sum-square}), the $(1,2,\eps)$
norm~(\ref{eq:findim-0-2-eps}) appears.
This proves $(H_\nu)$ for $\nu=0$.
From now on we fix the constant $c$ for which the estimate $(H_0)$ has
been established.

%\newpage
\medskip
\noindent
\textbf{Induction step: \boldmath$\nu-1\Rightarrow\nu$.}
Let $\nu\ge 1$ and assume that the
hypotheses~$(H_0),\dots,(H_{\nu-1})$
are true. Then we obtain that
\begin{equation}\label{eq:Z-nu}
\begin{split}
   \eps^{1/2}\norm{Z_\nu}_{0,\infty,\eps}+\norm{Z_\nu}_{1,2,\eps}
   &\stackrel{{\color{white}(H_{0\dots\nu-1})}}{\le}
   \sum_{k=0}^{\nu-1}\left(\eps^{1/2}\norm{\zeta_k}_{0,\infty,\eps}
   +\norm{\zeta_k}_{1,2,\eps}\right)
   \\
   &\stackrel{(H_{0\dots\nu-1})}{\le} c\eps^2\sum_{k=0}^{\nu-1} 2^{-k}
   \le 2c\eps^2 \le \delta
   \\
\end{split}
\end{equation}
(for the bound $\delta$ choose $\eps_0>0$ smaller if necessary)
and we also obtain that
\begin{equation}\label{eq:F(Z_nu)}
\begin{split}
   \norm{\Ff^\eps_q(Z_\nu)}_{0,2,\eps}
   &\stackrel{(H_{\nu-1})}{\le}\frac{c}{2^{\nu-1}}\eps^{5/2} .
\end{split}
\end{equation}
By~(\ref{eq:zeta-nu}), using the property of a right inverse, we have
\[
   D^\eps_q\zeta_\nu
   =-\Ff^\eps_q(Z_\nu),\quad
   \zeta_\nu\in\im(D^\eps_q)^* .
\]
Hence, together with the key estimate~(\ref{eq:thm:4.4.4}), (with constant
$c_1>0$), we get
\begin{equation}\label{eq:zeta-nu-12}
\begin{split}
   \eps^{1/2}\norm{\zeta_\nu}_{0,\infty,\eps}
   +\norm{\zeta_\nu}_{1,2,\eps}
   &\stackrel{(\ref{eq:thm:4.4.4})}{\le}
   c_1\norm{\Ff^\eps_q(Z_\nu)}_{0,2,\eps}
   \\
   &\stackrel{(\ref{eq:F(Z_nu)})}{\le}
   c_1\eps^{1/2} \frac{c}{2^{\nu-1}}\eps^2
   \le\frac{c}{2^{\nu}}\eps^2\le\delta .
\end{split}
\end{equation}
The last but one inequality holds if $9c_1\sqrt{\eps_0}\le\frac12$.
The last inequality holds by the last inequality in~(\ref{eq:Z-nu}).
This proves the first estimate in $(H_\nu)$.

In what follows in step~1 add twice zero and in step~2 apply the
quadratic estimates, Proposition~\ref{prop:quadest-I}
and~\ref{prop:quadest-II} (with constant $c_2>0$), in order to obtain
\begin{equation*}
\begin{split}
   &\norm{\Ff^\eps_q(Z_{\nu+1})}_{0,2,\eps}
   \\
   &\stackrel{{\color{white}(7.110)}}{\le}
   \norm{\Ff^\eps_q(Z_\nu+\zeta_\nu)-\Ff^\eps_q(Z_\nu)
   -d\Ff^\eps_q(Z_\nu)\zeta_\nu}_{0,2,\eps}
   +\norm{d\Ff^\eps_q(Z_\nu)\zeta_\nu-D^\eps_q\zeta_\nu}_{0,2,\eps}
   \\
%%%%
   &\stackrel{{\color{white}(7.110)}}{\le}
   \frac{c_2}{\eps}\norm{\hat X_\nu}_\infty
      \left(\norm{\hat X_\nu}
      +\eps\norm{\hat\ell_\nu}
      +\eps\norm{\Babla{s}\hat X_\nu}\right)
%
%   +c_2\norm{\hat\ell_\nu}_\infty \norm{\hat X_\nu}
%
   +\underline{c_2\norm{\Babla{s} X_\nu}\cdot\norm{\hat X_\nu}_\infty}
   \\
   &\quad\qquad
   +\frac{c_2}{\eps}\norm{X_\nu}_\infty
   \left(
   \norm{\hat X_\nu}
   +\eps\norm{\hat\ell_\nu}
   +\eps\norm{\Babla{s}\hat X_\nu}
   \right)
   +c_2\norm{\ell_\nu}_\infty\norm{\hat X_\nu}
   \\
%%%%
   &\stackrel{{\color{white}(7.110)}}{\le}
   \frac{c_2}{\eps}\left(\norm{\zeta_\nu}_{0,\infty,\eps}
   +\norm{Z_\nu}_{0,\infty,\eps}\right)
   \norm{\zeta_\nu}_{1,2,\eps}
   +\underline{c_2\eps^{-1}
      \norm{Z_\nu}_{1,2,\eps}\norm{\zeta_\nu}_{0,\infty,\eps}}
   \\
   &\stackrel{(\ref {eq:zeta-nu-12})}{\le}
   \underbrace{c_2\eps^{-1}
   \left(c\eps^{3/2}+2c\eps^{3/2}\right)
   c_1}_{\le 1/4}
   \frac{c}{2^{\nu-1}}\eps^{5/2}
   +\underbrace{c_22c\eps^{1/2}c_1}_{\le 1/4}\frac{c}{2^{\nu-1}}\eps^{5/2}
   \\
   &\stackrel{{\color{white}(7.110)}}{\le}
   \frac{c}{2^\nu}\eps^{5/2}.
\end{split}
\end{equation*}
In inequality two
we already estimated some factors $\norm{\hat X}_\infty\le 1$
and $\norm{X}_\infty\le 1$ in triple products.
The last inequality holds by choosing $\eps_0>0$ sufficiently small.
This completes the induction and proves $(H_\nu)$ for every
$\nu\in\N_0$.

\medskip
\noindent
\textbf{Conclusion.}
It follows from $(H_\nu)$ that $Z_\nu$ is a Cauchy sequence
with respect to $\norm{\cdot}_{1,2,\eps}$. We denote its limit by
\[
   Z^\eps:=\lim_{\nu\to\infty} Z_\nu
   =\sum_{\nu=0}^\infty \zeta_\nu
   \in W^{1,2}(\R,q^*TM\oplus\R).
\]
By construction, and since the image of $(D^\eps_q)^*$ is closed,
the limit satisfies
\[
   \eps^{1/2}\norm{Z^\eps}_{1,\infty,\eps}+\norm{Z^\eps}_{1,2,\eps}
   \stackrel{(\ref{eq:Z-nu})}{\le} 2c\eps^2
   ,\qquad
   \Ff^\eps_q(Z^\eps)=0
   ,\qquad
   Z^\eps\in\im(D^\eps_q)^* .
\]
This concludes the proof of Theorem~\ref{thm:existence-findim}.

%\newpage
\boldmath
%%%%%%%%%%%%%%%%%%%%%%%%%%%%%%%%%%%
%%%%%%%%%%%%%%%%%%%%%%%%%%%%%%%%%%%
\subsection{Uniqueness -- injectivity of $\Tt^\eps$}
\unboldmath

We prove Theorem~\ref{thm:uniqueness-findim} under the conventions
and notations of Section~\ref{sec:existence}, in particular
Section~\ref{sec:existence} provides $\eps_0\in(0,1]$, whereas
$\delta\in(0,1]$ is the constant that appears in the quadratic estimates.
Shrink $\delta_0>0$ such that $\delta_0\sqrt{\eps_0}\le\delta/4$.
Pick $q\in\Mm^0_{x^-,x^+}$ and $\eps\in(0,\eps_0]$.
Let the base point $Z=(X,\ell):=\Tt^\eps(q)$ be
the zero of the trivialized section $\Ff^\eps_q$ provided by the
existence Theorem~\ref{thm:existence-findim}. Then
\[
   Z\in\im(D^\eps_q)^*,\qquad
   \Ff^\eps_q(Z)=0,\qquad
   \eps^{1/2}\norm{Z}_{0,\infty,\eps}+\norm{Z}_{1,2,\eps}
   \le c\eps^2 \le
{\color{gray}\,
   \delta/4 .
}
\]
for a suitable constant $c>0$ and where the norms are defined
by~(\ref{eq:findim-0-2-eps}) and the $\delta$ estimate holds by
choosing $\eps_0>0$ smaller, if necessary. Shrink $\eps_0>0$
further such that $c\eps_0<\delta_0$.
Now assume $\zeta=(\hat X,\hat\ell)$ satisfies the hypotheses
of the present Theorem~\ref{thm:uniqueness-findim}, that is
\begin{equation*}%\label{eq:pf-uniq-hyp}
   \zeta=(\hat X,\hat\ell)\in\im(D^\eps_q)^*,\qquad
   \Ff^\eps_q(\zeta)=0,\qquad
   \norm{\hat X}_\infty
   \le\delta_0\eps^{1/2} .
%   ,\qquad \norm{\hat\ell}_\infty\le C.
\end{equation*}
The difference
\[
{\color{gray}
   (X^*,\ell^*) =
}\;
   \zeta^*
   :=\zeta-Z
\;{\color{gray}
   =(\hat X-X,\hat\ell-\ell)
   \in\im(D^\eps_q)^*
}
\]
then satisfies the inequalities\footnote{
  a numerical bound $\norm{\ell^*}_\infty< C$ is irrelevant in the
  proof, only finiteness ($<\infty$) matters
  }
\[
   \norm{X^*}_\infty
   \le\left(\delta_0+c\eps\right)\eps^{1/2}
   \le2\delta_0\eps^{1/2}
{\color{gray}\,
   \le \delta/2 ,\qquad
}
   \norm{\ell^*}_\infty <\infty.
%   \le C+c \sqrt{\eps} <\infty .
\]
%where the final inequality holds by shrinking $\eps_0$
%such that $c\sqrt{\eps_0}\le C$.
With the difference abbreviations~(\ref{eq:quadest-I-diff})
and~(\ref{eq:quadest-II-diff}) and since both
$\zeta=Z+\zeta^*$ and $Z$ are zeroes of $\Ff^\eps_q$ we get
the first identity in the following
\begin{equation*}
\begin{split}
   &\norm{{\color{brown}D^\eps_q\zeta^*}}_{0,2,\eps}\\
   &=\bigl\|\bigl(
   \underbrace{\Ff^\eps_q(Z+\zeta^*)-\Ff^\eps_q(Z)-d\Ff^\eps_q(Z)\zeta^*}
      _{=:(F,f)}\bigr)
   +\bigl(
   \underbrace{d\Ff^\eps_q(Z)\zeta^*
   -{\color{brown}d\Ff^\eps_q(0)\zeta^*}}_{=:(\FF,\ff)}\bigr)\bigr\|_{0,2,\eps}\\
   &=\norm{{\color{brown}(F+\FF,f+\ff)}}_{0,2,\eps}\\
   &\le\norm{F}+\norm{\FF}+\eps\norm{f}+\eps\norm{\ff}.
\end{split}
\end{equation*}
By definition~(\ref{no-eq:pi_eps-ff-Ansatz}) of
$\pi_\eps$ with ${\color{red} \beta=2}$ and $\alpha\in[1,2]$
and by Lemma~\ref{no-le:4.1.5-findim} we obtain
\begin{equation}\label{eq:beta=2}
\begin{split}
   \norm{\pi_\eps {\color{brown}D^\eps_q\zeta^*}}
   &=\norm{\pi_\eps{\color{brown}(F+\FF,f+\ff)}}\\
   &=\norm{(\1+\eps^\alpha\mu^2 P)^{-1}
      (\tan(F+\FF)+\eps^{\color{red} 2}(f+\ff)\Nabla{}\chi)}\\
   &\le\norm{F}+\norm{\FF}
   +\mu_\infty\eps^{\color{red} 2}\norm{f}
   +\mu_\infty\eps^{\color{red} 2}\norm{\ff}
\end{split}
\end{equation}
where we also used that $\norm{\tan}\le 1$
since the projection $\tan$ is orthogonal.
The choice ${\color{red} \beta=2}$ neutralizes
the toxic factor $\eps^{-2}$ that comes with the $f$ and $\ff$ terms.

Thus, by estimate four in the key estimate~(\ref{eq:thm:4.4.4}), with
a constant $c_1>0$, by the quadratic estimates~(\ref{eq:quadest-I})
and~(\ref{eq:quadest-II}), with a constant $c_2\ge 2$, we obtain
\begin{equation*}
\begin{split}
   &\norm{\ell^*}\cdot {\color{brown}\norm{X^*}_\infty}\\
   &\le c_1\norm{D^\eps_q\zeta^*}_{0,2,\eps}{\color{brown}\norm{X^*}_\infty}
   \\
   &\le c_1\left(
   \norm{F}+\norm{\FF}+{\color{red}\eps}\norm{f}+{\color{red}\eps}\norm{\ff}
   \right) {\color{brown}\norm{X^*}_\infty}
   \\
   &\le c_1c_2 {\color{brown}\norm{X^*}_\infty}
\Bigl(
   {\color{red}\tfrac{1}{\eps}} \underline{\norm{X^*}_\infty\norm{X^*}}
   + {\color{cyan}\norm{\ell^*}\cdot\norm{X^*}_\infty}
   +\norm{\Babla{s} X^*}\cdot \norm{X^*}_\infty^2
   \\
   &\qquad+\norm{X}_\infty
      \left({\color{red}\tfrac{1}{\eps}}\norm{X^*}
   +\norm{\ell^*}+\norm{\Babla{s} X^*}\right)
   +\norm{\ell}_\infty\norm{X^*}
   +\underline{\underline{\norm{X^*}_\infty}} \norm{\Babla{s} X}
\Bigr)
   \\
   &\le c_1c_2
\Bigl(
   \underline{4\delta_0^2}+8\delta_0^3\sqrt{\eps}+2c\delta_0\eps
   +2c\delta_0\eps +2c\delta_0\eps +2c\delta_0\eps
\Bigr)
   \norm{\zeta^*}_{1,2,\eps}\\
   &\quad+c_1c_22\delta_0\sqrt{\eps}
   \Bigl(\underline{\underline{\tfrac{1}{\sqrt{\eps}}\norm{X^*}
      +\sqrt{\eps}\norm{\Babla{s} X^*}}}\Bigr) c\eps
   +c_1c_22\delta_0\sqrt{\eps}{\,\color{cyan}\norm{\ell^*}\cdot\norm{X^*}_\infty}
   \\
%%%%
   &\le \tfrac{1}{8\cdot2\mu_\infty c_1c_2}\norm{\zeta^*}_{1,2,\eps}
   +\tfrac12 \norm{\ell^*}\cdot\norm{X^*}_\infty .
\end{split}
\end{equation*}
In inequality three we already discarded in a few triple products
some factors $\norm{X^*}_\infty\le 1$ or $\norm{X}_\infty\le 1$.
The once underlined term enforces the smallness
assumption in Theorem~\ref{thm:uniqueness-findim}.
The doubly underlined estimate in inequality three
is by~(\ref{eq:Linfty-est-beta}) with $\beta=1/2$.
The final inequality holds by choosing $\delta_0$ and $\eps_0$ sufficiently
small. We summarize the estimate, which comes in handy below, by
\[
   2\mu_\infty c_1c_2 
{\,\color{cyan}
   \norm{\ell^*}\cdot\norm{X^*}_\infty
}
   \le
{\,\color{cyan}
   \tfrac{1}{4} \norm{\zeta^*}_{1,2,\eps}
}
   .
\]
Similarly, by estimate one in the key estimate~(\ref{eq:thm:4.4.4}), with
a constant $c_1>0$, by the quadratic estimates~(\ref{eq:quadest-I})
and~(\ref{eq:quadest-II}), with a constant $c_2\ge 2$, and with the
constant $\mu_\infty$ defined by~(\ref{eq:mu-infty}), we obtain
\begin{equation*}
\begin{split}
   \norm{\zeta^*}_{1,2,\eps}
   &\le c_1\left(\eps \norm{D^\eps_q\zeta^*}_{0,2,\eps}
   +\norm{\pi_\eps (D^\eps_q\zeta^*)}\right)
   \\
   &\le 2\mu_\infty c_1\left(
   \norm{F}+\norm{\FF}+\eps^2\norm{f}+\eps^2\norm{\ff}
   \right)
   \\
   &\le 2\mu_\infty c_1c_2
\Bigl(
   \norm{X^*}\cdot\norm{X^*}_\infty
   +{\color{cyan}\norm{\ell^*}\cdot\norm{X^*}_\infty}
   +\norm{\Babla{s} X^*}\cdot\underline{\norm{X^*}_\infty^2}
   \\
   &\quad+\norm{X}_\infty
      \left(\norm{X^*}+\norm{\ell^*}+\norm{\Babla{s} X^*}\right)
   +\norm{\ell}_\infty\norm{X^*}
   +\underline{\underline{\norm{X^*}_\infty}} \norm{\Babla{s} X}
\Bigr)
   \\
   &\le 2\mu_\infty c_1c_2
\Bigl(
   \delta_0\sqrt{\eps}
   +\underline{\delta_0^2}
   +c\eps^{3/2}
   +c\eps^{1/2}
   +c\eps^{3/2}\eps^{-1}+c\sqrt{\eps}
\Bigr)
   \norm{\zeta^*}_{1,2,\eps}\\
   &\quad+ 2\mu_\infty c_1c_2
   \left(\underline{\underline{\eps^{-1/2}\norm{X^*}
      +\eps^{1/2}\norm{\Babla{s} X^*}}}\right) c\eps
   +{\color{cyan}  \tfrac{1}{4} \norm{\zeta^*}_{1,2,\eps}}\\
%%%%
   &\le \frac12
   \norm{\zeta^*}_{1,2,\eps} .
\end{split}
\end{equation*}
In inequality three we discarded in a few triple products
some factors $\norm{X^*}_\infty\le 1$ or $\norm{X}_\infty\le 1$.
The once underlined term enforces the smallness
assumption in Theorem~\ref{thm:uniqueness-findim}.
The doubly underlined estimate in inequality three
is by~(\ref{eq:Linfty-est-beta}) with $\beta=1/2$.
The final inequality holds by choosing $\delta_0$ and $\eps_0$ sufficiently
small. Thus the element $\zeta^*=\zeta-Z$ is zero in $W^{1,2}$.
This proves Theorem~\ref{thm:uniqueness-findim}.

%%%%%%%%%%%%%%%%%%%%%%%%%
%%%%%%%%% REFERENCES %%%%%%
%%%%%%%%%%%%%%%%%%%%%%%%
%\renewcommand{\bibname}{References}
%\bibliographystyle{plain}
         %   erzeugt:     [1] Joa Weber
%\bibliographystyle{abbrv}
         %  erzeugt:      [1] J. Weber and 
\bibliographystyle{alpha}
         %  article:    [Web05]  J. Weber
         %  book:      [Web05]  Joa Weber
         % more authors: [HZ87]
%%%%%%%%%%%%%%%%%%%%%%%%%
%% include Bibliography in TOC %%
% en.wikibooks.org/wiki/LaTeX/Bibliography_Management#Using_tocbibind
%%%%%%%%%%%%%%%%%%%%%%%%%
% Using hyperref, one should say:
%\cleardoublepage
%\phantomsection
\addcontentsline{toc}{section}{References}
\bibliography{$HOME/Dropbox/0-Libraries+app-data/Bibdesk-BibFiles/library_math,$HOME/Dropbox/0-Libraries+app-data/Bibdesk-BibFiles/library_math_2020,$HOME/Dropbox/0-Libraries+app-data/Bibdesk-BibFiles/library_physics}{}
%$
%%%%%%%%%%%%%%%%%%%%%%%%%
%%%%%%%%% standard %%%%%%%%%
%%%%%%%%%%%%%%%%%%%%%%%%%
%\begin{thebibliography}{00000}
%\small
%\end{thebibliography}

%%%%%%%%%%%%%%%%%%%%%%%%%%%%%%%%%%%%
%%%%%%%%%%%%% GLOSSARY %%%%%%%%%%%%%%%
%%%%%%%%%%%%%%%%%%%%%%%%%%%%%%%%%%%%
% Using hyperref, one should say:
%\cleardoublepage
%\phantomsection
%\printnomenclature
%
%This is $F$\label{nomen:F} 
%\nomenclature[EF]{$F$}{Objective function}{}{\pageref{nomen:F}}
%\clearpage

%%%%%%%%%%%%%%%%%%%%%%%%%
%%%%%%%%% INDEX %%%%%%%%%%
%%%%%%%%%%%%%%%%%%%%%%%%%
% Using hyperref, one should say:
%\cleardoublepage
%\phantomsection
%\addcontentsline{toc}{chapter}{Index}
%\printindex

\end{document}